\documentclass[11pt,reqno,a4paper]{amsart}
\usepackage[OT1]{fontenc}

\usepackage[american]{babel}

\usepackage[utf8]{inputenc}

\usepackage[bookmarks=true,%
    colorlinks=true,%
    linkcolor=pink,%
    citecolor=pink,%
    filecolor=blue,%
    menucolor=blue,%
    urlcolor=Lavender,pageanchor]{hyperref}
\usepackage[hmarginratio=1:1]{geometry}
\usepackage{caption}
\usepackage{subcaption}
\usepackage{import}
\usepackage{amsmath,mathrsfs,bm}
\usepackage{amssymb,amsthm}
\usepackage{mathtools}
\usepackage{graphicx}
\usepackage{color,transparent}
\usepackage[dvipsnames]{xcolor}
\usepackage{float}
\usepackage{tikz-cd}
\usepackage{svg}
\allowdisplaybreaks
\usepackage[toc,page]{appendix}
\usepackage{comment}
\usepackage{stmaryrd}

\usepackage{tkz-euclide}
\usetikzlibrary{shapes.geometric}
\usetikzlibrary{positioning, calc}
\usetikzlibrary{decorations.markings}

\usepackage{pgfplots}
\pgfplotsset{compat=1.11}
\usepgfplotslibrary{fillbetween}
\usetikzlibrary{intersections}

\newcommand{\hgline}[4]{
\pgfmathsetmacro{\thetaone}{#1}
\pgfmathsetmacro{\thetatwo}{#2}
\pgfmathsetmacro{\theta}{(\thetaone+\thetatwo)/2}
\pgfmathsetmacro{\phi}{abs(\thetaone-\thetatwo)/2}
\pgfmathsetmacro{\close}{less(abs(\phi-90),0.0001)}
\ifdim \close pt = 1pt
    \draw[thick, #3,name path = #4] (\theta+180:1) -- (\theta:1);
\else
    \pgfmathsetmacro{\R}{tan(\phi)}
    \pgfmathsetmacro{\distance}{sqrt(1+\R^2)}
    \draw[thick, #3,name path = #4] (\theta:\distance) circle (\R);
\fi
}
\pgfdeclarelayer{bg}
\pgfsetlayers{bg,main}

\usepackage[colorlinks=true]{hyperref}
\hypersetup{
    colorlinks=true,
    linkcolor=Orchid,
    citecolor=VioletRed
    }

\usepackage[multiple]{footmisc}
\let\oldFootnote\footnote
\newcommand\nextToken\relax

\renewcommand\footnote[1]{%
    \oldFootnote{#1}\futurelet\nextToken\isFootnote}

\newcommand\isFootnote{%
    \ifx\footnote\nextToken\textsuperscript{,}\fi}

\let\epsilon\varepsilon

\newcommand{\abs}[1]{\left|#1\right|}
\newcommand{\R}{{\mathbb{R}}}

\newcommand{\N}{\mathbb{N}}

\newcommand{\T}{\mathsf{T}} 

\newcommand{\identity}{\mathsf{id}}
\newcommand{\eps}{\varepsilon}

\newcommand{\st}{\,\colon\,} 

\newcommand{\trace}{\mathsf{tr}}

\newcommand{\PSL}[1]{\mathsf{PSL}(#1,\R)}
\newcommand{\SL}[1]{\mathsf{SL}(#1,\R)}

\newcommand{\Hom}{\mathsf{Hom}}

\newcommand{\flags}{\mathscr F}

\newcommand{\Hm}{\mathsf{H}} 

\newcommand{\Htwo}{\mathbb{H}^2}

\newcommand{\figeight}{\delta}

\newcommand{\Teich}{\mathcal T}

\newcommand{\CharVar}{\mathscr X}

\newcommand{\dev}{\mathsf{dev}}
\newcommand{\hol}{\mathsf{hol}}
\newcommand{\Cr}{\mathsf{cr}}
\newcommand{\Vvector}[3]{\begin{bmatrix}
	#1\\#2\\#3
\end{bmatrix}}
\newcommand{\Lline}[3]{\begin{bmatrix}
	#1:#2:#3
\end{bmatrix}}
\newcommand{\triple}{\mathsf{T}}
\newcommand{\triangulation}{\mathscr T}
\newcommand{\edge}{\mathsf E}
\newcommand{\vertex}{v}
\newcommand{\quiver}{\mathsf{Q}}

\newcommand{\Rad}{\mathsf{Rad}}
\newcommand{\vf}[1]{\Gamma(\T#1)}
\newcommand{\ringCasimir}{\mathscr C}
\newcommand{\SymLeaf}{\mathcal Q}
\newcommand{\HmFlow}{\widehat{\Phi}}
\newcommand{\SympFoliation}{\mathscr L}

\newcommand{\rechol}{\vartheta}
\newcommand{\pants}{P}
\newcommand{\eruption}{\mathcal{I}}
\newcommand{\hexagon}{\mathcal{E}}
\newcommand{\unipotentLocus}{\mathfrak{U}}

\newcommand{\teich}{\mathcal T}
\newcommand{\thweb}{m_\Theta}
\newcommand{\conjclass}{\mathcal C}

\newcommand{\RPTwo}{\mathbb{RP}^2}
\newcommand{\DefSpace}{\mathfrak{C}}
\newcommand{\casimir}{\ell}
\newcommand{\framedRepsConv}{\widehat{\mathscr{X}}^+_3(S)}
\newcommand{\RepsConv}{\mathscr X_3^+(S)}

\newcommand{\framedChar}{\widehat{\mathscr X}_3(S)}
\newcommand{\del}[1]{\frac{\partial}{\partial #1}}
\tikzset{
    set arrow inside/.code={\pgfqkeys{/tikz/arrow inside}{#1}},
    set arrow inside={end/.initial=>, opt/.initial=},
    /pgf/decoration/Mark/.style={
        mark/.expanded=at position #1 with
        {
            \noexpand\arrow[\pgfkeysvalueof{/tikz/arrow inside/opt}]{\pgfkeysvalueof{/tikz/arrow inside/end}}
        }
    },
    arrow inside/.style 2 args={
        set arrow inside={#1},
        postaction={
            decorate,decoration={
                markings,Mark/.list={#2}
            }
        }
    },
}

\newtheorem*{theorem*}{Theorem}
\newtheorem{theorem}{Theorem}[section]
\newtheorem{corollary}[theorem]{Corollary}
\newtheorem{proposition}[theorem]{Proposition}
\newtheorem{lemma}[theorem]{Lemma}
\newtheorem{thmA}{Theorem}

\newtheorem{corA}[thmA]{Corollary}

\newtheorem{propA}[thmA]{Proposition}

\theoremstyle{definition}

\newtheorem{definition}[theorem]{Definition}
\newtheorem{remark}[theorem]{Remark}

\newtheorem{fact}{Fact}
\newtheorem{question}[theorem]{Question}

\renewenvironment{proof}{{\bfseries Proof. }}{\qed\\}

\theoremstyle{remark}

\title[Self--intersecting curves and periodic orbits]{Self--intersecting curves on a pair of pants and periodic orbits of Hamiltonian flows}

\author[F. Camacho--Cadena]{Fernando Camacho--Cadena}
\address{Max Planck Institute for Mathematics in the Sciences, Inselstr. 22, 04103 Leipzig, Germany\vspace{.2cm}}
\email{fernando.camacho@mis.mpg.de}

\thanks{This project received funding from the Deutsche Forschungsgemeinschaft (DFG, German Research Foundation) – Project-ID 281071066 – TRR 191, as well as from the European Research Council (ERC) under the European Union’s Horizon 2020 research and innovation programme (grant agreement No 101018839). The author acknowledges the support of the Institut Henri Poincaré (UAR 839 CNRS-Sorbonne Université), and LabEx CARMIN (ANR-10-LABX-59-01)}

\begin{document}
\maketitle
\begin{abstract}
The character variety $\CharVar(S,G)$ associated to an oriented compact surface $S$ with boundary and a real reductive Lie group $G$ admits a Poisson structure and is foliated by symplectic leaves. When $G$ is a matrix group, any closed curve $c\in\pi_1(S)$ induces a trace function $\trace_c\colon[\rho]\mapsto \trace(\rho(c))$ on $\CharVar(S,G)$. In this article, we study the Hamiltonian flows of trace functions associated to self--intersecting curves. We prove that when $G=\PSL 3$ and $S$ is the pair of pants, every orbit of the Hamiltonian flow of the trace of a figure eight curve on $S$ is periodic and has a unique fixed point. The proof uses explicit computations in Fock--Goncharov coordinates. As an application, we prove a similar statement for the trace of the $\Theta$--web. Finally, we focus on the symplectic leaf corresponding to the unipotent locus, and derive similar results for two more self--intersecting curves: the commutator, and a curve going $k$ times around a boundary component.
\end{abstract}

\section{Introduction}\label{intro : sec : periodic Hamiltonian flows}
Let $S$ be an oriented compact surface of genus $g$ with $n\geq 1$ boundary components, of negative Euler characteristic and with fundamental group $\pi_1(S)$. Let $G$ be a connected real reductive Lie group. Associated to $S$ and $G$ is what we will call the \emph{character variety}
\[
\CharVar(S,G)\coloneqq \Hom^{\mathsf{red}}(\pi_1(S),G)/ G,
\]
which is the quotient of the space of reductive representations by the conjugation action by $G$. The smooth part of the character variety has a natural Poisson structure, see for example \cite{PoissonStructure_Audin,AlekseevMalkin_PoissonStructure}. The Poisson structure is originally due to Atiyah and Bott \cite{AtiyahBott}, where they deal with the case when $G$ is a compact Lie group, and $S$ is a closed surface (in which case the smooth part of the character variety is a symplectic manifold).\\

The smooth part of the character variety is then foliated by symplectic leaves, which are related to relative character varieties, see for example \cite[Theorem 2.2.1]{PoissonStructure_Audin}, \cite{ParabolicCohomology_GHJW,Poisson_BiswasJeffrey}. Given a tuple of conjugacy classes $\conjclass = (C_1,\dots,C_n)$ in $G$, the \emph{relative character variety associated to $\conjclass$} is 
\[
\CharVar_\conjclass(S,G) \coloneqq \left\{\rho\in\Hom^{\mathsf{red}}(\pi_1(S),G) \st \rho(c_i)\in C_i, i = 1,\dots,n\right\}/ G,
\]
where the $c_i\in\pi_1(S)$ are chosen generators for each boundary component of $S$.\\

Perhaps the best known example of this is the Teichm\"uller space of $S$, i.e. the space of hyperbolic structures on $S$ up to isotopy. The holonomy of a hyperbolic structure gives a map from the Teichm\"uller space to the character variety $\CharVar(S,\PSL 2)$. The relative character varieties are given by prescribing the length of boundary components and the symplectic structure on them is given by the Weil--Petersson symplectic form \cite{Wolpert_Symplectic}. \\

 When $G$ is a matrix group, an important class of functions on character varieties are trace functions associated to closed curves on $S$. Namely, any closed curve $c\in\pi_1(S)$ defines a \emph{trace function}
 \begin{align*}
 \trace_c\colon \CharVar(S,G)&\to\R\\
 [\rho]&\mapsto \trace(\rho(c)).
 \end{align*}
 By restricting a trace function $\trace_c$ to a symplectic leaf $\SymLeaf\subset \CharVar(S,G)$, we obtain its associated \emph{Hamiltonian flow}. This is the flow associated to the \emph{Hamiltonian vector field} $\Hm \trace_c$ given by $\omega_\SymLeaf(\cdot,\Hm\trace_c)=d\,\trace_c(\cdot)$, where $\omega_\SymLeaf$ is the symplectic form of $\SymLeaf$. In the case of Teichm\"uller space, one can consider the function $\abs{\trace_c}\colon[\rho]\mapsto \abs{\trace(\widetilde\rho(c))}$, where $\widetilde\rho$ is any lift of $\rho$ to $\SL 2$. This function is directly related to the length function $\ell_c$ (see for example \cite[Equation (2.10)]{InvFct_Goldman}), which measures the length of the unique geodesic in the free homotopy class of $c$ in a given hyperbolic structure. When $c$ is a simple closed curve, Wolpert \cite{Wolpert_TheFNDef} found that the Hamiltonian flows of $\ell_c$ correspond to twist flows.\\
 
In more generality, when $G$ is any real reductive matrix Lie group and $c$ is a \emph{simple} closed curve, Goldman fully described the Hamiltonian flow of the function $\trace_c$ \cite{InvFct_Goldman} (see also \cite{ErgodicityMCG_GoldmanXia} for surfaces with boundary). However, both in Teichm\"uller space and on more general character varieties, the study of Hamiltonian flows of trace functions associated to self--intersecting curves is much more limited. In the case when $S$ is a closed surface, Farre and Wienhard together with the author introduced invariant multi--functions in \cite{InvMultiFunctions_CCFW}, generalizing \cite{InvFct_Goldman}. In that article, the authors found a geometric and qualitative description of Hamiltonian flows associated to trace functions coming from self--intersecting curves. Namely, if the curve fills a subsurface $S_0$ of $S$, then the Hamiltonian flow only deforms $S_0$ and leaves the complement unchanged. Even though the result gave insight into the Hamiltonian flow, its behavior on the surface $S_0$ remained unknown outside some simple cases.\\

 In this article, we continue the study of Hamiltonian flows associated to trace functions of self--intersecting curves. To do this, we turn to the framework of Fock and Goncharov \cite{FockGoncharov}. A \emph{framed representation} is a representation $\rho\colon\pi_1(S)\to\SL d$ together with the data of $n$ full flags $F_1,\dots, F_n$ in $\R^d$ associated to each boundary component such that $\rho(c_i)\cdot F_i = F_i$ for each $i = 1,\dots,n$. Fock and Goncharov define a set of coordinates on the space of framed representations depending on a triangulation of $S$ (see Section \ref{sec : Big FG coords}). The positivity of the coordinates does not depend on the triangulation and therefore defines the set of \emph{positive framed representations}, which we denote by $\widehat\CharVar_d^+(S)$. Moreover, the coordinates give $\widehat\CharVar_d^+(S)$ the structure of a smooth manifold.\\
 
 The space of positive framed representations comes equipped with a Poisson structure \cite{PoissonStructure_FockRosly,FockGoncharov}. Its symplectic leaves are related to relative character varieties through the projection $\widehat\CharVar_d^+(S)\to \CharVar_d(S)$ which forgets the framing, and where we abbreviate $\CharVar(S,\SL d)$ by $\CharVar_d(S)$. The symplectic structure on the symplectic leaves coincides with the symplectic structure on the relative character varieties up to a constant \cite[Theorem 1.1]{SwapAlg_Sun}. Fock and Goncharov describe the Poisson structure explicitly in coordinates (see Section \ref{sec : general Poisson structure}), and the reconstruction of the representation from the coordinates is also explicit. This allows us to work with explicit computations to deduce results on Hamiltonian flows of trace functions associated to self--intersecting curves.\\
 
 Here we take the case $G = \PSL 3$, which is in particular a matrix group since $G \cong \SL 3$. The space of positive framed representations coincides with the space of \emph{framed real convex projective structures} on $S$ \cite{FockGoncharov_ConvexProjective}. These are triples $[\dev,\rho,\nu]$, where $\rho\colon\pi_1(S)\to\PSL 3$ is a representation, known as the \emph{holonomy}, $\dev\colon\widetilde S\to\RPTwo$ is a $\rho$--equivariant diffeomorphism onto a properly convex domain in $\RPTwo$,  known as the \emph{developing map}, and $\nu$ is a framing of $\rho$ (see Section \ref{sec : conv proj surfaces Prelim}). We denote the space of framed real convex projective structures on $S$ by $\widehat\DefSpace(S)$.  \\

In this article, we focus on the pair of pants $\pants$. The symplectic leaves in $\widehat\DefSpace(\pants)$ are two dimensional \cite{ConvexRealProj_Goldman}, and hence using Fock--Goncharov coordinates, we can plot level sets of trace functions to understand and guess the behavior of their associated Hamiltonian flows. We pick a presentation of the fundamental group $\pi_1(P) = \langle\alpha,\beta,\gamma\,|\,\alpha\beta\gamma = 1\rangle$ with $\alpha,\beta,\gamma$ corresponding to the peripheral loops as in Figure \ref{fig : intro pair of pants}.\\

Let $\figeight = \alpha\gamma^{-1}$ be a \emph{figure eight curve} in $P$, as shown in Figure \ref{fig : intro pair of pants}. Picking a symplectic leaf (see Section \ref{sec : Casimir functions and symplectic leaves} for more details), Figure \ref{fig : Intro level sets} shows the level sets of $\trace_\figeight$ using coordinates $\sigma_1$ and $\tau_1$ (see Section \ref{sec : alpha^kgamma^-1} for a more detailed description of which symplectic leaf the plot corresponds to).\\



Our main result is the following, which confirms the heuristic picture in Figure \ref{fig : Intro level sets}.

\begin{thmA}[Theorem \ref{thm : periodicity of figure 8 in general}]\label{thm Intro : main theorem}
	Let $\figeight = \alpha\gamma^{-1}$ be a figure eight curve on a pair of pants $\pants$ and let $\SymLeaf$ be any symplectic leaf in $\widehat\DefSpace(\pants)$. Then the restriction of the trace function $\trace_\figeight\big|_{\SymLeaf}\colon\SymLeaf\to\R$ attains a unique minimum. Moreover, every orbit of the Hamiltonian flow of $\trace_\figeight\big|_{\SymLeaf}$ is periodic and there is a unique fixed point.
\end{thmA} 
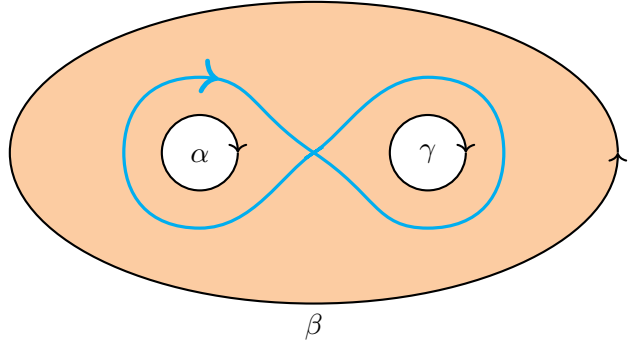
\begin{figure}[h]
	\centering
	\begin{tikzpicture}
    \fill[Apricot!70] (0,0) ellipse (4cm and 2cm);
    
    \fill[white] (1.5,0) circle (0.5cm);
    \fill[white] (-1.5,0) circle (0.5cm);
    
    \draw[->,thick] (2,0) arc[start angle = 360, end angle = 0, radius = 0.5];
    \draw[->,thick] (-1,0) arc[start angle = 360, end angle = 0, radius = 0.5];
    \draw[->,thick] (4,0) arc [start angle = 0,end angle = 360, x radius = 4cm, y radius = 2cm];
    
    \node[label = $\gamma$] at (1.5,-0.4) {};
    \node[label = $\alpha$] at (-1.5,-0.4) {};
    \node[label = $\beta$] at (0,-2.75) {};
    \draw[very thick,cyan,] plot [smooth cycle,tension = 1] coordinates {(0,0) (-1.5,-1) (-2.5,0) (-1.5,1) (0,0) (1.5,-1) (2.5,0) (1.5,1) (0,0)}  [arrow inside={opt={scale=1.5}}{0.25}];
	\end{tikzpicture}
\caption{A pair of pants, the generators for its fundamental group, and a figure eight curve.}
\label{fig : intro pair of pants}
\end{figure}

\begin{figure}[h]
	\centering
	\includegraphics[scale=0.5]{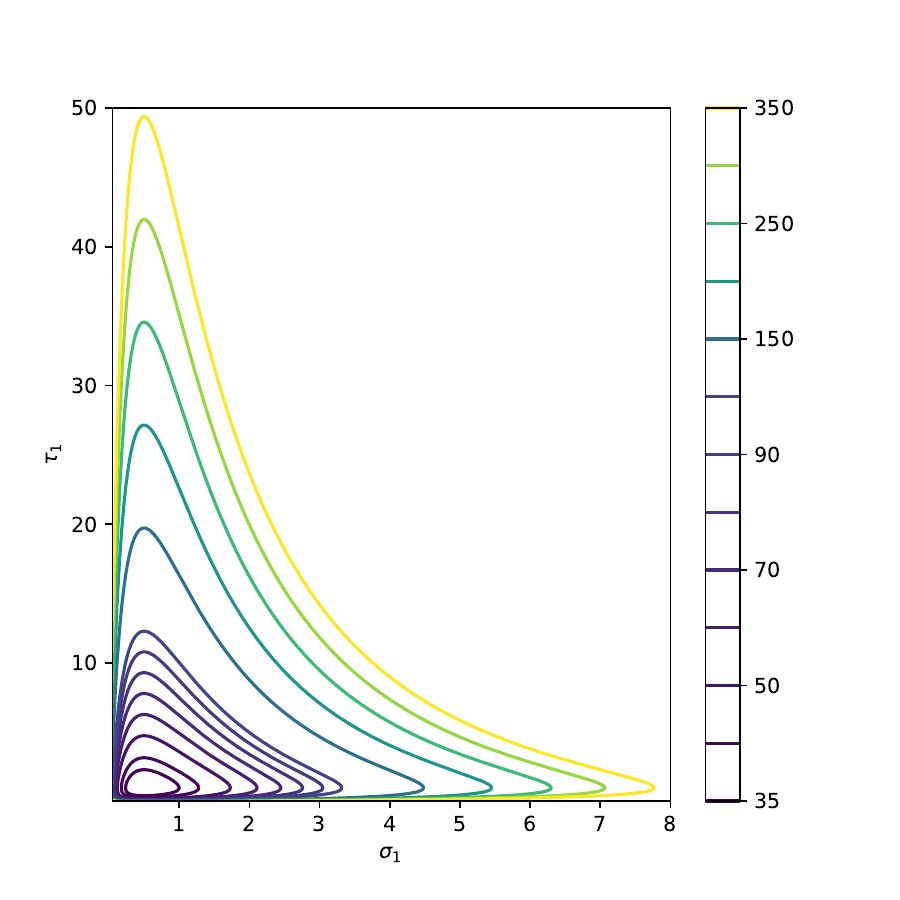}
	\caption{Level sets of $\trace_\delta$ on the unipotent locus.}
	\label{fig : Intro level sets}
\end{figure}

Theorem \ref{thm Intro : main theorem} has the following application. Sikora in \cite{Sikora:graphs} (see also \cite{DKS:webs}) introduced the notion of $3$--webs, which are $3$--regular bipartite graphs on a surface. Similar to the case of closed curves, a web $m$ defines a trace function $\trace_m\colon\widehat\DefSpace(\pants)\to\R$. For the case of the \emph{$\Theta$--web}, shown in Figure \ref{fig : theta web}, its trace, when restricted to a symplectic leaf, is given by $C-\trace_\delta$ for a constant $C$ depending on the symplectic leaf (see Section \ref{sec : trace of the theta web}). Thus by Theorem \ref{thm Intro : main theorem}, we immediately obtain the following.
\begin{corA}[Corollary \ref{cor : trace of the theta web is periodic}]\label{cor Intro : web periodicity}
	Let $m_\Theta$ be the $\Theta$--web and $\SymLeaf$ be any symplectic leaf in $\widehat\DefSpace(\pants)$. Then the restriction of the trace function $\trace_{m_\Theta}\big|_{\SymLeaf}\colon\SymLeaf\to\R$ attains a unique maximum. Moreover, every orbit of the Hamiltonian flow of $\trace_{m_\Theta}\big|_{\SymLeaf}$ is periodic and there is a unique fixed point.
\end{corA}

One may also consider the function $\trace_\figeight+\trace_{\figeight^{-1}}$. This function may be more natural in the sense that the map $c\mapsto \trace_c+\trace_{c^{-1}}$ is invariant under taking inverses of the curve. In particular, the function only depends on the curve and \emph{not} on its orientation. We note that in the case of the group $\SL 2$, the function $c\mapsto \trace_c$ is already invariant under taking inverses. For the symmetric function $\trace_\figeight+\trace_{\figeight^{-1}}$ in the $\PSL 3$ case, we obtain a similar result as Theorem \ref{thm Intro : main theorem} as well as more information about the fixed point. Recall that a real convex projective structure on $P$ is said to be \emph{hyperbolic} if its holonomy representation factors through an irreducible representation $\PSL 2\to\PSL 3$ (see Section \ref{sec : fuchsian locus}). 

\begin{thmA}\label{thm intro : symmetrized trace}
	Let $\figeight = \alpha\gamma^{-1}$ be a figure eight curve on a pair of pants $P$ and let $\SymLeaf$ be any symplectic leaf in $\widehat\DefSpace(\pants)$. Then the following statements hold.
	\begin{enumerate}
\item \label{thm intro : periodicity for symmetrized trace} [Theorem \ref{thm : app symmetric trace}] The restriction of the function $\trace_{\figeight}+\trace_{\figeight^{-1}}\big|_{\SymLeaf}\colon\SymLeaf\to\R$ attains a unique minimum. Moreover, every orbit of the Hamiltonian flow of  $\trace_{\figeight}+\trace_{\figeight^{-1}}\big|_{\SymLeaf}$ is periodic, and there is a unique fixed point.
\item \label{thm intro : symmetrized trace fixed point} [Theorem \ref{thm : fixed points of symmetric trace}] Let $\SymLeaf$ be a symplectic leaf containing a hyperbolic structure. Then the fixed point of $\trace_{\figeight}+\trace_{\figeight^{-1}}\big|_{\SymLeaf}$ is the unique hyperbolic structure in $\SymLeaf$.
	\end{enumerate}
\end{thmA}

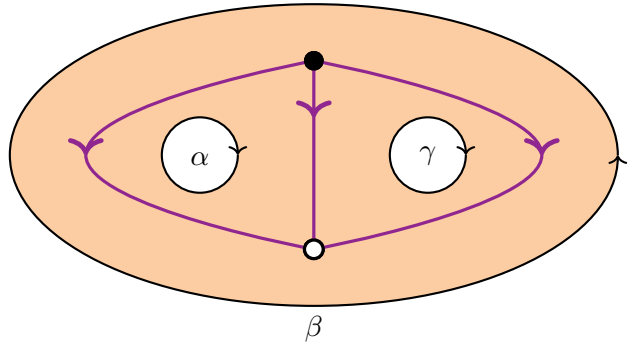
\begin{figure}[htb]
	\centering
	\begin{tikzpicture}
    \fill[Apricot!70] (0,0) ellipse (4cm and 2cm);
    
    \fill[white] (1.5,0) circle (0.5cm);
    \fill[white] (-1.5,0) circle (0.5cm);
    
    \draw[->,thick] (2,0) arc[start angle = 360, end angle = 0, radius = 0.5];
    \draw[->,thick] (-1,0) arc[start angle = 360, end angle = 0, radius = 0.5];
    \draw[->,thick] (4,0) arc [start angle = 0,end angle = 360, x radius = 4cm, y radius = 2cm];
    
    \node[label = $\gamma$] at (1.5,-0.4) {};
    \node[label = $\alpha$] at (-1.5,-0.4) {};
    \node[label = $\beta$] at (0,-2.75) {};
    
    \node[circle,draw=black,fill=black,inner sep=0pt,minimum size=7] (b) at (0,1.25) {};
    \node[very thick,circle,draw=black,fill=white,inner sep=0pt,minimum size=7] (w) at (0,-1.25) {};
    \draw[very thick,Plum] plot [smooth,tension = 1] coordinates {(b) (-3,0)(w)}[arrow inside={opt={scale=1.5}}{0.5}];
    \draw[very thick,Plum,] plot [smooth,tension = 1] coordinates {(b) (3,0)(w)}  [arrow inside={opt={scale=1.5}}{0.5}];
    \draw[very thick,Plum,] plot [smooth,tension = 1] coordinates {(b) (w)}  [arrow inside={opt={scale=1.5}}{0.3}];
    \node[circle,draw=black,fill=black,inner sep=0pt,minimum size=7] (b) at (0,1.25) {};
    \node[very thick,circle,draw=black,fill=white,inner sep=0pt,minimum size=7] (w) at (0,-1.25) {};
	\end{tikzpicture}
\caption{The $\Theta$--web on a pair of pants shown in purple.}
\label{fig : theta web}
\end{figure}

\subsection{Idea of the proof of Theorem \ref{thm Intro : main theorem}}
The proof of Theorem \ref{thm Intro : main theorem} relies on the following observation, which is a direct consequence of Kerckhoff's work on the Nielsen realization problem in \cite{Kerkhoff_NielsenRealization}. To state the observation we first make some definitions. Let $S = S_{g,n,b}$ be a surface of genus $g$ with $n\geq 0$ boundary components and $b\geq 0$ punctures such that $2g-2+n+b>0$, meaning that $S$ admits hyperbolic metrics. For a hyperbolic metric $m$ with geodesic boundary, let $\ell_c^m$ be the hyperbolic length of the geodesic representative of a closed curve $c$ on $S$. Let $c_1,\dots,c_n$ denote the boundary curves of $S$. Fix a \emph{length vector} $L=(\ell_1,\dots,\ell_n)\in\R^n_{>0}$ and let
\[
\teich_L(S)\coloneqq \{m \textnormal{ a hyperbolic metric on }S\st \ell_{c_i}^m = \ell_i\textnormal{ for all }i = 1,\dots,n\}/\mathsf{Diff_0(S)},
\]
called the \emph{Teichm\"uller space} of $S$. The dimension of $\teich_L(S)$ is $6g-6+2n+2b$.\\

The Teichm\"uller space $\teich_L(S)$ carries the well--known Weil--Petersson symplectic structure. Moreover, any closed curve $c$ in $S$ defines a length function 
\[
\ell_c\colon\teich_L(S)\to\R_{>0}.
\]
Using the Weil--Petersson symplectic structure, any length function gives rise to a Hamiltonian flow.\\

 A curve $c$ in $S$ is said to be \emph{filling} if its complement in $S$ is a disjoint union of disks, once punctured disks, and annuli which are homotopic to the boundary of $S$.

\begin{thmA}[Consequence of \cite{Kerkhoff_NielsenRealization}]\label{thm : periodicity for teichmueller}
	\sloppy Let $S$ be one of the surfaces $S_{1,0,1}, S_{1,1,0}, S_{0,0,4}, S_{0,1,3}, S_{0,2,2}, S_{0,3,1}, S_{0,4,0}$. Let $c$ be a closed filling curve in $S$ and $L$ a length vector. Then every orbit of the Hamiltonian flow of the function $\ell_c$ restricted to the Teichm\"uller space $\Teich_L(S)$ is periodic and there is a unique fixed point.
\end{thmA}

This observation is analogous to Theorem \ref{thm Intro : main theorem} since the figure eight curve $\figeight$ is a filling curve in the pair of pants $\pants$. The proof uses the notion of an \emph{earthquake}, which is a type of deformation in Teichm\"uller space. We do not define it here, as we will only need to know some facts about it. For any $t\in\R$, we will write $E_t\colon\teich_L(S)\to\teich_L(S)$ for an earthquake path in Teichm\"uller space.\\

The proof of the theorem hinges on the following four facts:

\begin{fact}\label{fact : teich of some surfaces is 2 dim}
	\sloppy The Teichm\"uller space of the surfaces $S_{1,0,1}$,  $S_{1,1,0}$, $S_{0,0,4}$, $S_{0,1,3}$, $S_{0,2,2}$, $S_{0,3,1}$, $S_{0,4,0}$ is two--dimensional. Moreover, these are the only surfaces whose Teichm\"uller space is two--dimensional. These surfaces are the one--holed torus with either a cusp or a boundary component, and the rest are four--holed spheres with different combinations of boundary components or cusps.
\end{fact}

\begin{fact}[Properness, Lemma 3.1 in \cite{Kerkhoff_NielsenRealization}]\label{fact 1 : properness}
	If $c$ is a filling curve in $S$, then the function $\ell_c\colon\teich_L(S)\to\R_{>0}$ is a proper function.
\end{fact}

\begin{fact}[Earthquake Theorem, \cite{ThurstonEarthquakes}]\label{fact 2 : earthquake theorem}
	Any two points in $\teich_L(S)$ can be connected via an earthquake path.
\end{fact}

\begin{fact}[Strict convexity, Theorem 2 in  \cite{Kerkhoff_NielsenRealization}]\label{fact : convexity}
	Let $E_t$ be an earthquake path. Then if $c$ is a filling curve, the function $t\mapsto \ell_c(E_t(m))$ is a strictly convex function for any $m\in\teich_L(S)$. 
\end{fact}

With these three results, we can prove Theorem \ref{thm : periodicity for teichmueller}.\\

\begin{proof}[of Theorem \ref{thm : periodicity for teichmueller}]
	To start, we make the same observation made in \cite[pp. 236]{Kerkhoff_NielsenRealization}, which we restate here and which holds for any surface $S$ of negative Euler characteristic. Since the function $\ell_c$ is proper by Fact \ref{fact 1 : properness} and $\ell_c$ is a positive function, it realizes a minimum. By Fact \ref{fact : convexity}, $\ell_c$ is strictly convex along any earthquake path. Since any pair of points in $\teich_L(S)$ can be connected by an earthquake path by Fact \ref{fact 2 : earthquake theorem}, $\ell_c$ attains a unique minimum. In particular, the function $\ell_c$ has a unique critical point. Therefore, the Hamiltonian flow of $\ell_c$ has a unique fixed point, corresponding to the minimum.\\
	
	Now we show that every orbit is periodic, where we need $S$ to be one of the surfaces in the statement of the theorem. Let $M$ be a non--empty level set of $\ell_c$ that does not correspond to the fixed point. In particular, $M$ is a regular level set and therefore a smooth codimension one submanifold of $\Teich_L(S)$. The Hamiltonian flow of $\ell_c$ preserves the level sets of $\ell_c$ (see for example \cite[pp. 99]{IntroSymplTop_McDuffSalamon}).  Properness of $\ell_c$ implies that $M$ is compact. By the fact that $\teich_L(S)$ is two--dimensional by Fact \ref{fact : teich of some surfaces is 2 dim}, $M$ is a compact one--dimensional manifold. Therefore $M$ is a topological circle. Once again, since $M$ does not contain any fixed points, the Hamiltonian vector field restricted to $M$ is bounded away from zero. This implies that the orbit is the whole level set $M$ and is therefore periodic.
\end{proof}

The strategy to prove Theorem \ref{thm Intro : main theorem} is to prove analogous results to Facts \ref{fact : teich of some surfaces is 2 dim}, \ref{fact 1 : properness}, \ref{fact 2 : earthquake theorem} and \ref{fact : convexity} for the symplectic leaves of $\widehat\DefSpace(\pants)$, where $P$ is the pair of pants. As mentioned above, a symplectic leaf $\SymLeaf$ of $\widehat\DefSpace(\pants)$ is a two--dimensional manifold (see Lemma \ref{cor : parametrization of symplectic leaf as level sets} for their parameterization). Hence, we already have an analogue of Fact \ref{fact : teich of some surfaces is 2 dim}. For properness (Fact \ref{fact 1 : properness}), we prove the following.

\begin{propA}[Proposition \ref{prop : trace of figure 8 is proper}]\label{prop : Intro properness figure 8}
		Let $\figeight = \alpha\gamma^{-1}$ be the figure eight curve and let $\SymLeaf\subset \widehat\DefSpace(\pants)$ be a symplectic leaf. Then the function $\trace_\figeight\big|_{\SymLeaf}\colon\SymLeaf\to\R$ is proper and positive. In particular, it realizes a minimum in $\SymLeaf$.
\end{propA}
This is proved by explicit computations of the trace function in Fock--Goncharov coordinates. There is no proper analogue of the earthquake theorem, but we work with the \emph{eruption} and the \emph{hexagon} flow (in analogy with the flows defined in \cite{WienhardZhang_Deforming,FlowsPGLVHitchin_SWZ}), which we discuss in Section \ref{sec : Hamiltonian vector fields and flows}. These flows commute, and any two points in $\SymLeaf$ can be connected by a combination of these two flows, as stated in Lemma \ref{lemma : any two points are connected by a mixed flow}. This provides a working analogue of Fact \ref{fact 2 : earthquake theorem}.\\

 For an analogue of Fact \ref{fact : convexity} about the convexity of length functions along earthquake paths, we prove the following result, which seems to be of independent interest.
\begin{thmA}[Theorem \ref{thm : convexity of the trace function along mixed flows}]\label{thm : Intro convexity}
		Let $\figeight=\alpha\gamma^{-1}$ be a figure eight curve and let $\SymLeaf\subset \widehat\DefSpace(\pants)$ be a symplectic leaf. Then the trace function $\trace_\figeight\big|_\SymLeaf\colon\SymLeaf\to\R$ is strictly convex along the eruption and the hexagon flows.
\end{thmA}

The proof relies heavily on the positivity of the Fock--Goncharov coordinates.

\begin{remark}
	As shown in \cite{GhostPolygons}, certain notions of length functions on geodesic laminations are convex along Hamiltonian flows of length functions. Theorem \ref{thm : Intro convexity} gives another type of flow along which length functions of certain curves (here interpreted as a trace function) are convex.
\end{remark}

The proof of Theorem \ref{thm Intro : main theorem} is then exactly the same as that of Theorem \ref{thm : periodicity for teichmueller} using all of the above results. The individual proofs of Proposition \ref{prop : Intro properness figure 8} and Theorem \ref{thm : Intro convexity} go through explicit computations using Mathematica and the Fock--Goncharov coordinates.
\begin{remark}
	In Equation \eqref{eq : Hamiltonian vector field in symplectic leaf in coordinates sigma1 tau1} we also provide a formula for the Hamiltonian vector field of a function given in coordinates $\sigma_1,\tau_1$ of a symplectic leaf. The resulting expressions for the trace of the figure eight curve are explicit, but it is still hard to find a closed--form formula for the solution of the differential equation. However, we can use Mathematica (and Python) to numerically solve the equations. Throughout the paper, we show some of the numerical solutions in simple cases.
\end{remark}

\subsection{Beyond the figure eight curve}
Theorem \ref{thm : periodicity for teichmueller} seems to suggest that if $\gamma$ is any \emph{filling} curve in $\pants$ and $\SymLeaf\subset \widehat\DefSpace(\pants)$ is a symplectic leaf, then an analogue of Theorem \ref{thm Intro : main theorem} should hold. Namely, that every orbit of the Hamiltonian flow of $\trace_\gamma\big|_\SymLeaf$ is periodic and that there is a unique fixed point corresponding to the minimum of the function. As our methods rely on explicit computations, they do not allow to make a very general statement about any curve. However, to provide evidence for the periodicity of Hamiltonian flows associated to filling curves, we provide two more examples.\\

For this, we focus on the \emph{unipotent locus} $\unipotentLocus\subset\widehat\DefSpace(\pants)$ (see Definition \ref{def : unipotent locus}). These are the framed convex projective structures whose holonomies for the peripheral curves are \emph{unipotent}, by which we mean that all of the eigenvalues are equal to one. With the same methods as above, we prove 
\begin{thmA}[Theorem \ref{thm : commutator periodicity and unique fixed point} and Theorem \ref{thm : periodicity alpha^kgammainverse}]\label{thm : intro commutator and k intersections}
	Consider the curves $[\alpha,\gamma]$ (see Figure \ref{fig : commutator}) and $\alpha^k\gamma^{-1}$ for $k\in\N_{>0}$ (see Figure \ref{fig : k self intersections}). The restriction of the trace functions $\trace_{[\alpha,\gamma]}\big|_{\unipotentLocus}\colon\unipotentLocus\to\R$ and $\trace_{\alpha^k\gamma^{-1}}\big|_{\unipotentLocus}\colon\unipotentLocus\to\R$ attain a unique minimum. Moreover, every orbit of the  Hamiltonian flows of $\trace_{[\alpha,\gamma]}\big|_{\unipotentLocus}$ and $\trace_{\alpha^k\gamma^{-1}}\big|_{\unipotentLocus}$ is periodic and there is a unique fixed point.
\end{thmA} 

In Section \ref{sec : unipotent locus} we also find the fixed points of the Hamiltonian flows appearing in the above theorem.

\subsubsection{Conjugating matrices for eruption and hexagon flows} In the unipotent locus, the expressions for the holonomies of the matrices simplify significantly, allowing us to make even more computations. The symplectic leaves of $\widehat\DefSpace(\pants)$ correspond to relative character varieties (see Lemma \ref{lemma : symplectic leaves are exactly relative character varieties}), which themselves are subsets of $\CharVar_3(\pants)$ where the holonomies of the peripheral elements are in fixed chosen conjugacy classes. Since $\pi_1(\pants)$ is generated by the peripheral elements $\alpha,\beta$ and $\gamma$ and the flows remain in a symplectic leaf, the flows must be realized by a conjugation of the peripheral elements. In Theorem \ref{thm : conjugating matrices for eruption} we find matrices $\zeta^\alpha_t,\zeta^\beta_t,\zeta^\gamma_t$ in $\PSL 3$ such that the flow of representations
 \[
\rho_t = \begin{cases}
 	\alpha\mapsto \zeta_t^\alpha\rho(\alpha)(\zeta_t^\alpha)^{-1}\\
 	\beta\mapsto \zeta_t^\beta\rho(\beta)(\zeta_t^\beta)^{-1}\\
 	\gamma\mapsto \zeta_t^\gamma\rho(\gamma)(\zeta_t^\gamma)^{-1}
 \end{cases}
\]
covers the holonomies of the eruption flow on the unipotent locus $\unipotentLocus$. Since $\rho_t$ are representations of the fundamental group of the pair of pants, the above matrices are a solution to
\[
\zeta_t^\alpha\rho(\alpha)(\zeta_t^\alpha)^{-1}\zeta_t^\beta\rho(\beta)(\zeta_t^\beta)^{-1}\zeta_t^\gamma\rho(\gamma)(\zeta_t^\gamma)^{-1} = \identity.
\]
This is a particular instance of a solution to the \emph{Deligne--Simpson problem} \cite{Kostov_Deligne-SimpsonProblem} and is also solved in our particular case in \cite[Section 4.2]{EvProds_KenyonOvenhouse}. Similarly, we describe in Theorem \ref{thm : conjugating matrices for hexagon} the hexagon flow in terms of conjugations of the peripheral elements.

\subsection{Questions}
In this article we only address very specific self--intersecting curves and as mentioned above, our methods rely on explicit computations. We may therefore ask the following

\begin{question}
Let $c$ be a filling curve (equivalently a self--intersecting curve) in the pair of pants $P$. Is every orbit of the Hamiltonian flow of $\trace_c$ restricted to a symplectic leaf $\SymLeaf\subset\widehat\DefSpace(P)$ periodic with a unique fixed point?
\end{question}

An important fact used to prove Theorem \ref{thm Intro : main theorem} and Theorem \ref{thm : periodicity for teichmueller} is that the symplectic leaves all have dimension two. This is a consequence of the topology of the surfaces we work with. This then motivates the following
\begin{question}
Let $c$ be a filling curve in a surface $S$ with negative Euler characteristic. Do there exist periodic orbits of the Hamiltonian flow of $\trace_c$ restricted to a symplectic leaf $\SymLeaf\subset\widehat\DefSpace(S)$?
\end{question}

We may ask the same question when considering length functions on the Teichm\"uller space $\teich_L(S)$. In this case, the length function of any filling curve is proper and there is a unique minimum. The above question therefore only asks if there exist periodic orbits (outside of the minimum).
\subsection{Brief explanation for the Mathematica code}
\sloppy Most of the results in this article are aided by computations done with Mathematica, and plots made using Python. The code can be downloaded here \href{https://github.com/CamachoCadena/Periodic-orbits-of-Hamiltonian-flows.git}{https://github.com/CamachoCadena/Periodic-orbits-of-Hamiltonian-flows.git}.\\

In order to run the Mathematica code, all sections must be run in the order they appear. The code is adapted so that it can take in different inputs; therefore not all the equations in this article will appear as output of the code. Rather, the user must run the appropriate sections of the code, potentially giving inputs themselves, so that the desired output is shown. Throughout the article, we will explain which sections are necessary to run. The sections in the code are numbered, and we briefly explain their contents and where they are used.
\begin{enumerate}
	\item \textit{Fock--Goncharov's reconstruction of the holonomy through coordinates:} From the construction recalled in Section \ref{sec : reconstructing the represenatation}, this section computes the holonomies of the boundary curves as shown in Section \ref{sec : holonomies pair of pants description}.
	\item \textit{Casimir functions (ratios of eigenvalues):} Computes the ratios of eigenvalues of the boundary curves presented in Lemma \ref{lemma : eigenvalue ratios}, as well as the Jacobian matrix induced by the Casimir functions used in the proof of Lemma \ref{lemma : casimir functions pair of pants}.
	\item \textit{Parameterization of symplectic leaves:} This is the computation needed in Lemma \ref{cor : parametrization of symplectic leaf as level sets}.
	\item \textit{Functions on symplectic leaves:} Defines trace functions on the symplectic leaves. The functions are used throughout Sections \ref{sec : trace of figure eight curve} and \ref{sec : unipotent locus}.
	\item \textit{Hamiltonian vector field:} Computes the Hamiltonian vector field in a symplectic leaf. This is used in Sections \ref{sec : numerical sol to Ham flow fig 8}, \ref{sec : commutator unipotent locus} and \ref{sec : alpha^kgamma^-1}.
	\item \textit{Convexity:} Computes second derivatives along different flows and is used in the proof of Theorem \ref{thm : convexity of the trace function along mixed flows}, Proposition \ref{prop : commutator strictly convex}.
	\item \textit{Computations for $\alpha^k \gamma^{-1}$ in the unipotent locus:} Computes second derivatives along different flows used in the proof of Proposition \ref{prop : a^kgamma^-1 convexity}, as well as the Hamiltonian vector field in Section \ref{sec : alpha^kgamma^-1}.
	\item \textit{Conjugating matrices for the eruption flow:} Verifies Theorem \ref{thm : conjugating matrices for eruption}.
	\item \textit{Conjugating matrices for the hexagon flow:} Verifies Theorem \ref{thm : conjugating matrices for hexagon}.
\end{enumerate}
The Python code is used to make plots and find numerical solutions to differential equations, as well as for Remark \ref{rmk : periods are not the same}. It is not used in any of the proofs.

\subsection{Organization of the article} In Section \ref{sec : Prelim Periodic} we recall some basics in Poisson geometry, cross ratios, and framed convex projective structures. In Section \ref{sec : Big FG coords} we recall the Fock--Goncharov coordinates and give the matrices of boundary curves of the pair of pants in coordinates. In Section \ref{sec : Poisson structure and symplectic leaves} we compute the Poisson structure of $\widehat\DefSpace(P)$, find the Casimir functions, parameterize the symplectic leaves, and give a formula for the symplectic form. In Section \ref{sec : trace of figure eight curve} we prove Theorem \ref{thm Intro : main theorem} and Corollary \ref{cor Intro : web periodicity}. In Section \ref{sec : unipotent locus} we focus on the unipotent locus and prove Theorem \ref{thm : intro commutator and k intersections}, as well as Theorems \ref{thm : conjugating matrices for eruption} and \ref{thm : conjugating matrices for hexagon}.

\subsection*{Acknowledgements} 
I would like to thank my advisors Anna Wienhard and James Farre for suggesting the problem, suggesting to explore it computationally and for helpful discussions. I would especially like to thank Marit Bobb for explaining and helping me compute the symplectic form on symplectic leaves from the Poisson structure in Section \ref{sec : Poisson structure computation}. I  would also like to thank Bill Goldman for suggesting to consider the symmetric trace, Zachary Greenberg for discussions on Fock--Goncharov coordinates, and Tengren Zhang for his remarks on framed representations. Finally, I thank Arnaud Maret for his helpful comments on the introduction and especially his remarks on symplectic structures on relative character varieties and useful references.
\tableofcontents 
\section{Preliminaries}\label{sec : Prelim Periodic}

\subsection{Some Poisson geometry}\label{sec : prelim poisson geometry}
Here we follow \cite{PoissonGeom} and \cite{Weinstein_PoissonManifolds}. We begin with a definition.
\begin{definition}
	A \emph{Poisson manifold} $M$ is a smooth manifold endowed with a \emph{Poisson bracket} $\{\cdot,\cdot\}$ on $C^\infty(M)$.
	
\end{definition}

Given a function $f\in C^\infty(M)$, its \emph{Hamiltonian vector field} is the vector field $\Hm f\in\vf{M}$ satisfying
\begin{equation}\label{eq : definition of Hamiltonian vector field}
dg(\Hm f)  = \{g,f\} 
\end{equation}
for every $g\in C^\infty(M)$. The flow of the Hamiltonian vector field of a function $f\in C^\infty(M)$ at time $t\in\R$ is denoted by $\Phi^t_f\colon M\to M$.

\begin{lemma}\label{lemma : Hamiltonian flows commute if Poisson bracket is constant}
	Let $f,g\in C^\infty(M)$ and assume that there is a constant $c\in\R$ such that $\{f,g\}\equiv c$. Then the Hamiltonian vector fields $\Hm f$ and $\Hm g$ commute.
\end{lemma}
\begin{proof}
	Since $\{f,g\}\equiv c$, it follows that $\Hm {\{f,g\}}\equiv 0$. Then by the fact that $\Hm {\{f,g\}} = [\Hm_f,\Hm_g]$ (see for example \cite[Section 18.3]{SymplecticGeom_daSilva}), the lemma follows.
\end{proof}

The Poisson bracket defines a \emph{cosymplectic structure}, which is a map $\omega^\vee\colon \Omega^1(M)\times\Omega^1(M)\to C^\infty(M)$ given locally by
\[
\omega^\vee(dx,dy) = \{x,y\},
\]
for local coordinates $x,y$ on $M$. The \emph{radical} of the cosymplectic structure is the subspace
\[
\Rad(\omega^\vee)\coloneqq \{\alpha\in\Omega^1(M)\st \omega^\vee (\alpha,\cdot) \equiv 0\}.
\]
The dimension of the fibers of $\Omega^1(M)/\Rad(\omega^\vee)$ is called the \emph{rank} of the Poisson structure. 

\begin{definition}
	A \emph{Casimir function} is a function $f\in C^\infty(M)$ such that
	\[
	\{f,g\} = 0
	\]
	for every $g\in C^\infty(M)$. The set of Casimir functions forms a ring, denoted by $\ringCasimir(M)$.
\end{definition}

From the cosymplectic form, we obtain a map
\[
[\omega^\vee]\colon\Omega^1(M)/\Rad(\omega^\vee)\to\vf M
\]
defined by
\[
\beta([\omega^\vee]([\alpha])) = \omega^\vee(\alpha,\beta)
\]
for any $\alpha,\beta\in\Omega^1(M)$.\\ 


The image of the map $[\omega^\vee]$ is given by 
\begin{equation}\label{eq : image of cosymplectic form}
	\{X\in\vf M\st \beta(X) \equiv 0\quad\textnormal{for every }\beta\in\Rad(\omega^\vee)\},
\end{equation}
and vector fields in the image of $[\omega^\vee]$ are generated by Hamiltonian vector fields.\\

An \emph{orbit} in a Poisson manifold is an equivalence class $\SymLeaf\subseteq M$ given by the following relation. Two points $p,q\in M$ are equivalent if there exist functions $f_1,\dots,f_k\in C^\infty(M)$ such that
	\[
	\Phi_{f_1}^1\circ\cdots\circ\Phi_{f_k}^1(p) = q.
	\]
It is a classical result that orbits in Poisson manifolds are symplectic submanifolds, whose symplectic structure is inherited by the Poisson structure, see for example \cite[Theorem 4.1]{PoissonGeom}.

\begin{definition}
	A \emph{symplectic leaf} of a Poisson manifold is a pair $(\SymLeaf,\omega_{\SymLeaf})$, where $\SymLeaf$ is an orbit, and $\omega_\SymLeaf$ is the induced symplectic structure. The \emph{symplectic foliation} of $M$ is the collection of symplectic leaves
	\[
	\SympFoliation = \{(\SymLeaf,\omega_\SymLeaf)\st\SymLeaf\textnormal{ is a symplectic leaf}\}.
	\]
\end{definition}

\begin{remark}\label{rmk : level sets of casimirs is symplectic leaf}
	In the special case when the rank of the Poisson structure is constant, a symplectic leaf is the common level set of the Casimir functions. This result is due to Weinstein in \cite[pp. 529]{Weinstein_PoissonManifolds}.
\end{remark}

Since the symplectic leaves are orbits, we give the following standard definition, in which we abuse nomenclature.
\begin{definition}\label{def : Hamiltonians generating the symplectic leaves}
	A collection of functions $\{f_1,\dots,f_k\}\subset C^\infty(M)$ whose Hamiltonian flows generate every symplectic leaf in $\SympFoliation$ are called the \emph{Hamiltonians} of the Poisson structure.
\end{definition}

\subsection{The full flag variety and invariants of flags}\label{sec : invariants of configurations of flags}
Here we describe the full flag variety of $\R^3$ and, following \cite{FockGoncharov,FockGoncharov_ConvexProjective}, give definitions of cross ratios and triple ratios. These are later used in Section \ref{sec : Fock Goncharov coordinates} to describe Fock and Goncharov's parameterization of framed positive representations.\\

The \emph{(full) flag variety of $\R^3$}, denoted by $\flags$, is the space of tuples $(p,\ell)\in\RPTwo\times(\RPTwo)^*$ such that $\ell(p)=0$. We say that two of flags $(p_1,\ell_1)$ and $(p_2,\ell_2)$ are \emph{transverse} or \emph{in generic position} if $\ell_1(p_2)\neq 0\neq \ell_2(p_1)$. With this, define $\flags_n$ to be the set of ordered $n$--tuples of flags that are pairwise transverse.\\

The kernel of a projective class of a linear functional defines a projective line in $\RPTwo$. If $(p,\ell)\in\flags$, then this means that $p$ is contained in the projective line defined by $\ell$. Throughout, we will not make a distinction between projective classes of linear functionals, which we write as row vectors, and projective lines.\\

Given any four pairwise distinct projective lines $\ell_1,\ell_2,\ell_3,\ell_4$ in $\RPTwo$ that go through a point $p\in\RPTwo$, we can define their \emph{cross ratio}. Namely, we take another projective line $h$ which intersects $\ell_1,\dots,\ell_4$ at distinct points $p_1,\dots,p_4$ respectively, and that are not equal to $p$. Then take any identification of $h$ with $\mathbb{RP}^1\cong\R\cup\{\infty\}$ sending $p_i$ to $x_i$. The cross ratio of $(\ell_1,\dots,\ell_4)$ is
\[
	\Cr(\ell_1,\dots,\ell_4)\coloneqq \frac{(x_1-x_2)(x_3-x_4)}{(x_1-x_4)(x_2-x_3)}.
\]
This convention for the cross ratio is the one such that $\Cr(\infty,-1,0,x_4) = x_4$. The cross ratio is a projective invariant, meaning that if $g\in\PSL 3$, then $\Cr(g\cdot\ell_1, g\cdot\ell_2,g\cdot\ell_3,g\cdot\ell_4) = \Cr(\ell_1,\dots,\ell_4)$.\\

The cross ratio on projective lines is used to define projective invariants of generic $4$--tuples of flags. For two distinct points $p,q\in\RPTwo$, let $\overline{pq}$ be the projective line containing $p$ and $q$.

\begin{definition}
	Let $(F_1,F_2,F_3,F_4) = ((p_1,\ell_1),(p_2,\ell_2),(p_3,\ell_3),(p_4,\ell_4))\in\flags_4$. Let
	\begin{align}
		\Cr_1(F_1,F_2,F_3,F_4)&\coloneqq -\frac{\ell_1(p_2)(\overline{p_1p_3})(p_4)}{\ell_1(p_4)(\overline{p_1p_3})(p_2)},\\
		\Cr_2(F_1,F_2,F_3,F_4)&\coloneqq-\frac{\ell_3(p_4)(\overline{p_1p_3})(p_2)}{\ell_3(p_2)(\overline{p_1p_3})(p_4)}.
	\end{align}
\end{definition}
These cross ratios can also be defined geometrically as follows.

\begin{lemma}\label{lemma : expression for FG cross ratios}
	Let $(F_1,F_2,F_3,F_4) = ((p_1,\ell_1),(p_2,\ell_2),(p_3,\ell_3),(p_4,\ell_4))\in\flags_4$. Then
	\begin{align*}
		\Cr(\ell_1,\overline{p_1p_2},\overline{p_1p_3},\overline{p_1p_4}) &= \Cr_1(F_1,F_2,F_3,F_4),\\
		\Cr(\ell_3,\overline{p_3p_4},\overline{p_3p_1},\overline{p_3p_2}) &= \Cr_2(F_1,F_2,F_3,F_4).
	\end{align*}
\end{lemma}

Note the lines $\ell_1,\overline{p_1p_2},\overline{p_1p_3},\overline{p_1p_4}$ all pass through $p_1$ and hence it makes sense to compute their cross ratio. The expressions on the left hand--side of the lemma are the invariants used by Fock and Goncharov in \cite{FockGoncharov_ConvexProjective}, and the lemma is included to give the formula for the cross ratios without having to pick identifications with $\R\mathbb P^1$ (similarly to the invariants defined in \cite{BonahonDreyer_GeneralLaminations}). As we will not use this lemma in the article, its proof is delayed to Appendix \ref{app : proof of cross ratios lemma}. Up to a sign, these are the cross ratios used in \cite{WienhardZhang_Deforming} to parameterize the space of convex projective structures on a surface.\\

Another invariant of flags is the triple ratio.
\begin{definition}
	Let $(F_1,F_2,F_3) = ((p_1,\ell_1),(p_2,\ell_2),(p_3,\ell_3))\in\flags_3$. The \emph{triple ratio} of these flags is given by
	\[
	\triple(F_1,F_2,F_3) = \frac{\ell_1(p_2)\ell_2(p_3)\ell_3(p_1)}{\ell_1(p_3)\ell_2(p_1)\ell_3(p_2)}.
	\]
\end{definition}

The triple ratio is also invariant under the $\PSL 3$ action, that is, for every $g\in\PSL 3$ and $(F_1,F_2,F_3)\in\flags_3$,
\[
\triple(g\cdot F_1,g\cdot F_2, g\cdot F_3) = \triple(F_1,F_2,F_3).
\]

\subsection{Convex projective structures on surfaces and framings}\label{sec : conv proj surfaces Prelim}
In this section, we recall the notion of $\RPTwo$ surfaces and their corresponding deformation spaces from \cite{ConvexRealProj_Goldman} and also following \cite{LoftinZhang_CoordinatesAugmented}. Throughout, $S$ is a closed surface with $n\geq 1$ boundary components. Denote the boundary loops by $c_1,\dots,c_n\in\pi_1(S)$. Such loops are called \emph{peripheral loops}.

\begin{definition}
	An \emph{$\RPTwo$ surface} $\Sigma$ is a smooth surface with boundary with a maximal collection of charts $\{\psi_\alpha\colon U_\alpha\to\RPTwo\}_\alpha$ such that
	\begin{itemize}
		\item Each $U_\alpha\subset \Sigma$ is a connected, simply connected open subset of the interior of $\Sigma$.
		\item For any $\psi_\alpha,\psi_\beta$ with $U_\alpha \cap U_\beta\neq\emptyset$, the map $\psi_\alpha\circ\psi_\beta^{-1}\colon U_\alpha \cap U_\beta\to \psi_\alpha(U_\alpha \cap U_\beta)$ is the restriction of a projective transformation of $\RPTwo$.
	\end{itemize}
\end{definition}

Let $\Sigma$ and $\Sigma'$ be two $\RPTwo$ surfaces with atlases $\{\psi_\alpha\}_\alpha$ and $\{\psi_\beta'\}_\beta$ respectively. A diffeomorphism  $f\colon \Sigma\to\Sigma'$ is called a \emph{projective isomorphism} if for any $U_\alpha$ and $U'_\beta$ such that $f(U_\alpha)\cap U_\beta\neq\emptyset$, the following map
\[
\psi'_\beta\circ f\circ\psi_\alpha^{-1}\colon\psi_\alpha(U_\alpha\cap f^{-1}(U'_\beta))\to\psi'_\beta(f(U_\alpha)\cap U_\beta)
\]
is the restriction of a projective map on $\RPTwo$ on each connected component.\\

The universal cover $\widetilde\Sigma$ of $\Sigma$ is also an $\RPTwo$ surface. And hence, by Theorem 2.2 in \cite{ConvexRealProj_Goldman}, there exists a smooth map $\dev_\Sigma\colon \widetilde\Sigma\to\RPTwo$, and a representation $\rho_\Sigma\colon\pi_1(\Sigma)\to \PSL 3$ such that the following diagram

\begin{figure}[h]
	\centering
	\begin{tikzcd}
\widetilde\Sigma \arrow[r,
"\dev_\Sigma"] \arrow[d,"\gamma" left]
& \RPTwo \arrow[d,
"\rho_\Sigma(\gamma)"] \\
\widetilde\Sigma \arrow[r,
"\dev_\Sigma"]
& \RPTwo
\end{tikzcd}
\end{figure}

commutes for every $\gamma\in\pi_1(\Sigma)$. The map $\dev_\Sigma$ is called a \emph{developing map}, and $\rho_\Sigma$ the \emph{holonomy}; together, the pair $(\dev_\Sigma,\rho_\Sigma)$ is called a \emph{developing pair} for $\Sigma$. Moreover, if $(\dev'_\Sigma,\rho_\Sigma')$ is another developing pair for $\Sigma$, there exists an element $g\in\PSL 3$ such that
\[
(\dev'_\Sigma,\rho'_\Sigma) = (g\cdot \dev_\Sigma,g\cdot \rho_\Sigma \cdot g^{-1}).
\]

We now focus on a particular class of $\RPTwo$ surfaces.

\begin{definition}
\begin{itemize}
	\item 	A \emph{properly convex domain} $\Omega\subset \RPTwo$ is an open subset whose closure does not contain any projective lines, and for any distinct $p,q\in\Omega$, there is a projective line segment connecting $p$ and $q$ and that is completely contained in $\Omega$.
	\item A connected $\RPTwo$ surface is \emph{convex} if some developing map of $\Sigma$ is a diffeomorphism onto a properly convex domain in $\RPTwo$ and it extends to the boundary. Moreover, we require that the development map sends boundary components to either points or line segments. If a boundary component is sent to a point, it is said to \emph{cuspidal}.
\end{itemize}

\end{definition}
\begin{definition}\label{def : deformation space of convex projective structures}
	The \emph{deformation space of convex projective structures on $S$} is
	\begin{equation}
	\DefSpace(S)\coloneqq \begin{Bmatrix}
		(f,\Sigma)\st \Sigma\textnormal{ is a convex } \RPTwo\textnormal{ surface }\\
		\textnormal{ and }f\colon S\to\Sigma\textnormal{ is a diffeomorphism}
	\end{Bmatrix}/\sim
	\end{equation}

	where $(f_1,\Sigma_1)\sim (f_2,\Sigma_2)$ if $f_1\circ f_2^{-1}\colon\Sigma_2\to\Sigma_1$ is homotopic to a projective isomorphism from $\Sigma_2$ to $\Sigma_1$. An equivalence class $[f,\Sigma]\in\DefSpace(S)$ is then called a \emph{(marked) convex projective structure on $S$}.
\end{definition}

The holonomy of an $\RPTwo$ surface provides a map
\[
\hol\colon\DefSpace(S)\to \CharVar_3(S)\coloneqq \Hom^{\mathsf{red}}(\pi_1(S),\PSL3)/\PSL 3
\]
as follows. Given a pair $(f,\Sigma)\in\DefSpace(S)$, we obtain, up to conjugation, a representation $\rho_\Sigma\colon\pi_1(\Sigma)\to\PSL 3$ and hence a representation $f^*\rho_\Sigma\colon\pi_1(S)\to\PSL 3$. Moreover, if two pairs $(f, \Sigma)$ and $(f',\Sigma')$ are equivalent, then the corresponding representations also differ by a conjugation. We denote the image of $\hol$ by $\RepsConv$.\\

In order to describe Fock--Goncharov coordinates, we need to have additional data on a convex projective structure, which is a framing. Recall that the boundary curves of $S$ are denoted by $c_i\in\pi_1(S)$ for $i = 1,\dots,n$.

\begin{definition}
	A \emph{framed representation} is a tuple $(\rho,F_1,\dots,F_n)$, where $\rho\colon\pi_1(S)\to\PSL 3$ is a representation, and $F_1,\dots,F_n\in\flags$ are flags such that $\rho(c_i)\cdot F_i = F_i$ for $i = 1,\dots,n$. The quotient of the space of framed representations by $\PSL 3$ will be written as $\framedChar$.
	\end{definition}

\begin{definition}
	A \emph{framed convex projective structure on $S$} is a triple $[f,\Sigma,\nu]$, where $[f,\Sigma]$ is convex projective structure on $S$ and $\nu$ is a framing of the holonomy representation. The space of framed convex $\RPTwo$ structures on $S$ is denoted by $\widehat\DefSpace(S)$. 
\end{definition}

	
	Fock and Goncharov provide a map 
		\[
	\widehat\hol\colon \widehat\DefSpace(S)\to\widehat\CharVar_3(S).
	\]
	in \cite[Theorem 2.5]{FockGoncharov_ConvexProjective} in which they assign a framing to the holonomy representations of a convex projective structure. We do not describe the map here, and simply use its existence.

\begin{definition}\cite{FockGoncharov_ConvexProjective}
	The image $\widehat\hol\left(\widehat\DefSpace(S)\right)\subset \framedChar$ will be denoted by $\framedRepsConv$, and is known as the space of \emph{positive framed representations}.
\end{definition}

There is a natural map 
\begin{align*}
\mu\colon\framedRepsConv&\to\CharVar_3^+(S)\\
[(\rho,F_1,\dots,F_n)]&\mapsto [\rho]
\end{align*}
which forgets the framing. This map is a ramified $2^n\colon 1$ cover \cite{FockGoncharov_ConvexProjective,FockGoncharov} (see also \cite{IntroPositive_Palesi}). The map ramifies over those representations where at least one of the $\rho(c_i)$ fixes fewer than $6$ flags, namely when $\rho(c_i)$ does not have $3$ distinct positive eigenvalues.\\

The space of positive representations is foliated itself by the following subsets. Recall that $c_1,\dots,c_n\in\pi_1(S)$ are the peripheral elements.
\begin{definition}\label{def : relative character variety}
	Let $\conjclass = (C_1,\dots,C_n)$ be a tuple of conjugacy classes in $\PSL 3$. The \emph{relative character variety} associated to $\conjclass$ is the subspace of the space of positive representations given by
	\[
	\CharVar_{3,\conjclass}^+(S)\coloneqq\{[\rho]\in\CharVar_3^+(S)\st\rho(c_i)\in C_i\textnormal{ for all }i=1,\dots,n\}.
	\]
\end{definition}

\begin{remark}\label{rmk : symplectic structure on relative character varieties}
	The relative character varieties are naturally equipped with a symplectic structure \cite{SymplecticNature_Goldman,ParabolicCohomology_GHJW}. We do not describe the symplectic structure here, as, up to a constant, the symplectic structure coincides with the symplectic structure from the symplectic leaves of $\widehat\DefSpace(S)$ through the map $\mu\circ\widehat\hol$ \cite[Section 5]{FockGoncharov_ConvexProjective},\cite{SwapAlg_Sun}. Using the Fock--Goncharov coordinates which we recall in Section \ref{sec : Fock Goncharov coordinates}, we describe the symplectic structure on the symplectic leaves explicitly in the case when $S$ is a pair of pants.
\end{remark}

\section{Fock--Goncharov coordinates}\label{sec : Big FG coords}
 Let $S$ be a surface of genus $g$ with $n\geq 1$ boundary components so that $2g-2+n>0$. Here we recall the Fock--Goncharov coordinates for framed convex projective structures. In the construction of Fock and Goncharov, a surface with punctures is required. Their construction is combinatorial and we therefore interpret the boundary of our surfaces as punctures by shrinking them to points. \textbf{We stress that whenever we refer to (framed) convex projective structures, $S$ has boundary components. Whenever we work with Fock--Goncharov coordinates, we interpret the boundaries as punctures. This is the same convention used in \cite{FockGoncharov_ConvexProjective}.}\\

We will describe the coordinates, explain how to reconstruct a representation from the coordinates, the Poisson structure, and give explicit matrices for the generators of the fundamental group of the pair of pants. The content of this section recalls constructions and results from \cite{FockGoncharov_ConvexProjective} (and more generally \cite{FockGoncharov}) and includes more explicit computations that allow us to derive the results of this article. To see another overview of the Fock--Goncharov results, see \cite{ModuliSpacesConvex_CTT,IntroPositive_Palesi}.

\subsection{The coordinates}\label{sec : Fock Goncharov coordinates}
In the first step, we interpret the boundary components of $S$ as punctures. Begin by taking an \emph{ideal triangulation} $\triangulation$ of $S$, meaning a triangulation of $S$ with vertices at the punctures. The triangulation consists of $2\abs{\chi(S)}$ ideal triangles and $6g+3n-6$ edges.\\

Fock and Goncharov define an isomorphism
\[
	\varphi_\triangulation\colon\widehat\DefSpace(S)\to \R_{>0}^{16g-16+8n}
\]
in \cite[Theorem 2.5]{FockGoncharov_ConvexProjective} giving coordinates on the space of framed convex projective structures, and in particular giving it the structure of a smooth manifold. The map is defined as follows. First lift the triangulation $\triangulation$ to a triangulation $\widetilde\triangulation$ of the universal cover $\widetilde S$. Let $[f,\Sigma,\nu]\in\widehat{\DefSpace}(S)$ and let $\dev_\Sigma$ be a developing map for $(f,\Sigma)$. We describe the two types of invariants associated to a framed convex projective structure.\\

The first is a cross ratio associated to edges of the triangulation. Let $\edge$ be an edge of the triangulation $\triangulation$ with vertices $p_{\edge,1}$ and $p_{\edge,2}$. The edge $\edge$ has two triangles $\triangle_1$ and $\triangle_2$ on either side with vertices $\{p_{\edge,1},p_{\edge,2},p_{\triangle,1}\}$ and $\{p_{\edge,1},p_{\edge,2},p_{\triangle,2}\}$ respectively. Take lifts $\{\widetilde p_{\edge,1},\widetilde p_{\edge,2},\widetilde p_{\triangle,1},\widetilde p_{\triangle,2}\}$ of the four points $\{p_{\edge,1},p_{\edge,2},p_{\triangle,1},p_{\triangle,2}\}$ to the universal cover so that $\{\widetilde p_{\edge,1},\widetilde p_{\edge,2},\widetilde p_{\triangle,1}\}$ are endpoints of a single lift $\widetilde\triangle_1$ of $\triangle_1$, $\{\widetilde p_{\edge,1},\widetilde p_{\edge,2},\widetilde p_{\triangle,2}\}$ are endpoints of a single lift $\widetilde\triangle_2$ of $\triangle_2$, and so that these lifts of the ideal triangles share the edge $\widetilde \edge$ with lifts of the endpoints $\{\widetilde p_{\edge,1},\widetilde p_{\edge,2}\}$. See Figure \ref{fig : invariants setup} for this setup. Up to renaming the points, we may choose an ordering of the lifts $\{\widetilde p_{\edge,1},\widetilde p_{\edge,2},\widetilde p_{\triangle,1},\widetilde p_{\triangle,2}\}$ such that
\[
\widetilde p_{\edge,1}\prec\widetilde p_{\triangle,1}\prec \widetilde p_{\edge,2}\prec\widetilde p_{\triangle,2}\prec \widetilde p_{\edge,1},
\]
where $\prec$ denotes the counterclockwise ordering on the circle.\\

\begin{figure}[htb]
	\centering
	\begin{tikzpicture}
		\node[minimum size=5cm, regular polygon,
  regular polygon sides=3,rotate = 90] (T1) at (0,0) {};
  \draw[very thick,Emerald] (T1.corner 3) -- (T1.corner 2);
   \draw[very thick, black] (T1.corner 1) -- (T1.corner 2);
    \draw[very thick, black] (T1.corner 3) -- (T1.corner 1);
    \node[label = left :$\widetilde p_{\triangle,1}$] at (T1.corner 1) {};
    \node[label = left :$\widetilde\triangle_1$] at (T1) {};
    \node[circle,inner sep=0pt,minimum size=6,label=left:$\widetilde\edge$] (7) at ($1/2*(T1.corner 2)+1/2*(T1.corner 3)$) {};
    \node[label = above :$\widetilde p_{\edge,1}$] at (T1.corner 3) {};
    \node[label = below :$\widetilde p_{\edge,2}$] at (T1.corner 2) {};

\node[minimum size=5cm, regular polygon,
  regular polygon sides=3,rotate = -90] (T2) at (2.5,0) {};
     \draw[very thick, black] (T2.corner 1) -- (T2.corner 2);
    \draw[very thick, black] (T2.corner 3) -- (T2.corner 1);
\node[label = right :$\widetilde p_{\triangle,2}$] at (T2.corner 1) {};
    \node[label = right:$\widetilde\triangle_2$] at (T2) {};
	\end{tikzpicture}
	\caption{Computing edge invariants.}
	\label{fig : invariants setup}
\end{figure}
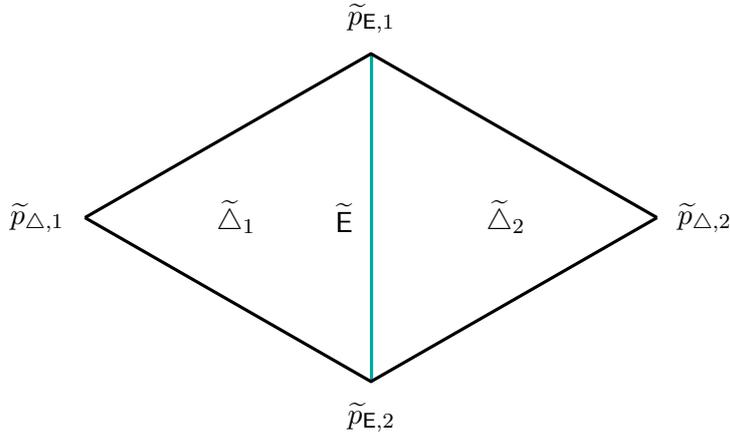

Through the developing map $\dev_\Sigma$ and the assignment of a framed representation, we obtain for each of the vertices, flags $F_{\edge,1},F_{\triangle,1},F_{\edge,2},F_{\triangle,2}\in\flags$ respectively. Now we can define the coordinates
\begin{align*}
\sigma_\edge^1([f,\Sigma,\nu])&\coloneqq \Cr_1(F_{\edge,1},F_{\triangle,1},F_{\edge,2},F_{\triangle,2})\\
\sigma_\edge^2([f,\Sigma,\nu])&\coloneqq \Cr_2(F_{\edge,1},F_{\triangle,1},F_{\edge,2},F_{\triangle,2})
\end{align*}

The second type of coordinate is the triple ratio associated to ideal triangles. Let $\triangle$ be an ideal triangle in $\triangulation$ with vertices $p_1,p_2,p_3$. Take lifts $\{\widetilde p_1,\widetilde p_2,\widetilde p_3\}$ of the three vertices to the universal cover such that $\{\widetilde p_1,\widetilde p_2,\widetilde p_3\}$ are vertices of a lift of $\triangle$. Up to renaming, assume that the vertices are ordered such that $\widetilde p_1\prec\widetilde p_2\prec\widetilde p_3\prec\widetilde p_1$. Similarly as above, the developing map $\dev_\Sigma$ together with the associated framing, we obtain three flags $F_1,F_2,F_3\in\flags$. Then let
\[
\tau_\triangle([f,\Sigma,\nu])\coloneqq \triple(F_1,F_2,F_3).
\]

The map $\varphi_\triangulation$ is then defined as
\[
[f,\Sigma]_\nu\mapsto \left(\left(\sigma^1_\edge([f,\Sigma,\nu]),\sigma^2_\edge([f,\Sigma,\nu])\right)_{\edge},(\tau_\triangle([f,\Sigma,\nu]))_\triangle\right)
\]
as the edges $\edge$ and triangles $\triangle$ vary in the triangulation $\triangulation$. The fact that all of these coordinates are positive is the content of Lemma 2.3 and Lemma 2.4 in \cite{FockGoncharov_ConvexProjective}.
\begin{remark}\label{rmk : Bonahon Dreyer coordinates are different than FG coords}
	We stress here that these coordinates, although closely related to the Bonahon--Dreyer coordinates in \cite{BonahonDreyer_GeneralLaminations} for closed surfaces, are not exactly the same coordinates. The difference lies in that the cross ratios and triple ratios are taken with respect to slightly different tuples of flags. This becomes evident in the computation of the holonomies, where we observe that the Bonahon--Dreyer closed leaf inequalities have a slightly different form.
\end{remark}

\subsection{Reconstructing the representation from coordinates}\label{sec : reconstructing the represenatation}
Recall that there is a natural map
\[
\framedRepsConv\coloneqq \widehat\hol\left(\widehat\DefSpace(S)\right)\to \RepsConv
\] 
which simply forgets the framing. In particular, there is a map
\[
\rechol\colon\widehat\DefSpace(S)\to\RepsConv
\]
which factors through the holonomy map $\widehat\hol$. In this section, we describe how to obtain a representation (up to conjugation) from coordinates, as described in \cite[Section 5]{FockGoncharov_ConvexProjective}. Fix an ideal triangulation $\triangulation$ of $S$ (after shrinking the boundaries to punctures).\\

We begin by constructing an embedded quiver $\quiver_\triangulation$ on $S$ as follows:
\begin{itemize}
	\item On the interior of each edge $\edge$ of the triangulation, place two distinct vertices $\vertex_{\edge,1}$ and $\vertex_{\edge,2}$.
	\item On the interior of each ideal triangle $\triangle$, place an additional vertex $\vertex_\triangle$.
	\item On an ideal triangle $\triangle$, place three $3$--cycles as in Figure \ref{fig : quiver}.
\end{itemize}

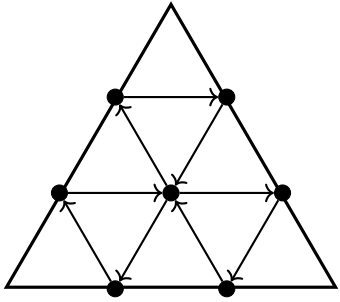
\begin{figure}[htb]
	\centering
	\begin{tikzpicture}
		\node[regular polygon,draw,regular polygon sides=3, minimum size=5cm, very thick,] (p) at (0,0) {};
        \node[circle,draw=black,fill=black,inner sep=0pt,minimum size=6] (2) at ($2/3*(p.corner 1)+1/3*(p.corner 2)$) {};
        \node[circle,draw=black,fill=black,inner sep=0pt,minimum size=6] (1) at ($2/3*(p.corner 2)+1/3*(p.corner 1)$) {};
        \node[circle,draw=black,fill=black,inner sep=0pt,minimum size=6] (6) at ($2/3*(p.corner 2)+1/3*(p.corner 3)$) {};
        \node[circle,draw=black,fill=black,inner sep=0pt,minimum size=6] (5) at ($2/3*(p.corner 3)+1/3*(p.corner 2)$) {};
        \node[circle,draw=black,fill=black,inner sep=0pt,minimum size=6] (3) at ($2/3*(p.corner 1)+1/3*(p.corner 3)$) {};
        \node[circle,draw=black,fill=black,inner sep=0pt,minimum size=6] (4) at ($2/3*(p.corner 3)+1/3*(p.corner 1)$) {};
        \node[circle,draw=black,fill=black,inner sep=0pt,minimum size=6] (7) at ($1/2*(1)+1/2*(4)$) {};
        \draw[thick,->] (2) to (3);
        \draw[thick,->] (3) to (7);
        \draw[thick,->] (7) to (2);
        
        \draw[thick,->] (1) to (7);
        \draw[thick,->] (7) to (6);
        \draw[thick,->] (6) to (1);
        
        \draw[thick,->] (7) to (4);
         \draw[thick,->] (4) to (5);
          \draw[thick,->] (5) to (7);
	\end{tikzpicture}
	\caption{Quiver embedded in an ideal triangle.}
	\label{fig : quiver}
\end{figure}

Each vertex in the quiver $\quiver_\triangulation$ is itself a coordinate function as follows. For each vertex $\vertex_\triangle$, we assign the coordinate function $\tau_\triangle$. For an edge $\edge$ with vertices $p_{\edge,1},p_{\edge,2}$ (the vertices here being the punctures of the surface), let $\{\widetilde p_{\edge,1},\widetilde p_{\edge,2},\widetilde p_{\triangle,1},\widetilde p_{\triangle,2}\}$ be the vertices of the lifts of the endpoints of the two triangles sharing the edge $\edge$ as in Section \ref{sec : Fock Goncharov coordinates}. Up to renaming, we may assume that the vertex $\vertex_{\edge,1}$ is the one closest to $p_{\edge,1}$. Then assign to the vertex $\vertex_{\edge,i}$ the coordinate function $\sigma^i_\edge$ whenever $\widetilde p_{\edge,1}\prec\widetilde p_{\triangle,1}\prec \widetilde p_{\edge,2}\prec\widetilde p_{\triangle,2}\prec \widetilde p_{\edge,1}$.\\

With the quiver $\quiver_\triangulation$ and the coordinate functions on the vertices,  construct a new oriented graph $\Gamma = (V,E)$ on the surface as follows:
\begin{itemize}
	\item Place three vertices in the interior of each triangle $\triangle$ of the triangulation $\triangulation$ and create a counterclockwise $3$--cycle connecting the vertices.
	\item For a given edge $\edge$ of the triangulation, place an edge (with an arbitrary orientation), connecting the vertices of $\Gamma$ closest to the edge $\edge$.
\end{itemize}
The graph $\Gamma$ is shown in Figure \ref{fig : reconstructing} for two adjacent ideal triangles.

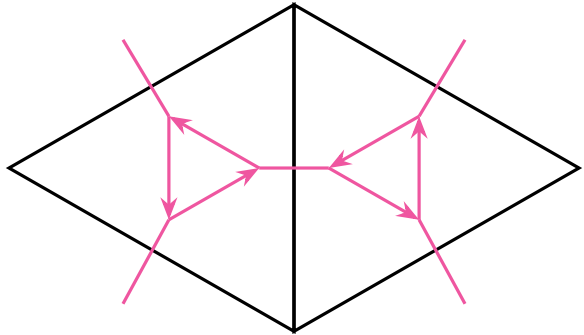
\begin{figure}[h!]
	\centering
	\begin{tikzpicture}
		\node[draw,very thick,minimum size=5cm, regular polygon,
  regular polygon sides=3,rotate = 90] (T1) at (0,0) {};
  \node[very thick,minimum size=1.5cm, regular polygon,
  regular polygon sides=3,rotate = -90] (IntT1) at (0,0) {};
  \draw[very thick,-Stealth,VioletRed] (IntT1.corner 1) -- (IntT1.corner 2);
  \draw[very thick,-Stealth,VioletRed] (IntT1.corner 2) -- (IntT1.corner 3);
  \draw[very thick,-Stealth,VioletRed] (IntT1.corner 3) -- (IntT1.corner 1);
\node[draw,very thick,minimum size=5cm, regular polygon,
  regular polygon sides=3,rotate = -90] (T2) at (2.5,0) {};
    \node[very thick,minimum size=1.5cm, regular polygon,
  regular polygon sides=3,rotate = 90] (IntT2) at (2.5,0) {};
    \draw[very thick,-Stealth,VioletRed] (IntT2.corner 1) -- (IntT2.corner 2);
  \draw[very thick,-Stealth,VioletRed] (IntT2.corner 2) -- (IntT2.corner 3);
  \draw[very thick,-Stealth,VioletRed] (IntT2.corner 3) -- (IntT2.corner 1);
  
    \draw[very thick,VioletRed] (IntT1.corner 1) -- (IntT2.corner 1);
    \draw[very thick,VioletRed] (IntT1.corner 2) -- (-1,1.7);
    \draw[very thick,VioletRed] (IntT1.corner 3) -- (-1,-1.8);
    \draw[very thick,VioletRed] (IntT2.corner 2) -- (3.5,-1.8);
    \draw[very thick,VioletRed] (IntT2.corner 3) -- (3.5,1.7);

	\end{tikzpicture}
	\caption{A portion of the graph $\Gamma$ used to reconstruct the holonomy from the Fock--Goncharov coordinates.}
	\label{fig : reconstructing}
\end{figure}

For $x,z,w>0$, define the following matrices in $\PSL3$:
\[
T(x)\coloneqq \frac{1}{x^{1/3}}\begin{bmatrix}
	0 & 0 & 1\\
	0 & -1 & -1\\
	x & 1+x &1
\end{bmatrix},\quad E(z,w)\coloneqq \frac{z^{1/3}}{w^{1/3}}\begin{bmatrix}
	0&0&\frac{1}{z}\\
	0&-1&0\\
	w&0&0
\end{bmatrix}.
\]
Now we are ready to describe the map $\rechol$. Let $[f,\Sigma,\nu]\in\widehat\DefSpace(S)$ with its corresponding coordinates given by $\varphi_\triangulation$. Consider the lift $\widetilde \Gamma$ of $\Gamma$ to the universal cover $\widetilde S$ and fix a vertex $p$ of $\widetilde \Gamma$. For a curve $\alpha\in\pi_1(S)$, we describe the holonomy $\rechol([f,\Sigma,\nu])(\alpha)$ (we abuse notation here picking a representative in the equivalence class of representations). Lift $\alpha$ to a curve $\widetilde\alpha\subset\widetilde S$ starting at the point $p$ and homotope it (relative endpoints) so that it lies on the graph $\widetilde \Gamma$\footnote{Here we have chosen a covering $\widetilde S\to S$ whose associated group of deck transformations we identify with the fundamental group $\pi_1(S,p)$.}. The endpoint of this path is $\alpha\cdot p$ (interpreting $\alpha$ as a deck transformation). The holonomy $\rechol([f,\Sigma,\nu])(\alpha)$ is the following product in the matrices $T(x)$ and $E(z,w)$. Every time the curve $\widetilde\alpha$ goes through an edge in $\widetilde\Gamma$ internal to a triangle $\triangle$, multiply with the matrix $T(\tau_\triangle([f,\Sigma,\nu]))^\eps$, where $\eps\in\{\pm 1\}$ depending on whether the edge is crossed according to the orientation of $\widetilde\Gamma$ or not. Whenever the path $\widetilde\alpha$ goes through an edge crossing an edge $\edge$ of the triangulation, multiply with the matrix $E(\sigma^1_\edge([f,\Sigma,\nu]),\sigma^2_\edge([f,\Sigma,\nu]))$ where we assume that the vertex $\vertex_{\edge,1}$ of the quiver $\quiver_\triangulation$ lies to the right of the segment of $\widetilde\alpha$ crossing $\edge$. The matrices are multiplied from left to right as $\widetilde\alpha$ traverses edges of $\widetilde\Gamma$.

\begin{remark}
	Since the coordinates given by $\varphi_\triangulation$ are projective invariants, the matrices $T(x)$ and $E(z,w)$ do not depend on the lifts chosen for the triangles or edges.
\end{remark}

\subsection{The Poisson structure}\label{sec : general Poisson structure}
Given a triangulation $\triangulation$ and its associated quiver $\quiver_\triangulation$ as in Section \ref{sec : Fock Goncharov coordinates}, the Poisson structure on $\widehat \DefSpace(S)\cong\R^{16g-16+8n}_{>0}$ is given as follows. For coordinate functions $X_i,X_j\colon \widehat \DefSpace(S)\to\R_{>0}$, their Poisson bracket is given by the function
\begin{equation}\label{eq : Poisson structure from FG}
\{X_i,X_j\} = 2\eps_{ij}X_iX_j,
\end{equation}
where
\[
\eps_{ij} = \#\{\textnormal{arrows from }i \textnormal{ to }j\} - \#\{\textnormal{arrows from }j \textnormal{ to }i\}
\]
in the quiver $\quiver_\triangulation$ where the coordinates are thought of as vertices \cite[Section 5.1]{FockGoncharov_ConvexProjective}.

\subsection{Describing the holonomies for convex projective structures on a pair of pants}\label{sec : holonomies pair of pants description}
In Section \ref{sec : reconstructing the represenatation}, we described how Fock and Goncharov reconstruct a representation given their coordinates for a general surface. In this section, we focus on the case when $S=\pants$ is a pair of pants. Shrinking the boundaries to punctures, we name the punctures $p_\alpha,p_\beta$ and $p_\gamma$ and choose generators $\alpha,\beta,\gamma$ for the fundamental group $\pi_1(\pants)$ satisfying the relation $\alpha\beta\gamma = 1$, and each going around the respective puncture. We then pick the ideal triangulation $\triangulation$ of $\pants$ given by the two triangles $\triangle_1,\triangle_2$ with vertices $p_\alpha,p_\beta,p_\gamma$ each. The associated quiver $\quiver_\triangulation$ is shown in Figure \ref{fig : quiver pair of pants}. For the pair of pants, we have the Fock--Goncharov coordinates
\begin{align*}
	\varphi_\triangulation\colon\widehat\DefSpace(\pants)&\to\R_{>0}^8\\
	[f,\Sigma,\nu]&\mapsto (\sigma_1,\dots,\sigma_6,\tau_1,\tau_2)
\end{align*}
where we renamed the coordinates as in Figure \ref{fig : quiver pair of pants} dropping the dependence on the edge and triangle in the notation. Following the construction of the map $\rechol\colon\widehat\DefSpace(\pants)\to\CharVar_3^+(S)$ in Section \ref{sec : reconstructing the represenatation}, and following Figure \ref{fig : holonomies for abc}, we obtain that an equivalence class $[\rho]$ in the image of $\rechol$ with coming from the coordinates $(\sigma_1,\dots,\sigma_6,\tau_1,\tau_2)$ is given by
\begin{align*}
	\rho(\alpha) &= E(\sigma_2,\sigma_1)T(\tau_2)E(\sigma_3,\sigma_4)T(\tau_1)\\
	&=\left[
\begin{array}{ccc}
 \sqrt[3]{\frac{\tau _1^2 \tau _2^2}{\sigma _1 \sigma_2^2 \sigma _3^2
   \sigma _4}} & \frac{\sqrt[3]{\frac{\sigma _2 \sigma _3}{\sigma _1 \sigma
   _4 \tau _1 \tau _2}} \left(\sigma _3 \tau _2+\sigma _3+\tau _1 \tau
   _2+\tau _2\right)}{\sigma _2 \sigma _3} & \frac{\sqrt[3]{\frac{\sigma _2
   \sigma _3}{\sigma _1 \sigma _4 \tau _1 \tau _2}} \left(\sigma _3
   \left(\sigma _4+\tau _2+1\right)+\tau _2\right)}{\sigma _2 \sigma _3} \\
 0 & \sqrt[3]{\frac{\sigma _2 \sigma _3}{\sigma _1 \sigma _4 \tau _1 \tau
   _2}} & \left(\sigma _4+1\right) \sqrt[3]{\frac{\sigma _2 \sigma
   _3}{\sigma _1 \sigma _4 \tau _1 \tau _2}} \\
 0 & 0 & \sqrt[3]{\frac{\sigma _1^2 \sigma _2 \sigma _3 \sigma _4^2}{\tau _1
   \tau _2}} \\
\end{array}
\right],\\
\rho(\gamma)&=T(\tau_1)E(\sigma_6,\sigma_5)T(\tau_2)E(\sigma_1,\sigma_2)\\
&= \left[
\begin{array}{ccc}
 \sqrt[3]{\frac{\sigma _1 \sigma _2^2 \sigma _5^2 \sigma _6}{\tau _1 \tau
   _2}} & 0 & 0 \\
 -\sigma _2 \left(\sigma _5+1\right) \sqrt[3]{\frac{\sigma _1 \sigma
   _6}{\sigma _2 \sigma _5 \tau _1 \tau _2}} & \sqrt[3]{\frac{\sigma _1
   \sigma _6}{\sigma _2 \sigma _5 \tau _1 \tau _2}} & 0 \\
 \frac{\sigma _2 \left(\sigma _6 \left(\sigma _5+\tau _1+1\right)+\tau
   _1\right) \sqrt[3]{\frac{\sigma _1 \sigma _6}{\sigma _2 \sigma _5 \tau _1
   \tau _2}}}{\sigma _6} & -\frac{\sqrt[3]{\frac{\sigma _1 \sigma _6}{\sigma
   _2 \sigma _5 \tau _1 \tau _2}} \left(\sigma _6 \left(\tau
   _1+1\right)+\tau _1 \left(\tau _2+1\right)\right)}{\sigma _6} &
   \sqrt[3]{\frac{\tau _1^2 \tau _2^2}{\sigma _1^2 \sigma _2 \sigma _5
   \sigma _6^2}} \\
\end{array}
\right]\\
\rho(\beta) &= T(\tau_1)^{-1}E(\sigma_4,\sigma_3)T(\tau_2)E(\sigma_5,\sigma_6)T(\tau_1)^{-1}.
\end{align*}
The expression for $\rho(\beta)$ is too complicated to fit in one line, so we give it here in terms of its entries:

\begin{gather*}
	\rho(\beta)_{11} = \frac{\sqrt[3]{\frac{\sigma _4 \sigma _5}{\sigma _3 \sigma _6 \tau _1 \tau
   _2}} \left(\sigma _4 \sigma _5 \left(\sigma _6 \left(\sigma _3+\tau
   _1+1\right)+\tau _1+1\right)+\tau _1 \left(\sigma _5 \left(\sigma _6+\tau
   _2+1\right)+\tau _2\right)\right)}{\sigma _4 \sigma _5},\\
   \rho(\beta)_{12} = \frac{\sqrt[3]{\frac{\sigma _4 \sigma _5}{\sigma _3 \sigma _6 \tau _1^4 \tau
   _2}} \left(\sigma _4 \left(\tau _1+1\right) \left(\sigma _6 \left(\sigma
   _3+\tau _1+1\right)+\tau _1\right)+\tau _1 \left(\sigma _6 \left(\tau
   _1+1\right)+\tau _1 \left(\tau _2+1\right)\right)\right)}{\sigma _4},\\
   \rho(\beta)_{13} = \frac{\sigma _6 \left(\sigma _4 \left(\sigma _3+\tau _1+1\right)+\tau
   _1\right) \sqrt[3]{\frac{\sigma _4 \sigma _5}{\sigma _3 \sigma _6 \tau
   _1^4 \tau _2}}}{\sigma _4},\\
   \rho(\beta)_{21} = -\frac{\tau _1^{2/3} \left(\sigma _5 \left(\sigma _6+\sigma _4 \left(\sigma
   _6+1\right)+\tau _2+1\right)+\tau _2\right)}{\sqrt[3]{\sigma _3 \sigma
   _4^2 \sigma _5^2 \sigma _6 \tau _2}},\\
   \rho(\beta)_{22} = -\frac{\sqrt[3]{\frac{\sigma _4 \sigma _5}{\sigma _3 \sigma _6 \tau _1 \tau
   _2}} \left(\left(\sigma _4+1\right) \sigma _6 \left(\tau _1+1\right)+\tau
   _1 \left(\sigma _4+\tau _2+1\right)\right)}{\sigma _4},\\
   \rho(\beta)_{23} = -\frac{\left(\sigma _4+1\right) \sigma _6 \sqrt[3]{\frac{\sigma _4 \sigma
   _5}{\sigma _3 \sigma _6 \tau _1 \tau _2}}}{\sigma _4},\\
   \rho(\beta)_{31} = \frac{\tau _1^{2/3} \left(\sigma _5 \left(\sigma _6+\tau _2+1\right)+\tau
   _2\right)}{\sqrt[3]{\sigma _3 \sigma _4^2 \sigma _5^2 \sigma _6 \tau _2}},\\
   \rho(\beta)_{32} = \frac{\sqrt[3]{\frac{\sigma _4 \sigma _5}{\sigma _3 \sigma _6 \tau _1 \tau
   _2}} \left(\sigma _6 \left(\tau _1+1\right)+\tau _1 \left(\tau
   _2+1\right)\right)}{\sigma _4},\\
   \rho(\beta)_{33} = \sqrt[3]{\frac{\sigma _5 \sigma _6^2}{\sigma _3 \sigma _4^2 \tau _1 \tau
   _2}}.
\end{gather*} 

The computations are found in Section 1 of the Mathematica code. From a direct computation found in Section 2 of the Mathematica code, we obtain the following.
\begin{lemma}\label{lemma : eigenvalue ratios}
	Let $(\sigma_1,\dots,\sigma_6,\tau_1,\tau_2)\in\widehat\DefSpace(\pants)$. Then the ratios of pairs of eigenvalues of $\rechol(\sigma_1,\dots,\sigma_6,\tau_1,\tau_2)(\alpha)$ are given by
	\[
	\sigma_1\sigma_4 \textnormal{ and }\frac{\tau_1\tau_2}{\sigma_2\sigma_3}.
	\]
	The ratios of pairs of eigenvalues of $\rechol(\sigma_1,\dots,\sigma_6,\tau_1,\tau_2)(\beta)$ are given by
	\[
	\sigma_3\sigma_6\textnormal{ and }\frac{\tau_1\tau_2}{\sigma_4\sigma_5}.
	\]
	The ratios of pairs of eigenvalues of $\rechol(\sigma_1,\dots,\sigma_6,\tau_1,\tau_2)(\gamma)$ are given by
	\[
		\sigma_2\sigma_5\textnormal{ and }\frac{\tau_1\tau_2}{\sigma_1\sigma_6}.
	\]
\end{lemma}

\begin{figure}[htb]
	\centering
	\begin{tikzpicture}
		\node[regular polygon,draw,regular polygon sides=3, minimum size=5cm, very thick,] (p) at (0,0) {};
        \node[circle,draw=black,fill=black,inner sep=0pt,minimum size=6, label = left:$\sigma_2$] (2) at ($2/3*(p.corner 1)+1/3*(p.corner 2)$) {};
        \node[circle,draw=black,fill=black,inner sep=0pt,minimum size=6,label = left:$\sigma_1$] (1) at ($2/3*(p.corner 2)+1/3*(p.corner 1)$) {};
        \node[circle,draw=black,fill=black,inner sep=0pt,minimum size=6,label={[label distance=-1cm]:$\sigma_6$}] (6) at ($2/3*(p.corner 2)+1/3*(p.corner 3)$) {};
        \node[circle,draw=black,fill=black,inner sep=0pt,minimum size=6,label ={[label distance=-1cm]:$\sigma_5$}] (5) at ($2/3*(p.corner 3)+1/3*(p.corner 2)$) {};
        \node[circle,draw=black,fill=black,inner sep=0pt,minimum size=6,label = right:$\sigma_3$] (3) at ($2/3*(p.corner 1)+1/3*(p.corner 3)$) {};
        \node[circle,draw=black,fill=black,inner sep=0pt,minimum size=6,label = right:$\sigma_4$] (4) at ($2/3*(p.corner 3)+1/3*(p.corner 1)$) {};
        \node[circle,draw=black,fill=black,inner sep=0pt,minimum size=6,label={[label distance=-1cm]:$\tau_1$}] (7) at ($1/2*(1)+1/2*(4)$) {};
        \draw[thick,->] (2) to (3);
        \draw[thick,->] (3) to (7);
        \draw[thick,->] (7) to (2);
        
        \draw[thick,->] (1) to (7);
        \draw[thick,->] (7) to (6);
        \draw[thick,->] (6) to (1);
        
        \draw[thick,->] (7) to (4);
         \draw[thick,->] (4) to (5);
          \draw[thick,->] (5) to (7);
          
       		\node[regular polygon,draw,regular polygon sides=3, minimum size=5cm, very thick,rotate=180] (T2) at (0,-2.5) {};
        \node[circle,draw=black,fill=black,inner sep=0pt,minimum size=6, label = right:$\sigma_3$] (T23) at ($2/3*(T2.corner 1)+1/3*(T2.corner 2)$) {};
        \node[circle,draw=black,fill=black,inner sep=0pt,minimum size=6,label = right:$\sigma_4$] (T24) at ($2/3*(T2.corner 2)+1/3*(T2.corner 1)$) {};
        \node[circle,draw=black,fill=black,inner sep=0pt,minimum size=6,label = left:$\sigma_2$] (T22) at ($2/3*(T2.corner 1)+1/3*(T2.corner 3)$) {};
        \node[circle,draw=black,fill=black,inner sep=0pt,minimum size=6,label = left:$\sigma_1$] (T21) at ($2/3*(T2.corner 3)+1/3*(T2.corner 1)$) {};
        \node[circle,draw=black,fill=black,inner sep=0pt,minimum size=6,label={[label distance=-1cm]:$\tau_2$}] (T2t) at ($1/2*(T21)+1/2*(T24)$) {};

                  \draw[thick,->] (6) to (T2t);
        \draw[thick,->] (T21) -> (6);
        \draw[thick,->] (T2t) to (T21);
        
        \draw[thick,->] (T2t) to (5);
        \draw[thick,->] (5) to (T24);
        \draw[thick,->] (T24) to (T2t);
        
        \draw[thick,->] (T2t) to (T23);
         \draw[thick,->] (T23) to (T22);
          \draw[thick,->] (T22) to (T2t);
          \node[circle,draw=black,fill=pink,inner sep=0pt,minimum size=6,label=above:$p_\alpha$] (pa) at (p.corner 1) {};
          \node[circle,draw=black,fill=pink,inner sep=0pt,minimum size=6,label=right:$p_\beta$] (pb) at (p.corner 3) {};
          \node[circle,draw=black,fill=pink,inner sep=0pt,minimum size=6,label=left:$p_\gamma$] (pc) at (p.corner 2) {};
          \node[circle,draw=black,fill=pink,inner sep=0pt,minimum size=6,label=below:$p_\alpha$] (pap) at (T2.corner 1) {};
          
	\end{tikzpicture}
	\caption{Quiver on ideal triangulation of a pair of pants. The top is triangle $\triangle_1$ and the bottom is triangle $\triangle_2$. The gluing pattern can be seen from the identification of the edge invariants. The punctures $p_\alpha,p_\beta$ and $p_\gamma$ are drawn in pink.}
	\label{fig : quiver pair of pants}
\end{figure}
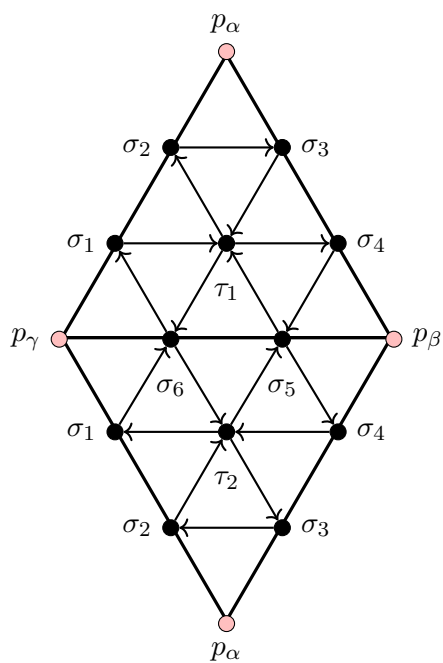

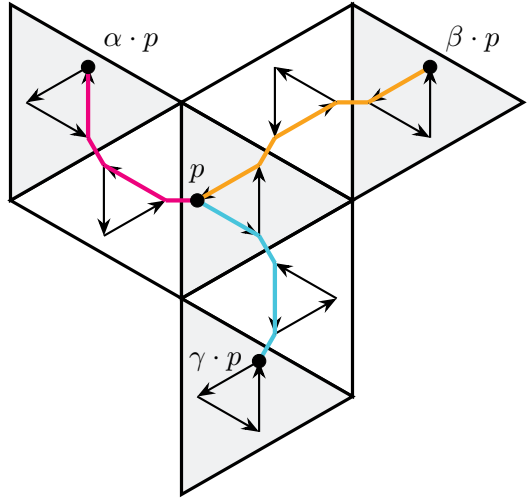
\begin{figure}[h!]
	\centering
	\begin{tikzpicture}
		\node[draw,very thick,minimum size=3cm, regular polygon,
  regular polygon sides=3,rotate = 90] (T1) at (0,0) {};
  \node[very thick,minimum size=1cm, regular polygon,
  regular polygon sides=3,rotate = -90] (IntT1) at (T1) {};
  \draw[thick,-Stealth,] (IntT1.corner 1) -- (IntT1.corner 2);
  \draw[thick,-Stealth,] (IntT1.corner 2) -- (IntT1.corner 3);
  \draw[thick,-Stealth,] (IntT1.corner 3) -- (IntT1.corner 1);

\node[draw,very thick,minimum size=3cm, regular polygon,
  regular polygon sides=3,rotate = -90,fill=Gray!12] (T2) at (1.5,0) {};
    \node[very thick,minimum size=1cm, regular polygon,
  regular polygon sides=3,rotate = 90] (IntT2) at (T2) {};
    \draw[thick,-Stealth,] (IntT2.corner 1) -- (IntT2.corner 2);
  \draw[thick,-Stealth,] (IntT2.corner 2) -- (IntT2.corner 3);
  \draw[thick,-Stealth,] (IntT2.corner 3) -- (IntT2.corner 1);
  
  \node[draw,very thick,minimum size=3cm, regular polygon,
  regular polygon sides=3,rotate = 90] (T3) at (2.25,1.3) {};
    \node[very thick,minimum size=1cm, regular polygon,
  regular polygon sides=3,rotate = -90] (IntT3) at (T3) {};
    \draw[thick,-Stealth,] (IntT3.corner 1) -- (IntT3.corner 2);
  \draw[thick,-Stealth,] (IntT3.corner 2) -- (IntT3.corner 3);
  \draw[thick,-Stealth,] (IntT3.corner 3) -- (IntT3.corner 1);
  
    \node[draw,very thick,minimum size=3cm, regular polygon,
  regular polygon sides=3,rotate = -90,fill=Gray!12] (T4) at (3.75,1.3) {};
    \node[very thick,minimum size=1cm, regular polygon,
  regular polygon sides=3,rotate = 90] (IntT4) at (T4) {};
    \draw[thick,-Stealth,] (IntT4.corner 1) -- (IntT4.corner 2);
  \draw[thick,-Stealth,] (IntT4.corner 2) -- (IntT4.corner 3);
  \draw[thick,-Stealth,] (IntT4.corner 3) -- (IntT4.corner 1);

    \node[draw,very thick,minimum size=3cm, regular polygon,
  regular polygon sides=3,rotate = -90,fill = Gray!12] (T5) at (-0.75,1.3) {};
    \node[very thick,minimum size=1cm, regular polygon,
  regular polygon sides=3,rotate = 90] (IntT5) at (T5) {};
    \draw[thick,-Stealth,] (IntT5.corner 1) -- (IntT5.corner 2);
  \draw[thick,-Stealth,] (IntT5.corner 2) -- (IntT5.corner 3);
  \draw[thick,-Stealth,] (IntT5.corner 3) -- (IntT5.corner 1);
  
    \node[draw,very thick,minimum size=3cm, regular polygon,
  regular polygon sides=3,rotate = 90] (T6) at (2.25,-1.3) {};
    \node[very thick,minimum size=1cm, regular polygon,
  regular polygon sides=3,rotate = -90] (IntT6) at (T6) {};
    \draw[thick,-Stealth,] (IntT6.corner 1) -- (IntT6.corner 2);
  \draw[thick,-Stealth,] (IntT6.corner 2) -- (IntT6.corner 3);
  \draw[thick,-Stealth,] (IntT6.corner 3) -- (IntT6.corner 1);
  
      \node[draw,very thick,minimum size=3cm, regular polygon,
  regular polygon sides=3,rotate = -90,fill = Gray!12] (T7) at (1.5,-2.6) {};
    \node[very thick,minimum size=1cm, regular polygon,
  regular polygon sides=3,rotate = 90] (IntT7) at (T7) {};
    \draw[thick,-Stealth,] (IntT7.corner 1) -- (IntT7.corner 2);
  \draw[thick,-Stealth,] (IntT7.corner 2) -- (IntT7.corner 3);
  \draw[thick,-Stealth,] (IntT7.corner 3) -- (IntT7.corner 1);
      \draw[thick,VioletRed] (IntT1.corner 1) -- (IntT2.corner 1);
    \draw[thick,] (IntT2.corner 3) -- (IntT3.corner 3);
    \draw[thick,] (IntT3.corner 1) -- (IntT4.corner 1);
    \draw[thick,] (IntT2.corner 2) -- (IntT6.corner 2);
    \draw[thick,] (IntT6.corner 3) -- (IntT7.corner 3);
    \draw[thick,] (IntT5.corner 2) -- (IntT1.corner 2);
      
    \draw[ultra thick, RubineRed] (IntT2.corner 1) -- (IntT1.corner 1);
    \draw[ultra thick, RubineRed] (IntT1.corner 1) -- (IntT1.corner 2);
    \draw[ultra thick, RubineRed] (IntT1.corner 2) -- (IntT5.corner 2);
    \draw[ultra thick, RubineRed] (IntT5.corner 2) -- (IntT5.corner 3);
    
    \draw[ultra thick, YellowOrange] (IntT2.corner 1) -- (IntT2.corner 3);
    \draw[ultra thick, YellowOrange] (IntT2.corner 3) -- (IntT3.corner 3);
    \draw[ultra thick, YellowOrange] (IntT3.corner 3) -- (IntT3.corner 1);
    \draw[ultra thick, YellowOrange] (IntT3.corner 1) -- (IntT4.corner 1);
    \draw[ultra thick, YellowOrange] (IntT4.corner 1) -- (IntT4.corner 3);
    
    \draw[ultra thick, SkyBlue] (IntT2.corner 1) -- (IntT2.corner 2);
    \draw[ultra thick, SkyBlue] (IntT2.corner 2) -- (IntT6.corner 2);
    \draw[ultra thick, SkyBlue] (IntT6.corner 2) -- (IntT6.corner 3);
    \draw[ultra thick, SkyBlue] (IntT6.corner 3) -- (IntT7.corner 3);

 \node[circle,draw=black,fill=black,inner sep=0pt,minimum size=5,label=$p$] (point) at (IntT2.corner 1){};
 
  \node[circle,draw=black,fill=black,inner sep=0pt,minimum size=5,label=above right:$\alpha\cdot p$] (point) at (IntT5.corner 3){};
  \node[circle,draw=black,fill=black,inner sep=0pt,minimum size=5,label=above right:$\beta\cdot p$] (point) at (IntT4.corner 3){};
  \node[circle,draw=black,fill=black,inner sep=0pt,minimum size=5,label=left:$\gamma\cdot p$] (point) at (IntT7.corner 3){};
 
	\end{tikzpicture}
	\caption{Computing the holonomies of the peripheral curves. To avoid a crowded figure, we do not include the labels for the coordinates functions. This figure shows part of the universal cover $\widetilde \pants$. The filled grey triangles correspond to lifts of $\triangle_1$ (whose triangle invariant is $\tau_1$), and the non--filled triangles are lifts of $\triangle_2$. To compute the holonomy, we pick an arbitrary point $p$ in the embedded graph $\Gamma$ from Figure \ref{fig : reconstructing}. The pink curve corresponds to a lift of $\alpha$, the orange curve corresponds to a lift of $\beta$, and the teal curve corresponds to a lift of $\gamma$. The respective endpoints of the lifts are labeled in the figure.}
	\label{fig : holonomies for abc}
\end{figure}

\section{The Poisson structure and symplectic leaves for the pair of pants}\label{sec : Poisson structure and symplectic leaves}

Here we describe explicitly the Poisson structure, symplectic leaves, Casimir functions, Hamiltonian functions, and flows for the pair of pants $\pants$. We also find an explicit parameterization of the symplectic leaves, allowing us to provide a closed--form formula for the symplectic structure.\\

\subsection{Poisson structure}\label{sec : Poisson structure computation}
Fix an ideal triangulation $\triangulation = \{\triangle_1,\triangle_2\}$ of $\pants$ with vertices $p_\alpha,p_\beta,p_\gamma$ as in Section \ref{sec : holonomies pair of pants description}. Let $\tau_i$ denote the triple ratio coordinate for the triangle $\triangle_i$. Then let $\sigma_1,\sigma_2,\dots,\sigma_6$ be the coordinates on the edges as shown in Figure \ref{fig : quiver pair of pants}. Using the basis $X_i = \sigma_i$ for $i = 1,\dots,6$ and $X_7 = \tau_1, X_8 = \tau_2$, the matrix $\eps_{ij}$ from \eqref{eq : Poisson structure from FG} is
\begin{equation}\label{eq : eps ij matrix for Poisson structure}
(\eps_{ij})_{i,j = 1,\dots,8}=\begin{bmatrix}
	0 &0&0&0&0&0&1&-1\\
	0 &0&0&0&0&0&-1&1\\
	0 &0&0&0&0&0&1&-1\\
	0 &0&0&0&0&0&-1&1\\
	0 &0&0&0&0&0&1&-1\\
	0 &0&0&0&0&0&-1&1\\
	-1&1&-1&1&-1&1&0&0\\
	1&-1&1&-1&1&-1&0&0
\end{bmatrix}.
\end{equation}
We compute that the radical of the associated cosymplectic structure is given by
\begin{align*}
	\Rad(\omega^\vee) = \langle\tau_1\cdot d\tau_2+\tau_2\cdot d\tau_1,\sigma_1\cdot d\sigma_2+\sigma_2\cdot d\sigma_1,\sigma_2\cdot d\sigma_3+\sigma_3\cdot d\sigma_2,\\
	\sigma_3\cdot d\sigma_4+\sigma_4\cdot d\sigma_3, \sigma_5\cdot d\sigma_6 + \sigma_6\cdot d\sigma_5 \rangle.
\end{align*}
Since the coordinates are positive, $\Rad(\omega^\vee)$ has constant dimension along $\widehat\DefSpace(\pants)$. In particular by Remark \ref{rmk : level sets of casimirs is symplectic leaf}, the symplectic leaves of $\widehat\DefSpace(\pants)$ are determined by the common level sets of the Casimir functions.

\subsection{Casimir functions and symplectic leaves}\label{sec : Casimir functions and symplectic leaves}
Here we give expressions for the Casimir functions and parameterize the symplectic leaves in terms of their common level sets.
\begin{lemma}\label{lemma : casimir functions pair of pants}
	The ring $\ringCasimir\subset C^\infty(\widehat\DefSpace(\pants))$ of Casimir functions of $(\widehat\DefSpace(\pants),\{\cdot,\cdot\})$ is generated by the functions
	\begin{align}
\casimir_{\alpha,1}\coloneqq \sigma_1\sigma_4, \qquad &\casimir_{\alpha,2}\coloneqq\frac{\tau_1\tau_2}{\sigma_2\sigma_3},\\
\casimir_{\beta,1}\coloneqq \sigma_3\sigma_6, \qquad & \casimir_{\beta,2} \coloneqq \frac{\tau_1\tau_2}{\sigma_4\sigma_5},\\
\casimir_{\gamma,1}\coloneqq \sigma_2\sigma_5,\qquad &\casimir_{\gamma,2}\coloneqq \frac{\tau_1\tau_2}{\sigma_1\sigma_6}.
\end{align}	
\end{lemma}

\begin{remark}\label{rmk : Casimirs are eigenvalue ratios}
	Note that from Lemma \ref{lemma : eigenvalue ratios}, the Casimir functions are exactly the eigenvalue ratios of the holonomies of the curves $\alpha,\beta$ and $\gamma$ respectively. 
\end{remark}

\begin{proof}
	Since the Poisson bracket is bilinear, it is enough to prove that the above functions are Casimir on a basis. By the definition of the Hamiltonian vector field in \eqref{eq : definition of Hamiltonian vector field}, we have that for any coordinate function $X_i$
	\begin{gather*}
	\{\casimir_{\alpha,1},X_i\}=\{\sigma_1\sigma_4,X_i\} = \sigma_1\{\sigma_4,X_i\}+\sigma_4\{\sigma_1,X_i\},\\
	\{\casimir_{\alpha,2},X_i\} = \left\{\frac{\tau_1\tau_2}{\sigma_2\sigma_3},X_i\right\} = \frac{\tau_2}{\sigma_2\sigma_3}\{\tau_1,X_i\}+\frac{\tau_1}{\sigma_2\sigma_3}\{\tau_2,X_i\}-\frac{\tau_1\tau_2}{\sigma_2^2\sigma_3}\{\sigma_2,X_i\}-\frac{\tau_1\tau_2}{\sigma_2\sigma_3^2}\{\sigma_3,X_i\},\\
	\{\casimir_{\beta,1},X_i\} = \{\sigma_3\sigma_6,X_i\} = \sigma_3\{\sigma_6,X_i\}+\sigma_6\{\sigma_3,X_i\},\\
	\{\casimir_{\beta,2},X_i\} = \left\{\frac{\tau_1\tau_2}{\sigma_4\sigma_5},X_i\right\} = \frac{\tau_2}{\sigma_4\sigma_5}\{\tau_1,X_i\}+\frac{\tau_1}{\sigma_4\sigma_5}\{\tau_2,X_i\}-\frac{\tau_1\tau_2}{\sigma_4^2\sigma_5}\{\sigma_4,X_i\}-\frac{\tau_1\tau_2}{\sigma_4\sigma_5^2}\{\sigma_5,X_i\},\\
	\{\casimir_{\gamma,1},X_i\} = \{\sigma_2\sigma_5,X_i\} = \sigma_2\{\sigma_5,X_i\}+\sigma_5\{\sigma_2,X_i\},\\
	\{\casimir_{\gamma,2},X_i\} = \left\{\frac{\tau_1\tau_2}{\sigma_1\sigma_6},X_i\right\} = \frac{\tau_2}{\sigma_1\sigma_6}\{\tau_1,X_i\}+\frac{\tau_1}{\sigma_1\sigma_6}\{\tau_2,X_i\}-\frac{\tau_1\tau_2}{\sigma_1^2\sigma_6}\{\sigma_1,X_i\}-\frac{\tau_1\tau_2}{\sigma_1\sigma_6^2}\{\sigma_6,X_i\}.
	\end{gather*}
	The fact that the functions in the proposition are Casimir then immediately follows from the form of the matrix in \eqref{eq : eps ij matrix for Poisson structure} defining the Poisson structure.\\
	
	To see that the functions generate the ring of Casimirs, we observe that the (fibers) of the dimension of the radical $\Rad(\omega^\vee)$ is $6$. Hence, we need to show that the Jacobian matrix of the family of functions $\{\casimir_{\alpha,1},\casimir_{\alpha,2},\casimir_{\beta,1},\casimir_{\beta,2},\casimir_{\gamma,1},\casimir_{\gamma,2}\}$ always has rank $6$. Indeed, we compute in Section 2 of the Mathematica code that the kernel of the Jacobian matrix 
	\begin{align*}
\left(\frac{\partial f}{\partial X_i}\right)_{\substack{f\in\{\casimir_{\alpha,1},\casimir_{\alpha,2},\casimir_{\beta,1},\casimir_{\beta,2},\casimir_{\gamma,1},\casimir_{\gamma,2}\},\\X_i\in\{\sigma_1,\dots,\sigma_6,\tau_1,\tau_2\}}}
\end{align*}
is given by 
\begin{equation}\label{eq : kernel of Jacobian matrix}
\left\langle \tau_1\del{\tau_1}-\tau_2\del{\tau_2},\sum_{i = 1}^6(-1)^{i+1}\sigma_i\del{\sigma_i} \right\rangle.
\end{equation}
Since the coordinates are always positive, the rank of the Jacobian matrix is always $6$, as desired.
\end{proof}

\begin{remark}
	These equations correspond to the \emph{(weak) closed leaf inequalities} of Bonahon--Dreyer in \cite{BonahonDreyer_FiniteLaminations} (for the pair of pants case, see \cite[pp. 27]{LoftinZhang_CoordinatesAugmented}). One can observe however that the equations themselves look slightly different. This is due to the fact that flags used in the Bonahon--Dreyer coordinates differ from the flags used by Fock--Goncharov (see Remark \ref{rmk : Bonahon Dreyer coordinates are different than FG coords}).
\end{remark}

\begin{lemma}\label{cor : parametrization of symplectic leaf as level sets}
	Let $L = (\ell_{\alpha,1},\ell_{\alpha,2},\ell_{\beta,1},\ell_{\beta,2}.\ell_{\gamma,1},\ell_{\gamma,2})\in\R^6_{>0}$. Let
	\begin{align*}
	\sigma_{2,L}(\sigma_1,\tau_1)&\coloneqq \frac{(\ell_{\alpha,1}\ell_{\beta,2}\ell_{\gamma,1})^{2/3}}{\sigma_1(\ell_{\alpha,2}\ell_{\beta,1}\ell_{\gamma,2})^{1/3}}	,\\
	\sigma_{3,L}(\sigma_1,\tau_1)&\coloneqq \frac{(\ell_{\beta,1}\ell_{\gamma,2})^{2/3}\sigma_1}{(\ell_{\alpha,1}\ell_{\alpha,2}\ell_{\beta,2}\ell_{\gamma,1})^{1/3}},\\
	\sigma_{4,L}(\sigma_1,\tau_1)&\coloneqq\frac{\ell_{\alpha,1}}{\sigma_1},\\
	\sigma_{5,L}(\sigma_1,\tau_1)&\coloneqq\frac{(\ell_{\alpha,2}\ell_{\beta,1}\ell_{\gamma,1}\ell_{\gamma,2})^{1/3}\sigma_1}{(\ell_{\alpha,1}\ell_{\beta,2})^{2/3}},\\
	\sigma_{6,L}(\sigma_1,\tau_1)&\coloneqq\frac{(\ell_{\alpha,1}\ell_{\alpha,2}\ell_{\beta,1}\ell_{\beta,2}\ell_{\gamma,1})^{1/3}}{\ell_{\gamma,2}^{2/3}\sigma_1},\\
	\tau_{2,L}(\sigma_1,\tau_1)&\coloneqq\frac{(\ell_{\alpha,1}\ell_{\alpha,2}\ell_{\beta,1}\ell_{\beta,2}\ell_{\gamma,1}\ell_{\gamma,2})^{1/3}}{\tau_1}.
	\end{align*}
	
	The symplectic leaf corresponding to $L$, denoted by $\SymLeaf_L$ is given by
	\begin{align*}
		\SymLeaf_L = \{(\sigma_1,\sigma_{2,L},\sigma_{3,L},\sigma_{4,L},\sigma_{5,L},\sigma_{6,L},\tau_1,\tau_{2,L})\in\widehat\DefSpace(\pants)\st \sigma_1,\tau_1>0\}.
	\end{align*}
	In particular, any symplectic leaf is two--dimensional. In the description of the set, we dropped the dependence of the functions $\sigma_{i,L}$ and $\tau_{2,L}$ for readability.
\end{lemma}
\begin{proof}
	The proof is a computation solving for $\sigma_2,\dots,\sigma_6$ and $\tau_2$ in terms of $\sigma_1,\tau_1$ and the vector $L$, found in Section 3 of the Mathematica code.
\end{proof} 

\begin{definition}
	We call a vector $L\in\R^6_{>0}$ that defines a symplectic leaf a \emph{length vector}.
\end{definition}

The symplectic leaves in $\widehat\DefSpace(\pants)$ correspond precisely to the relative character varieties in $\CharVar_3^+(\pants)$ via the map $\rechol$.

\begin{lemma}\label{lemma : symplectic leaves are exactly relative character varieties}
	For any length vector $L\in\R^6_{>0}$, there is a tuple of conjugacy classes $\conjclass(L)$ in $\PSL 3$ such that the map
	\[
	\rechol\big|_{\SymLeaf_L}\colon\SymLeaf_L\to\CharVar_{3,\conjclass(L)}^+(\pants)
	\]
	is an isomorphism.
\end{lemma}
To prove this lemma, we need the following result due to Marquis.
\begin{theorem}\label{thm : Marquis}\cite{ConvexProj_Marquis}
	Let $[\rho]\in\CharVar_3^+(S)$. If $\gamma$ is a peripheral element, then the conjugacy class of $\rho(\gamma)$ has to contain 
	\[
	\begin{pmatrix}
		1 &1&0\\0&1&1\\0&0&1
	\end{pmatrix}, \begin{pmatrix}
		\lambda_1&0&0\\0&\lambda_2 &1\\0&0&\lambda_2
	\end{pmatrix} \textnormal{ or } \begin{pmatrix}
		\lambda_1&0&0\\0&\lambda_2&0\\0&0&\lambda_3
	\end{pmatrix}
	\]
	for some pairwise distinct and positive $\lambda_1,\lambda_2,\lambda_3$. 
\end{theorem} 

\begin{proof}[of Lemma \ref{lemma : symplectic leaves are exactly relative character varieties}]
	By definition, the map $\rechol\colon\widehat\CharVar_3^+(\pants)\to \CharVar_3^+(\pants)$ is surjective. Then since $\widehat\CharVar_3^+(\pants)$ is foliated by the symplectic leaves $\SymLeaf_L$ as $L$ ranges in $\R^6_{>0}$, we only have to show that for each $L\in\R^6_{>0}$ there is a tuple of conjugacy classes $\conjclass(L)$ such that $\rechol(\SymLeaf_L)\subset \CharVar_{3,\conjclass(L)}^+(\pants)$ and that restricted to a symplectic leaf, $\rechol$ is injective.\\
	
	The eigenvalue ratios of the peripheral holonomies are exactly the Casimir functions (see Remark \ref{rmk : Casimirs are eigenvalue ratios}). By Theorem \ref{thm : Marquis}, we see that the conjugacy class of the holonomy of a peripheral lement is completely determined by the eigenvalue ratios, and hence there is a map $L\mapsto \conjclass(L)$. In particular, $\rechol(\SymLeaf_L)\subset \CharVar_{3,\conjclass(L)}^+(\pants)$. To see that the restriction to a symplectic leaf is injective, note that changing the framing changes the invariants. Therefore the framed convex projective structures with different framings lie in different symplectic leaves. Since the only ambiguity in the map $\rechol$ comes from the framing, this implies that $\rechol$ is injective when restricted to the symplectic leaves.
\end{proof}
\begin{remark}
	The map $L\mapsto \conjclass(L)$ is in general not injective, meaning that two symplectic leaves may be mapped to the same relative character variety. Indeed, if $L$ contains a pair $\ell_{x,1}\neq \ell_{x,2}\neq 1$ for $x\in\{\alpha,\beta,\gamma\}$, then interchanging $\ell_{x,1}$ with $\ell_{x,2}$ will define the same conjugacy class. This is because the Weyl group action on $\PSL 3$ can interchange the order of the eigenvalues. However, in the case when $\conjclass$ consists of only unipotent conjugacy classes, there is only one vector $L\in\R^6_{>0}$ mapping to $\conjclass$, namely the vector $(1,1,1,1,1,1)$.
\end{remark}

In Section \ref{sec : unipotent locus} we will focus on a special symplectic leaf, corresponding to the case when the peripheral holonomies are \emph{unipotent}. An element in $\PSL 3$ is unipotent if all of its eigenvalues are equal to one.

\begin{definition}\label{def : unipotent locus}
	The \emph{unipotent locus} of $\widehat\DefSpace(\pants)$, denoted by $\unipotentLocus$, is the symplectic leaf where the Casimir functions are all equal to $1$. That is
	\[
	\unipotentLocus\coloneqq\SymLeaf_{(1,1,1,1,1,1)}
	\]
	in the notation of Lemma \ref{cor : parametrization of symplectic leaf as level sets}. Through the map $\rechol$, $\unipotentLocus$ is identified with the relative character variety where all the peripheral elements are unipotent.
\end{definition}

Putting the spaces of framed convex projective structures, symplectic leaves, and framed representations, together with their counterparts without the framing, we provide the diagram below describing the relationships between all of these spaces:

\begin{figure}[H]
	\centering
	\begin{tikzcd}
\widehat{\DefSpace}(\pants)\arrow[dr,"\rechol"] \arrow[r] \arrow[d,"\widehat\hol" left]
& \DefSpace(P) \arrow[d,
"\hol"] \\
\R^8_{>0}\cong \widehat\CharVar_3^+(\pants) \arrow[r,swap,"\mu"]
& \CharVar_3^+(\pants)\\
\R^2_{>0}\cong\SymLeaf_L\arrow[u, phantom, sloped, "\subseteq"]\arrow[r, "\cong"]& \CharVar_{3,\conjclass}^+(\pants)\arrow[u, phantom, sloped, "\subseteq"]
\end{tikzcd}
\end{figure}

On the left are spaces which include framed structures, and on the right are the non--framed structures. The map $\rechol$ is from Section \ref{sec : reconstructing the represenatation}, which forgets the framing of the holonomy representation, the map $\mu$ from Section \ref{sec : conv proj surfaces Prelim} forgets the framing. The isomorphism at the bottom of the diagram is the content of Lemma \ref{lemma : symplectic leaves are exactly relative character varieties}. 

\subsection{The Fuchsian locus}\label{sec : fuchsian locus}
A (framed) convex projective structure on $\pants$ is said to be \emph{hyperbolic} (or \emph{Fuchsian}) if it corresponds to a hyperbolic structure on $P$ (either with cusps or with geodesic boundary). In terms of its holonomy representation, this means that it factors through an irreducible representation $\PSL 2\to\PSL 3$. As we are dealing with a pair of pants, there is a unique hyperbolic structure once the lengths of the boundary data are fixed. In particular, a symplectic leaf contains at most one hyperbolic structure. Since all ideal triangles in $\Htwo\subset \R\mathbb P^2$ are equivalent, one can compute that the triple ratio of an ideal triangle is always equal to one. Moreover, there is a single cross ratio that can be associated to a quadruple of flags, and hence in the coordinates, we have that 
\begin{equation}\label{eq : sigmas for fuchsian locus}
\sigma_1 = \sigma_2, \quad \sigma_3 = \sigma_4\quad \textnormal{ and }\quad \sigma_5 = \sigma_6.
\end{equation}
Let $F = (\ell_\alpha,\ell_\beta,\ell_\gamma)\in\R^3_{>0}$. By Lemma \ref{cor : parametrization of symplectic leaf as level sets}, we observe that the only symplectic leaves that contain a hyperbolic structure are the ones corresponding to the length vectors
\begin{equation}\label{eq : Fuchsian locus lengths L(F)}
L(F) = \left(\ell_\alpha,\frac{1}{\ell_\alpha},\ell_\beta,\frac{1}{\ell_\beta},\ell_\gamma,\frac{1}{\ell_\gamma}\right).
\end{equation}
Moreover, by solving the Casimir equations from Lemma \ref{lemma : casimir functions pair of pants}, we have that for $F = (\ell_\alpha,\ell_\beta,\ell_\gamma)\in\R^3_{>0}$, the Fuchsian structure in $\SymLeaf_{L(F)}$ is given by the coordinates
\begin{equation}\label{eq : coordinates for fuchsian locus}
(\sigma_1,\dots,\sigma_6,\tau_1,\tau_2)=\left(\sqrt\frac{\ell_\alpha\ell_\gamma}{\ell_\beta},\sqrt{\frac{\ell_\beta}{\ell_\alpha\ell_\gamma}},\sqrt{\frac{\ell_\alpha\ell_\beta}{\ell_\gamma}},\sqrt{\frac{\ell_\gamma}{\ell_\alpha\ell_\beta}},\sqrt{\frac{\ell_\beta\ell_\gamma}{\ell_\alpha}},\sqrt{\frac{\ell_\alpha}{\ell_\beta\ell_\gamma}},1,1\right).
\end{equation}
\subsection{Hamiltonian vector fields and flows}\label{sec : Hamiltonian vector fields and flows}
Following the discussion surrounding Equation \eqref{eq : image of cosymplectic form} regarding the symplectic leaves, we compute that for a point $p = (\sigma_1,\dots,\sigma_6,\tau_1,\tau_2)\in\widehat\DefSpace(\pants)$ and the symplectic leaf $\SymLeaf$ going through $p$:
\begin{equation}\label{eq : tangent space to symplectic leaves}
\T_p\SymLeaf = \left\langle \tau_1\del{\tau_1}-\tau_2\del{\tau_2},\sum_{i = 1}^6(-1)^{i+1}\sigma_i\del{\sigma_i} \right\rangle.
\end{equation}
Equivalently, by the fact the rank of the Poisson structure is constant (Remark \ref{rmk : level sets of casimirs is symplectic leaf}), we see that $\T_p\SymLeaf$ is exactly the kernel of the Jacobian matrix associated to the Casimir functions in Proposition \ref{lemma : casimir functions pair of pants} (see Equation \eqref{eq : kernel of Jacobian matrix}).\\

Recall from Definition \ref{def : Hamiltonians generating the symplectic leaves}, that the Hamiltonian functions are a family of functions generating the symplectic leaves.
\begin{proposition}\label{prop : Hamiltonian functions and vector fields}
	The Hamiltonian functions generating the symplectic leaves of $\widehat\DefSpace(\pants)$ are given by
	\begin{align*}
			\eruption &= \frac{\log\tau_1-\log\tau_2}{4},\\
			\hexagon &= -\frac{1}{12}\sum_{i = 1}^6(-1)^{i+1}\log\sigma_i.
	\end{align*}
	Their Hamiltonian vector fields are given by
	\begin{align*}
		\Hm_\eruption &= \sum_{i = 1}^6(-1)^{i+1}\sigma_i\del{\sigma_i},\quad\textnormal{and}\\
		\Hm_\hexagon &= \tau_1\del{\tau_1}-\tau_2\del{\tau_2}.
	\end{align*}
Moreover, their Poisson bracket is
\begin{equation}\label{eq : Poisson bracket between eruption and hexagon vector fields}
	\{\eruption,\hexagon\} = \frac{1}{2}
\end{equation}
In particular, the Hamiltonian flows of $\eruption$ and $\hexagon$ commute.
\end{proposition}
\begin{proof}
	To see that the functions $\eruption$ and $\hexagon$ generate the symplectic leaves, observe that according to Equation \eqref{eq : tangent space to symplectic leaves}, the vector fields $\Hm_\eruption$ and $\Hm_\hexagon$ generate the tangent space to the symplectic leaves. We begin by showing that the Hamiltonian vector fields $\Hm_\eruption$ and $\Hm_\hexagon$ are indeed the Hamiltonian vector fields of the functions $\eruption$ and $\hexagon$. To see that the Hamiltonian vector fields are those in the proposition, it is enough to check Equation \eqref{eq : definition of Hamiltonian vector field} on the coordinate functions. We compute that for any coordinate function $X_j$
	\begin{align*}
		\{X_j,\eruption\} &= \frac{1}{4}\left(-\{\log\tau_1,X_j\}+\{\log\tau_2,X_j\}\right)\\
		&=\frac{1}{4}\left(\frac{1}{\tau_2}d\tau_2(\Hm_{X_j})-\frac{1}{\tau_1}d\tau_1(\Hm_{X_j})\right)\\
		&=\frac{1}{4}\left(\frac{1}{\tau_2}\{\tau_2,X_j\}-\frac{1}{\tau_1}\{\tau_1,X_j\}\right).
		\end{align*}
		In particular, from the matrix \eqref{eq : eps ij matrix for Poisson structure} and the Poisson structure \eqref{eq : Poisson structure from FG}, we obtain that
		\[
		\{\sigma_i,\eruption\} = (-1)^{i+1}\sigma_i,\qquad \{\tau_i,\eruption\}=0.
		\]
		On the other hand,
		\[
		dX_j(\Hm_\eruption) = \sum_{i = 1}^6(-1)^{i+1}\sigma_i\,dX_j\left(\del{\sigma_i}\right)
		\]
		and we see from the above equations that indeed $\{X_j,\eruption\} = dX_j(\Hm_\eruption)$ as desired.\\
		
	Following a similar computation for the function $\hexagon$, we see that for any coordinate function $X_j$:
	\begin{align*}
		\{X_j,\hexagon\} = \frac{1}{12}\sum_{i = 1}^6\frac{(-1)^{i}}{\sigma_i} \{X_j,\sigma_i\}.
	\end{align*}	
		It follows that
		\[
		\{\sigma_i,\hexagon\} = 0,\qquad \{\tau_i,\hexagon\} = \tau_i.
		\]
		On the other hand,
		\[
		dX_j(\Hm_\hexagon) = \tau_1\,dX_j\left(\del{\tau_1}\right)-\tau_2\,dX_j\left(\del{\tau_2}\right)
		\]
		and we see from the above equations that indeed $\{X_j,\hexagon\} = dX_j(\Hm_\hexagon)$ for every coordinate function. These computations prove the first part of the proposition.\\
		
		To see that the flows commute, we compute that their Poisson bracket is given by
		\begin{equation*}
		\{\eruption,\hexagon\} = \frac{1}{48}\left(\sum_{i = 1}^6\frac{(-1)^i}{\tau_1\sigma_i}\{\tau_1,\sigma_i\}+\sum_{i = 1}^6\frac{(-1)^{i+1}}{\tau_2\sigma_i}\{\tau_2,\sigma_i\}\right) = \frac{1}{2}.
		\end{equation*}
		Since the Poisson bracket is constant, the corresponding Hamiltonian vector fields commute by Lemma \ref{lemma : Hamiltonian flows commute if Poisson bracket is constant}.
\end{proof}

The expression for the vector fields allows us to solve for the Hamiltonian flow itself.

\begin{corollary}
	The Hamiltonian flows of the Hamiltonian functions are given by
	\begin{align*}
	\HmFlow_\hexagon^t\colon (\sigma_1,\dots,\sigma_6,\tau_1,\tau_2) &\mapsto (\sigma_1,\dots,\sigma_6,e^t\tau_1,e^{-t}\tau_2)\\
	\HmFlow^t_\eruption\colon(\sigma_1,\dots,\sigma_6,\tau_1,\tau_2)&\mapsto (e^t\sigma_1,e^{-t}\sigma_2,e^t\sigma_3,e^{-t}\sigma_4,e^t\sigma_5,e^{-t}\sigma_6,\tau_1,\tau_2).
	\end{align*}
\end{corollary}
\begin{proof}
	This corollary is an immediate consequence of the form of the Hamiltonian vector fields in Proposition \ref{prop : Hamiltonian functions and vector fields}. That is, the above flows solve the ordinary differential equations defined by the corresponding vector fields.
\end{proof}

Similar flows already appear in \cite{WienhardZhang_Deforming,FlowsPGLVHitchin_SWZ}. Given their similarity, we adopt the same names for the flows.

\begin{definition}\cite{WienhardZhang_Deforming,FlowsPGLVHitchin_SWZ}
	The Hamiltonian flow of the function $\hexagon$ is called the \emph{eruption flow}, and the Hamiltonian flow of the function $\eruption$ is called the \emph{hexagon flow}.
\end{definition}

\begin{remark}\label{rmk : eruption and hexagon flows same name but actually different flows}
	The difference between the flows in \cite{WienhardZhang_Deforming,FlowsPGLVHitchin_SWZ} and the flows in this article, is that the invariants used by Fock and Goncharov in \cite{FockGoncharov_ConvexProjective} are different from the ones used in the above papers.
\end{remark}

These two flows give a symplectic trivialization of the two--dimensional symplectic leaves. In particular, it means that for any pairs $q_1,q_2\in\widehat\DefSpace(\pants)$, there are unique numbers $s,t\in\R$ such that 
\begin{equation}\label{eq : q2 = flow eruption hexagon q1}
q_2 = \HmFlow^t_\eruption\circ\HmFlow^s_\hexagon(q_1) = \HmFlow^s_\hexagon\circ\HmFlow^t_\eruption(q_1).
\end{equation}

This fact motivates the following.

\begin{definition}\label{def : mixed flow}
	Let $a\in\R$. Any of the two flows
	\[
\HmFlow^t_\eruption\circ\HmFlow^{at}_\hexagon\quad \textnormal{or}\quad \HmFlow^{at}_\eruption\circ\HmFlow^{t}_\hexagon
	\]
	is called a \emph{mixed flow}. A mixed flow defined by $a\in\R$ will be written as $\Psi^t_a$.
\end{definition}
The following is a direct consequence of Equation \eqref{eq : q2 = flow eruption hexagon q1}, and is the required analogue of Fact \ref{fact 2 : earthquake theorem} in the introduction, allowing to connect any pair of points in a symplectic leaf via a mixed flow.
\begin{lemma}\label{lemma : any two points are connected by a mixed flow}
	Let $L\in\R^6_{>0}$ define a symplectic leaf $\SymLeaf_L\subset \widehat\DefSpace(\pants)$. For any pair $q_1,q_2\in \SymLeaf_L$, there is a constant $a\in\R$ defining a mixed flow and a time $t\in\R$ such that $q_2 = \Psi^t_a(q_1)$.
\end{lemma}

\subsection{The symplectic form on a symplectic leaf} Here we compute the symplectic form on the leaves explicitly first by using the functions $\eruption$ and $\hexagon$. We then use the parametrization of the symplectic leaves in Lemma \ref{cor : parametrization of symplectic leaf as level sets} to give a more useful expression for the symplectic form.

\begin{lemma}\label{cor : symplectic form in coordinates}
	Let $\SymLeaf$ be a symplectic leaf of $\widehat\DefSpace(\pants)$. The symplectic form $\omega_\SymLeaf$ is given by
	\[
	\omega_\SymLeaf = 2\,d \eruption\wedge d\hexagon.
	\]
	In terms of the Fock--Goncharov coordinates, the symplectic form reads
	\begin{equation}\label{eq : symplectic form in terms of fock goncharov coordinates}
	\omega_\SymLeaf = \frac{1}{24}\sum_{i = 1}^6(-1)^{i}\left(\frac{1}{\tau_1\sigma_i}d\tau_1\wedge d\sigma_i-\frac{1}{\tau_2\sigma_i}d\tau_2\wedge d\sigma_i\right).
	\end{equation}
\end{lemma}
\begin{remark}
	The form of the symplectic structure here is analogous to the Darboux system given by Sun and Zhang in \cite[Corollary 7.11]{DarbouxOnHitchin_SunZhang}. Recall from Remark \ref{rmk : eruption and hexagon flows same name but actually different flows} that the difference lies in the description of the invariants associated to flags.
\end{remark}
\begin{proof}
	By Proposition \ref{prop : Hamiltonian functions and vector fields}, we only need to compute $\omega_\SymLeaf(\Hm_\eruption,\Hm_\hexagon)$. By Equation \eqref{eq : Poisson bracket between eruption and hexagon vector fields} we have that
	\[
	\omega_\SymLeaf(\Hm_\eruption,\Hm_\hexagon) = \{\eruption,\hexagon\} = \frac{1}{2}.
	\]
	On the other hand, 
	\[
	2 d\eruption\wedge d\hexagon (\Hm_\eruption,\Hm_\hexagon) = 2 \{\eruption,\hexagon\}^2 = \frac{1}{2}
	\]
	as desired. Writing the symplectic form in terms of Fock--Goncharov coordinates is a computation using the expression of the functions in Proposition \ref{prop : Hamiltonian functions and vector fields}.
\end{proof}

This expression allows us to compute the Hamiltonian vector fields and Hamiltonian flows of all the coordinate functions.

\begin{corollary}
	The Hamiltonian vector field of the coordinate functions $\sigma_1,\dots,\sigma_6,\tau_1,\tau_2$ is given by
	\begin{align*}
		\Hm_{\sigma_i} = (-1)^{i+1}\sigma_i\Hm_{\hexagon},\quad \Hm_{\tau_1} = \tau_1\Hm_\eruption,\quad\textnormal{and }\Hm_{\tau_2} = -\tau_2\Hm_{\eruption}.
	\end{align*}
	In particular, their corresponding Hamiltonian flows are given by
	\begin{align*}
		\HmFlow^t_{\sigma_i}(\sigma_1,\dots,\sigma_6,\tau_1,\tau_2)&\mapsto \left(\sigma_1,\dots,\sigma_6,e^{(-1)^it\sigma_i}\tau_1,e^{(-1)^{i+1}t\sigma_i}\tau_2\right),\\
		\HmFlow^t_{\tau_1}(\sigma_1,\dots,\sigma_6,\tau_1,\tau_2)&\mapsto\left(e^{t\tau_1}\sigma_1,e^{-t\tau_1}\sigma_2,e^{t\tau_1}\sigma_3,e^{-t\tau_1}\sigma_4,e^{t\tau_1}\sigma_5,e^{-t\tau_1}\sigma_6,\tau_1,\tau_2\right),\\
		\HmFlow^t_{\tau_2}(\sigma_1,\dots,\sigma_6,\tau_1,\tau_2)&\mapsto\left(e^{-t\tau_2}\sigma_1,e^{t\tau_2}\sigma_2,e^{-t\tau_2}\sigma_3,e^{t\tau_2}\sigma_4,e^{-t\tau_2}\sigma_5,e^{t\tau_2}\sigma_6,\tau_1,\tau_2\right).
	\end{align*}
\end{corollary}
\begin{proof}
	The computation of the Hamiltonian vector fields is a direct application of the expression of $\omega_\SymLeaf$ in Fock--Goncharov coordinates in Equation \eqref{eq : symplectic form in terms of fock goncharov coordinates}. Similarly, the solutions to the ordinary differential equations arising from the vector fields are seen to be given by the flows in the corollary.
\end{proof}

Restricted to symplectic leaves determined by a length vector $L\in\R^6_{>0}$, we can use the coordinates $\sigma_1$ and $\tau_1$ as in Lemma \ref{cor : parametrization of symplectic leaf as level sets}. In the following corollary, we describe the symplectic form in terms of the coordinates $\sigma_1$ and $\tau_1$.

\begin{lemma}
	Let $L = (\ell_{\alpha,1},\ell_{\alpha,2},\ell_{\beta,1},\ell_{\beta,2},\ell_{\gamma,1},\ell_{\gamma,2})\in\R^6_{>0}$. The symplectic form of the symplectic leaf $\SymLeaf_L$ parameterized as in Lemma \ref{cor : parametrization of symplectic leaf as level sets} is given by
	\[
	\omega_{\SymLeaf_L} = \frac{1}{2\sigma_1\tau_1}d\sigma_1\wedge d\tau_1.
	\]
\end{lemma}

\begin{proof}
	To prove the formula, we compute that
\begin{align*}
	d\sigma_{2,L}&= -\frac{1}{\sigma_1^2}\frac{(\ell_{\alpha,1}\ell_{\beta,2}\ell_{\gamma,1})^{2/3}}{(\ell_{\alpha,2}\ell_{\beta,1}\ell_{\gamma,2})^{1/3}}	d\sigma_1,\\
	d\sigma_{3,L}&= \frac{(\ell_{\beta,1}\ell_{\gamma,2})^{2/3}}{(\ell_{\alpha,1}\ell_{\alpha,2}\ell_{\beta,2}\ell_{\gamma,1})^{1/3}}d\sigma_1,\\
	d\sigma_{4,L}&=-\frac{\ell_{\alpha,1}}{\sigma_1^2}d\sigma_1,\\
	d\sigma_{5,L}&=\frac{(\ell_{\alpha,2}\ell_{\beta,1}\ell_{\gamma,1}\ell_{\gamma,2})^{1/3}}{(\ell_{\alpha,1}\ell_{\beta,2})^{2/3}}d\sigma_1,\\
	d\sigma_{6,L}&=-\frac{1}{\sigma_1^2}\frac{(\ell_{\alpha,1}\ell_{\alpha,2}\ell_{\beta,1}\ell_{\beta,2}\ell_{\gamma,1})^{1/3}}{\ell_{\gamma,2}^{2/3}}d\sigma_1,\\
	d\tau_{2,L}&=-\frac{(\ell_{\alpha,1}\ell_{\alpha,2}\ell_{\beta,1}\ell_{\beta,2}\ell_{\gamma,1}\ell_{\gamma,2})^{1/3}}{\tau_1^2}d\tau_1.
	\end{align*}
	Replacing $\sigma_{i,L},\tau_{2,L}$ and their derivatives in Equation \eqref{eq : symplectic form in terms of fock goncharov coordinates}, we observe that
	\[
	\omega_{\SymLeaf_L} = \frac{1}{24}\sum_{i = 1}^6\frac{-2}{\tau_1\sigma_1}d\tau_1\wedge d\sigma_1 = \frac{1}{2\sigma_1\tau_1}d\sigma_1\wedge d\tau_1
	\]
	as claimed.
\end{proof}

This form of the symplectic structure gives us a straight forward computation of the Hamiltonian vector field of a given function. Namely, let $L\in\R^6_{>0}$ define a symplectic leaf $\SymLeaf_L$. Consider a function $\phi\colon \SymLeaf_L\cong\R^2_{>0}\to\R$ given in coordinates $(\sigma_1,\tau_1)$ using the parameterization of Lemma \ref{cor : parametrization of symplectic leaf as level sets}. Then a linear algebra computation shows that the Hamiltonian vector field of $\phi$ is given by
\begin{equation}\label{eq : Hamiltonian vector field in symplectic leaf in coordinates sigma1 tau1}
	\Hm_\phi = 2\sigma_1\tau_1\left(-\frac{\partial \phi}{\partial \tau_1}\cdot \del{\sigma_1}+\frac{\partial\phi}{\partial\sigma_1}\cdot\del{\tau_1}\right).
\end{equation}
Writing down the associated differential equation, we have that a path $(\sigma_1,\tau_1)\colon\R\to\SymLeaf_L$ is a flow line of the Hamiltonian vector field of $\phi$ if
\begin{equation}\label{eq : differential equation associated to the Hamiltonian vector field}
	\begin{dcases}\dot\sigma_1(t) &= -2\sigma_1(t)\tau_1(t)\frac{\partial\phi}{\partial\tau_1}(\sigma_1(t),\tau_1(t)),\\
	\dot\tau_1(t) &= 2\sigma_1(t)\tau_1(t)\frac{\partial\phi}{\partial\sigma_1}(\sigma_1(t),\tau_1(t)).
	\end{dcases}
\end{equation}
\section{Trace of the figure eight curve}\label{sec : trace of figure eight curve}
We are now ready to address Theorem \ref{thm Intro : main theorem}. Recall that any closed curve $c\in\pi_1(S)$ defines the function
\begin{align*}
\trace_c\colon \widehat\DefSpace(\pants)&\to\R \\
[f,\Sigma,\nu]&\mapsto \trace(\rechol([f,\Sigma,\nu])(c)).
\end{align*}

In this section, we focus on the figure eight curve, i.e. a curve with a single self--intersection. The identification $\widehat{\DefSpace}(\pants)\cong\R^8_{>0}$ together with the reconstruction of representations from coordinates (recall Section \ref{sec : reconstructing the represenatation}) allows us to write traces of curves as rational functions in the coordinates. \\

The main goal of this section is to provide a proof of Theorem \ref{thm Intro : main theorem}, which we restate here.
\begin{theorem}\label{thm : periodicity of figure 8 in general}
	Let $\figeight=\alpha\gamma^{-1}$ be the figure eight curve on $\pants$ and let $L\in\R^6_{>0}$ define a symplectic leaf $\SymLeaf_L\subset\widehat\DefSpace(\pants)$. Then the function $\trace_\figeight\big|_{\SymLeaf_L}\colon\SymLeaf_L\to\R$ attains a unique minimum. Moreover, every orbit of the Hamiltonian flow of $\trace_\figeight\big|_{\SymLeaf_L}$ is periodic and there is a unique fixed point.
\end{theorem}
As explained in the introduction, we need to prove Proposition \ref{prop : Intro properness figure 8} (analogous to Fact \ref{fact 1 : properness}) on properness of the trace function, and Theorem \ref{thm : Intro convexity} (analogous to Fact \ref{fact : convexity}) on convexity. We begin by presenting the expression for the trace of the figure eight curve in a given symplectic leaf.

\begin{lemma}\label{lemma : trace figure eight in coordinates}
	Let $\figeight = \alpha\gamma^{-1}$ and $L = (\ell_{\alpha,1},\ell_{\alpha,2},\ell_{\beta,1},\ell_{\beta,2},\ell_{\gamma,1},\ell_{\gamma,2})\in\R^6_{>0}$ be a length vector defining a symplectic leaf $\SymLeaf_L$. Then
\begin{gather*}
\trace_\figeight\big|_{\SymLeaf_L}\colon \SymLeaf_L\to\R\\
(\sigma_1,\tau_1)\mapsto \frac{1}{\sigma _1 \tau _1 \ell _{\alpha ,1}^{4/3} \sqrt[3]{\ell _{\beta ,1}} \ell
   _{\beta ,2} \left(\ell _{\alpha ,2} \ell _{\gamma ,1} \ell _{\gamma
   ,2}\right){}^{2/3}}\cdot\\
\big(\sigma _1^3 \tau _1^2 \ell _{\gamma ,2} \sqrt[3]{\ell _{\alpha ,2} \ell
   _{\beta ,1}}+\ell _{\alpha ,1}^{5/3} \left(\ell _{\beta ,2} \ell _{\gamma
   ,1} \ell _{\gamma ,2}\right){}^{2/3} \left(\left(\sigma _1+\tau
   _1+1\right) \ell _{\alpha ,2} \ell _{\beta ,2}+\sigma _1 \tau _1^2\right)+\\
   +\sigma _1^2 \tau _1 \left(\ell _{\alpha ,2} \ell _{\beta ,1}\right){}^{2/3}
   \sqrt[3]{\ell _{\alpha ,1} \ell _{\beta ,2} \ell _{\gamma ,1} \ell
   _{\gamma ,2}} \left(2 \sigma _1 \ell _{\gamma ,2}+\ell _{\gamma ,2}+\tau
   _1\right)+\\
   +\sigma _1 \ell _{\alpha ,1}^{4/3} \left(\ell _{\alpha ,2} \ell _{\beta
   ,1}\right){}^{2/3} \sqrt[3]{\ell _{\beta ,2} \ell _{\gamma ,1} \ell
   _{\gamma ,2}} \left(\left(\sigma _1+1\right) \ell _{\beta ,2} \ell
   _{\gamma ,2}+\tau _1 \left(\sigma _1 \ell _{\gamma ,2}+\tau
   _1\right)\right)+\\
   +\tau _1 \ell _{\alpha ,1}^2 \ell _{\beta ,2} \sqrt[3]{\ell _{\alpha ,2} \ell
   _{\beta ,1}} \left(\ell _{\gamma ,1} \left(\sigma _1 \ell _{\gamma
   ,2}+\sigma _1+\tau _1+1\right)+\sigma _1 \ell _{\gamma ,2}\right)\\
   +\sigma _1^2 \left(\ell _{\alpha ,1} \ell _{\beta ,2} \ell _{\gamma ,1} \ell
   _{\gamma ,2}\right){}^{2/3} \left(\ell _{\alpha ,2} \left(\ell _{\beta
   ,1} \left(\sigma _1 \ell _{\gamma ,2}+\ell _{\gamma ,2}+\tau
   _1\right)+\tau _1\right)+\tau _1^2\right)\\
+\sigma _1 \tau _1 \ell _{\alpha ,1} \sqrt[3]{\ell _{\alpha ,2} \ell _{\beta
   ,1}} (\ell _{\beta ,2} \ell _{\gamma ,1} (\sigma _1 \ell
   _{\gamma ,2}+\ell _{\gamma ,2}+\tau _1+1)+(\sigma _1+1)
   \ell _{\beta ,2} \ell _{\gamma ,2}+\sigma _1 \tau _1 \ell _{\gamma
   ,2}+\\
   +\ell _{\alpha ,2} \ell _{\beta ,2} \left(\ell _{\gamma ,1}
   \left(\left(\sigma _1+\tau _1+1\right) \ell _{\gamma ,2}+\tau
   _1\right)+\tau _1 \ell _{\gamma ,2}\right))\big).
\end{gather*}
In the unipotent locus, i.e. when $L = (1,1,1,1,1,1)$, we have that
\begin{align*}
\trace_\figeight\colon\unipotentLocus&\to\R\\
(\sigma_1,\tau_1) &\mapsto \frac{\sigma _1^3 \left(\tau _1+1\right){}^2+3 \sigma _1^2 \left(\tau
   _1+1\right){}^2+3 \sigma _1 \left(\tau _1^2+3 \tau _1+1\right)+\left(\tau
   _1+1\right){}^2}{\sigma _1 \tau _1}
\end{align*}
\end{lemma}
\begin{proof}
	The proof of this fact is a computation, found in Section 4 of the Mathematica code. The function \texttt{traceFigure8} will give the above output with the length vectors as input. For the unipotent locus, replace the length vector with ones.
\end{proof}

\subsection{Properness of the trace function}
We can now prove Proposition \ref{prop : Intro properness figure 8}, which we restate here for convenience.
\begin{proposition}\label{prop : trace of figure 8 is proper}
	Let $\figeight = \alpha\gamma^{-1}$ and $L = (\ell_{\alpha,1},\ell_{\alpha,2},\ell_{\beta,1},\ell_{\beta,2},\ell_{\gamma,1},\ell_{\gamma,2})\in\R^6_{>0}$ be a length vector defining a symplectic leaf $\SymLeaf_L$. Then the function $\trace_\figeight\colon\SymLeaf_L\to\R$ is proper. In particular, it realizes a minimum in $\SymLeaf_L$.
\end{proposition}

\begin{proof}
	By the parameterization in Lemma \ref{cor : parametrization of symplectic leaf as level sets}, the trace of $\figeight$ defines the rational function in Lemma \ref{lemma : trace figure eight in coordinates} as a map from $\R^2_{>0}$ to $\R$. Let $(\sigma_{1,n},\tau_{1,n})_{n\in\N}$ be a sequence such that as $n$ goes to $\infty$, $(\sigma_{1,n},\tau_{1,n})$ goes to a tuple in $\{(\infty,\infty), (0,\infty),(\infty,0),(x,0),(0,y)\st x,y\geq 0\}$, i.e. a sequence escaping every compact set in $\R^2_{>0}$. We need to show that in any of these cases, $\trace_\figeight(\sigma_{1,n},\tau_{1,n})\to\infty$ as $n\to\infty$. Notice that all the variables are positive and that all the signs on the monomials are positive as well. Hence, it is enough to find terms in the expression in Lemma \ref{lemma : trace figure eight in coordinates} that diverges along any of the above sequences.
	\begin{enumerate}
		\item Assume $(\sigma_{1,n},\tau_{1,n})\to(\infty,\infty)$, then the term 
		\[
		 \frac{\sigma _{1,n}^3 \tau _{1,n}^2 \ell _{\gamma ,2} \sqrt[3]{\ell _{\alpha ,2} \ell
   _{\beta ,1}}}{\sigma _{1,n} \tau _{1,n} \ell _{\alpha ,1}^{4/3} \sqrt[3]{\ell _{\beta ,1}} \ell
   _{\beta ,2} \left(\ell _{\alpha ,2} \ell _{\gamma ,1} \ell _{\gamma
   ,2}\right){}^{2/3}}\xrightarrow{n\to\infty}\infty.
		\]
		\item Assume $(\sigma_{1,n},\tau_{1,n})\to(\infty,0)$, then the term 
		\[
		\frac{2\sigma_{1,n}^3\tau_{1,n}(\ell_{\alpha,2}\ell_{\beta,1})^{2/3}\sqrt[3]{\ell _{\alpha ,1} \ell _{\beta ,2} \ell _{\gamma ,1} \ell
   _{\gamma ,2}}\ell_{\gamma,2}}{\sigma _{1,n} \tau _{1,n} \ell _{\alpha ,1}^{4/3} \sqrt[3]{\ell _{\beta ,1}} \ell
   _{\beta ,2} \left(\ell _{\alpha ,2} \ell _{\gamma ,1} \ell _{\gamma
   ,2}\right){}^{2/3}}\xrightarrow{n\to\infty}\infty.
		\]
		\item Assume $(\sigma_{1,n},\tau_{1,n})\to(0,\infty)$, then the term 
		\[
		\frac{\ell _{\alpha ,1}^{5/3} \left(\ell _{\beta ,2} \ell _{\gamma
   ,1} \ell _{\gamma ,2}\right){}^{2/3}\sigma_{1,n}\tau_{1,n}^2}{\sigma _{1,n} \tau _{1,n} \ell _{\alpha ,1}^{4/3} \sqrt[3]{\ell _{\beta ,1}} \ell
   _{\beta ,2} \left(\ell _{\alpha ,2} \ell _{\gamma ,1} \ell _{\gamma
   ,2}\right){}^{2/3}}\xrightarrow{n\to\infty}\infty.
		\]
		\item Assume $(\sigma_{1,n},\tau_{1,n})\to(x,0)$ or $(0,y)$ for $x,y\geq 0$, then the term
		\[
		\frac{\ell_{\alpha,1}^{5/3}(\ell_{\beta,2}\ell_{\gamma,1}\ell_{\gamma,2})^{2/3}\ell_{\alpha,2}\ell_{\beta,2}}{\sigma _{1,n} \tau _{1,n} \ell _{\alpha ,1}^{4/3} \sqrt[3]{\ell _{\beta ,1}} \ell
   _{\beta ,2} \left(\ell _{\alpha ,2} \ell _{\gamma ,1} \ell _{\gamma
   ,2}\right){}^{2/3}}\xrightarrow{n\to\infty}\infty.
		\]
	\end{enumerate}
	
This finishes the proof of properness of the function $\trace_\figeight$ on the symplectic leaves. To see that the function realizes a minimum, note that the function is positive.
\end{proof}

\subsection{Convexity of $\trace_\figeight$}
The next step is to prove Theorem \ref{thm : Intro convexity}. Here we once again use more Mathematica for the computations. Since we have explicit expressions for the Hamiltonian flows associated to coordinate functions, as well as the mixed flows, we can use them to compute second derivatives. That is, given a function $\phi\colon\SymLeaf_L\to\R$ and a flow $\HmFlow^t\colon\SymLeaf_L\to\SymLeaf_L$, we compute
\[
t\mapsto \frac{\partial^2}{\partial t^2}\phi\left(\HmFlow^t(q)\right)
\]
for any $q\in\SymLeaf_L$ and check whether it is a strictly positive function.
\begin{theorem}\label{thm : convexity of the trace function along mixed flows}
	Let $L\in\R^6_{>0}$ and $\figeight = \alpha\gamma^{-1}$ be a figure eight curve. Then the trace function $\trace_\figeight\big|_{\SymLeaf_L}\colon\SymLeaf_L\to\R$ is strictly convex along any mixed flow.
\end{theorem}

\begin{proof}
	Let $L = (\ell_{\alpha,1},\ell_{\alpha,2},\ell_{\beta,1},\ell_{\beta,2},\ell_{\gamma,1},\ell_{\gamma,2})\in\R^6_{>0}$ be a length vector and let $q = (\sigma_1,\tau_1)\in\SymLeaf_L$. Let $a\in\R$. We need to check the two different types of mixed flows from Definition \ref{def : mixed flow}. We begin with the case when the mixed flow defined by $a$ is given by $\Psi^t_a = \HmFlow_\eruption^{at}\circ\HmFlow^t_\hexagon$. In Section 6 of the Mathematica code, set \texttt{testf} equal to \texttt{traceFig8}. Then inputting the function \texttt{testfMix1} to \texttt{SecondDerFlow}, we obtain that
	
	\begin{gather*}
		\frac{\partial^2}{\partial t^2}\trace_\figeight\left(\Psi_a^t(q)\right) = \frac{e^{-t(a+1)}}{\sigma _1 \tau _1 \ell _{\beta ,2} \sqrt[3]{\ell _{\alpha ,1}^4
   \ell _{\alpha ,2}^2 \ell _{\beta ,1} \ell _{\gamma ,1}^2 \ell _{\gamma
   ,2}^2}}\cdot\\
   \bigg( e^t \tau _1 \ell _{\alpha ,1}^2 \ell _{\beta ,2} \ell _{\gamma ,1} \left(e^t\tau_1(a-1)^2+ a^2\right)
   \sqrt[3]{\ell _{\alpha ,2} \ell _{\beta ,1}}+\sigma _1 \tau _1^2 e^{(a+2)
   t} \ell _{\alpha ,1}^{5/3} \left(\ell _{\beta ,2} \ell _{\gamma ,1} \ell
   _{\gamma ,2}\right){}^{2/3}+\\
   \sigma_1\cdot\bigg((2 a+1)^2 \sigma _1^2 \tau _1^2 e^{(3 a+2) t} \ell _{\gamma ,2}
   \sqrt[3]{\ell _{\alpha ,2} \ell _{\beta ,1}}+(a-1)^2 \sigma _1 e^{2 a t}
   \ell _{\alpha ,2} \ell _{\beta ,1} \ell _{\gamma ,2}^{5/3} \left(\ell
   _{\alpha ,1} \ell _{\beta ,2} \ell _{\gamma ,1}\right){}^{2/3}+\\
   (1-2 a)^2 \sigma _1^2 e^{3 a t} \ell _{\alpha ,2} \ell _{\beta ,1} \ell
   _{\gamma ,2}^{5/3} \left(\ell _{\alpha ,1} \ell _{\beta ,2} \ell _{\gamma
   ,1}\right){}^{2/3}+\tau _1^2 e^{(a+2) t} \sqrt[3]{\ell _{\alpha ,1}^4
   \ell _{\alpha ,2}^2 \ell _{\beta ,1}^2 \ell _{\beta ,2} \ell _{\gamma ,1}
   \ell _{\gamma ,2}} 
   \,+\\
8 a^2 \sigma _1^2 \tau _1 e^{3 a t+t} \sqrt[3]{\ell _{\alpha ,1} \ell
   _{\alpha ,2}^2 \ell _{\beta ,1}^2 \ell _{\beta ,2} \ell _{\gamma ,1} \ell
   _{\gamma ,2}^4}+\\
   (a+1)^2 \sigma _1 \tau _1^2 e^{2 (a+1) t} \sqrt[3]{\ell _{\alpha ,1} \ell
   _{\beta ,2} \ell _{\gamma ,1} \ell _{\gamma ,2}} \left(\sqrt[3]{\ell
   _{\alpha ,1} \ell _{\beta ,2} \ell _{\gamma ,1} \ell _{\gamma
   ,2}}+\left(\ell _{\alpha ,2} \ell _{\beta ,1}\right){}^{2/3}\right)+\\
   a^2 \sigma _1 \tau _1 e^{2 a t+t} \bigg((\ell _{\alpha ,2} \left(\ell _{\beta
   ,1}+1\right) \left(\ell _{\alpha ,1} \ell _{\beta ,2} \ell _{\gamma ,1}
   \ell _{\gamma ,2}\right){}^{2/3}+\\\sqrt[3]{\ell _{\alpha ,1} \ell _{\alpha
   ,2}^2 \ell _{\beta ,1}^2 \ell _{\beta ,2} \ell _{\gamma ,1} \ell _{\gamma
   ,2}^4}+\sqrt[3]{\ell _{\alpha ,1}^4 \ell _{\alpha ,2}^2 \ell _{\beta
   ,1}^2 \ell _{\beta ,2} \ell _{\gamma ,1} \ell _{\gamma ,2}^4}\bigg)\bigg)+\\
   \ell_{\alpha,1}\bigg(\ell _{\alpha ,2} \ell _{\beta ,2} \left(a^2 e^t \tau _1+\sigma _1 e^{a
   t}+(a+1)^2\right) \left(\ell _{\alpha ,1} \ell _{\beta ,2} \ell _{\gamma
   ,1} \ell _{\gamma ,2}\right){}^{2/3}+\\
   \sigma_1\bigg((a+1)^2 \sigma _1 \tau _1^2 e^{2 (a+1) t} \ell _{\gamma ,2} \sqrt[3]{\ell
   _{\alpha ,2} \ell _{\beta ,1}}+\tau _1^2 e^{(a+2) t} \ell _{\beta ,2}
   \ell _{\gamma ,1} \sqrt[3]{\ell _{\alpha ,2} \ell _{\beta ,1}}+\\
   a^2 \sigma _1 \tau _1 e^{2 a t+t} \ell _{\beta ,2} \left(\ell _{\gamma
   ,1}+1\right) \ell _{\gamma ,2} \sqrt[3]{\ell _{\alpha ,2} \ell _{\beta
   ,1}}+\\
   e^{a t} \ell _{\beta ,2} \ell _{\gamma ,2} \left(\ell _{\gamma ,1}
   \sqrt[3]{\ell _{\alpha ,2}^4 \ell _{\beta ,1}}+\sqrt[3]{\ell _{\alpha ,1}
   \ell _{\alpha ,2}^2 \ell _{\beta ,1}^2 \ell _{\beta ,2} \ell _{\gamma ,1}
   \ell _{\gamma ,2}}\right)+\\
   (a-1)^2 \sigma _1 e^{2 a t} \ell _{\beta ,2} \ell _{\gamma ,2} \left(\ell
   _{\gamma ,1} \sqrt[3]{\ell _{\alpha ,2}^4 \ell _{\beta ,1}}+\sqrt[3]{\ell
   _{\alpha ,1} \ell _{\alpha ,2}^2 \ell _{\beta ,1}^2 \ell _{\beta ,2} \ell
   _{\gamma ,1} \ell _{\gamma ,2}}\right)\bigg)\bigg)\bigg).
	\end{gather*}
	Examining the terms, we see that $a$ always appears either within a square, or in an exponential. Moreover, $t$ also always appears only in an exponential. All the lengths are positive, and so are the coordinates $\sigma_1$ and $\tau_1$. Thus, the second derivative is always strictly larger than zero, and hence the function $t\mapsto \trace_\delta\big|_{\SymLeaf_L}(\Psi^t_a(q))$ is strictly convex for any $q\in \SymLeaf_L$.\\
	
	The case for the other mixed flow is similar. Consider now the mixed flow $\Psi^t_a = \HmFlow^t_\eruption\circ\HmFlow^{at}_\hexagon$ and $q = (\sigma_1,\tau_1)\in \SymLeaf_L$. In Section 6 of the Mathematica code, set \texttt{testf} equal to \texttt{traceFig8}. Then inputting the function \texttt{testfMix2} to \texttt{SecondDerFlow}, we obtain that
	\begin{gather*}
		\frac{\partial^2}{\partial t^2}\trace_\figeight\left(\Psi_a^t(q)\right)=\frac{e^{-t(a+1)}}{\sigma _1 \tau _1 \ell _{\beta ,2} \sqrt[3]{\ell
   _{\alpha ,1}^4 \ell _{\alpha ,2}^2 \ell _{\beta ,1} \ell _{\gamma ,1}^2
   \ell _{\gamma ,2}^2}}\cdot\\
   \bigg(a^2 \sigma _1 \tau _1^2 e^{2 a t+t} \ell _{\alpha ,1}^{5/3} \left(\ell
   _{\beta ,2} \ell _{\gamma ,1} \ell _{\gamma ,2}\right){}^{2/3}+\tau _1
   e^{a t} \ell _{\alpha ,1}^2 \ell _{\beta ,2} \ell _{\gamma ,1}
   \left((a-1)^2 \tau _1 e^{a t}+1\right) \sqrt[3]{\ell _{\alpha ,2} \ell
   _{\beta ,1}}+\\
   \sigma_1\bigg((a+2)^2 \sigma _1^2 \tau _1^2 e^{(2 a+3) t} \ell _{\gamma ,2} \sqrt[3]{\ell
   _{\alpha ,2} \ell _{\beta ,1}}+(a-1)^2 \sigma _1 e^{2 t} \ell _{\alpha
   ,2} \ell _{\beta ,1} \ell _{\gamma ,2}^{5/3} \left(\ell _{\alpha ,1} \ell
   _{\beta ,2} \ell _{\gamma ,1}\right){}^{2/3}+\\
   a^2 \tau _1^2 e^{2 a t+t} \sqrt[3]{\ell _{\alpha ,1}^4 \ell _{\alpha ,2}^2
   \ell _{\beta ,1}^2 \ell _{\beta ,2} \ell _{\gamma ,1} \ell _{\gamma
   ,2}}+(a-2)^2 \sigma _1^2 e^{3 t} \ell _{\alpha ,2} \ell _{\beta ,1} \ell
   _{\gamma ,2}^{5/3} \left(\ell _{\alpha ,1} \ell _{\beta ,2} \ell _{\gamma
   ,1}\right){}^{2/3}+\\
   8 \sigma _1^2 \tau _1 e^{(3+a) t} \sqrt[3]{\ell _{\alpha ,1} \ell _{\alpha
   ,2}^2 \ell _{\beta ,1}^2 \ell _{\beta ,2} \ell _{\gamma ,1} \ell _{\gamma
   ,2}^4}+\\
   (a+1)^2 \sigma _1 \tau _1^2 e^{2 (a+1) t} \sqrt[3]{\ell _{\alpha ,1} \ell
   _{\beta ,2} \ell _{\gamma ,1} \ell _{\gamma ,2}} \left(\sqrt[3]{\ell
   _{\alpha ,1} \ell _{\beta ,2} \ell _{\gamma ,1} \ell _{\gamma
   ,2}}+\left(\ell _{\alpha ,2} \ell _{\beta ,1}\right){}^{2/3}\right)+\\
   \sigma _1 \tau _1 e^{(a+2) t} \left(\ell _{\alpha ,2} \left(\ell _{\beta
   ,1}+1\right) \bigg((\ell _{\alpha ,1} \ell _{\beta ,2} \ell _{\gamma ,1}
   \ell _{\gamma ,2}\right){}^{2/3}\\
   +\sqrt[3]{\ell _{\alpha ,1} \ell _{\alpha
   ,2}^2 \ell _{\beta ,1}^2 \ell _{\beta ,2} \ell _{\gamma ,1} \ell _{\gamma
   ,2}^4}+\sqrt[3]{\ell _{\alpha ,1}^4 \ell _{\alpha ,2}^2 \ell _{\beta
   ,1}^2 \ell _{\beta ,2} \ell _{\gamma ,1} \ell _{\gamma ,2}^4}\bigg)\bigg)+\\
   \ell_{\alpha,1}\bigg(\ell _{\alpha ,2} \ell _{\beta ,2} \left(a^2\sigma_1e^t+\tau_1e^{at}+(a+1)^2\right) \left(\ell _{\alpha ,1} \ell
   _{\beta ,2} \ell _{\gamma ,1} \ell _{\gamma ,2}\right){}^{2/3}+\\
   \sigma_1\bigg(a^2 \tau _1^2 e^{2 a t+t} \ell _{\beta ,2} \ell _{\gamma ,1} \sqrt[3]{\ell
   _{\alpha ,2} \ell _{\beta ,1}}+(a+1)^2 \sigma _1 \tau _1^2 e^{2 (a+1) t}
   \ell _{\gamma ,2} \sqrt[3]{\ell _{\alpha ,2} \ell _{\beta ,1}}+\\
   \sigma _1 \tau _1 e^{(a+2) t} \ell _{\beta ,2} \left(\ell _{\gamma
   ,1}+1\right) \ell _{\gamma ,2} \sqrt[3]{\ell _{\alpha ,2} \ell _{\beta
   ,1}}+\\
   a^2 e^t \ell _{\beta ,2} \ell _{\gamma ,2} \left(\ell _{\gamma ,1}
   \sqrt[3]{\ell _{\alpha ,2}^4 \ell _{\beta ,1}}+\sqrt[3]{\ell _{\alpha ,1}
   \ell _{\alpha ,2}^2 \ell _{\beta ,1}^2 \ell _{\beta ,2} \ell _{\gamma ,1}
   \ell _{\gamma ,2}}\right)+\\
   (a-1)^2 \sigma _1 e^{2 t} \ell _{\beta ,2} \ell _{\gamma ,2} \left(\ell
   _{\gamma ,1} \sqrt[3]{\ell _{\alpha ,2}^4 \ell _{\beta ,1}}+\sqrt[3]{\ell
   _{\alpha ,1} \ell _{\alpha ,2}^2 \ell _{\beta ,1}^2 \ell _{\beta ,2} \ell
   _{\gamma ,1} \ell _{\gamma ,2}}\right)\bigg)\bigg)\bigg).
	\end{gather*}
	Once again, $t$ appears only as an exponential, $a$ appears either as an exponential or within a square, and all of the other variables are positive. Hence, $t\mapsto \trace_\delta\big|_{\SymLeaf_L}(\Psi^t_a(q))$ is strictly convex for any $q\in \SymLeaf_L$.
\end{proof}

With this strict convexity, we immediately obtain the following.

\begin{corollary}\label{cor : unique fixed point}
Let $L\in\R^6_{>0}$ and $\figeight = \alpha\gamma^{-1}$ be a figure eight curve. Then the trace function $\trace_\figeight\big|_{\SymLeaf_L}\colon\SymLeaf_L\to\R$ has a unique critical point, corresponding to the unique minimum.
\end{corollary}
\subsection{Proof of Theorem \ref{thm : periodicity of figure 8 in general}}\label{sec : proof of main thm}
With the above sections, we can now prove Theorem \ref{thm : periodicity of figure 8 in general}, which is exactly the same proof as that of Theorem \ref{thm : periodicity for teichmueller}.\\

\begin{proof}[of Theorem \ref{thm : periodicity of figure 8 in general}]
As stated above in Corollary \ref{cor : unique fixed point}, strict convexity of the trace function $\trace_\figeight\big|_{\SymLeaf_L}$ along any mixed give that there is a unique critical point and corresponds to the unique minimum. Therefore, the Hamiltonian flow of $\trace_\figeight\big|_{\SymLeaf_L}$ has a unique fixed point.\\

	Now we show that every orbit is periodic. Let $M$ be a non--empty level set of $\trace_\figeight\big|_{\SymLeaf_L}$ that does not correspond to the fixed point. In particular, $M$ is a a regular level set and therefore a smooth co--dimension one submanifold of $\SymLeaf_L$. The Hamiltonian flow of $\trace_\figeight\big|_{\SymLeaf_L}$ preserves the level sets of $\trace_\figeight\big|_{\SymLeaf_L}$.  Since $\trace_\figeight\big|_{\SymLeaf_L}$ is a proper function by Proposition \ref{prop : trace of figure 8 is proper},  $M$ is compact. By the fact that $\SymLeaf_L$ is two--dimensional (\cite{ConvexRealProj_Goldman}, also see Lemma \ref{cor : parametrization of symplectic leaf as level sets}), $M$ is a compact one--dimensional manifold without boundary. Therefore $M$ is a topological circle. Once again, since $M$ does not contain any fixed points, the Hamiltonian vector field restricted to $M$ is bounded away from zero. This implies that the orbit is the whole level set $M$ and is therefore periodic.
\end{proof}	

\subsection{Numerical solutions to the Hamiltonian flow}\label{sec : numerical sol to Ham flow fig 8}
From Equation \eqref{eq : Hamiltonian vector field in symplectic leaf in coordinates sigma1 tau1}, we can explicitly write down the vector field associated to the trace function $\trace_\figeight$ in coordinates. Let $L\in\R^6_{>0}$ define a symplectic leaf $\SymLeaf_L$. By Equation \eqref{eq : differential equation associated to the Hamiltonian vector field}, a path $(\sigma_1,\tau_1)\colon\R\to\SymLeaf_L$ is a flow line of the Hamiltonian flow of $\trace_\figeight\big|_{\SymLeaf_L}$ if 
\begin{gather*}
		\dot\sigma_1 = \frac{2}{\tau_1 (\ell _{\alpha ,1}^4 \ell _{\alpha ,2}^2 \ell _{\beta ,1} \ell _{\beta ,2}^3
   \ell _{\gamma ,1}^2 \ell _{\gamma ,2}^2)^{1/3}}\cdot \\
   \bigg(\sigma _1^3 \left(\ell _{\alpha ,2} \ell _{\beta ,1} \ell _{\gamma ,2}^{5/3}
   \left(\ell _{\alpha ,1} \ell _{\beta ,2} \ell _{\gamma
   ,1}\right){}^{2/3}-\tau _1^2 \ell _{\gamma ,2} \sqrt[3]{\ell _{\alpha ,2}
   \ell _{\beta ,1}}\right)+\\
   \ell _{\alpha ,1} \ell _{\beta ,2} \left(\ell _{\alpha ,2} \left(\ell
   _{\alpha ,1} \ell _{\beta ,2} \ell _{\gamma ,1} \ell _{\gamma
   ,2}\right){}^{2/3}-\tau _1^2 \ell _{\alpha ,1} \ell _{\gamma ,1}
   \sqrt[3]{\ell _{\alpha ,2} \ell _{\beta ,1}}\right)+\\
   \sigma_1^2\bigg(\ell _{\alpha ,2} \ell _{\beta ,1} \ell _{\gamma ,2}^{5/3} \left(\ell
   _{\alpha ,1} \ell _{\beta ,2} \ell _{\gamma ,1}\right){}^{2/3}+\ell
   _{\alpha ,1} \ell _{\beta ,2} \ell _{\gamma ,1} \ell _{\gamma ,2}
   \sqrt[3]{\ell _{\alpha ,2}^4 \ell _{\beta ,1}}+\sqrt[3]{\ell _{\alpha
   ,1}^4 \ell _{\alpha ,2}^2 \ell _{\beta ,1}^2 \ell _{\beta ,2}^4 \ell
   _{\gamma ,1} \ell _{\gamma ,2}^4}-\\
   \tau _1^2 \left(\ell _{\alpha ,1} \ell _{\gamma ,2} \sqrt[3]{\ell _{\alpha
   ,2} \ell _{\beta ,1}}+\left(\ell _{\alpha ,1} \ell _{\beta ,2} \ell
   _{\gamma ,1} \ell _{\gamma ,2}\right){}^{2/3}+\sqrt[3]{\ell _{\alpha ,1}
   \ell _{\alpha ,2}^2 \ell _{\beta ,1}^2 \ell _{\beta ,2} \ell _{\gamma ,1}
   \ell _{\gamma ,2}}\right)\bigg)+\\
   \sigma_1\bigg(\ell _{\alpha ,2} \sqrt[3]{\ell _{\alpha ,1}^5 \ell _{\beta ,2}^5 \ell
   _{\gamma ,1}^2 \ell _{\gamma ,2}^2}+\ell _{\alpha ,1} \ell _{\beta ,2}
   \ell _{\gamma ,1} \ell _{\gamma ,2} \sqrt[3]{\ell _{\alpha ,2}^4 \ell
   _{\beta ,1}}+\sqrt[3]{\ell _{\alpha ,1}^4 \ell _{\alpha ,2}^2 \ell
   _{\beta ,1}^2 \ell _{\beta ,2}^4 \ell _{\gamma ,1} \ell _{\gamma ,2}^4}-\\
   \tau _1^2 \left(\ell _{\alpha ,1}^{5/3} \left(\ell _{\beta ,2} \ell _{\gamma
   ,1} \ell _{\gamma ,2}\right){}^{2/3}+\ell _{\alpha ,1} \ell _{\beta ,2}
   \ell _{\gamma ,1} \sqrt[3]{\ell _{\alpha ,2} \ell _{\beta
   ,1}}+\sqrt[3]{\ell _{\alpha ,1}^4 \ell _{\alpha ,2}^2 \ell _{\beta ,1}^2
   \ell _{\beta ,2} \ell _{\gamma ,1} \ell _{\gamma ,2}}\right)\bigg)\bigg)
\end{gather*}
and
\begin{gather*}
	\dot\tau_1 = \frac{2}{\sigma _1 \ell _{\beta ,2} \sqrt[3]{\ell _{\alpha ,1}^4 \ell _{\alpha ,2}^2
   \ell _{\beta ,1} \ell _{\gamma ,1}^2 \ell _{\gamma ,2}^2}}\cdot\\
   \bigg(2 \sigma _1^3 \ell _{\alpha ,2} \ell _{\beta ,1} \ell _{\gamma ,2}^{5/3}
   \left(\ell _{\alpha ,1} \ell _{\beta ,2} \ell _{\gamma
   ,1}\right){}^{2/3}-\ell _{\alpha ,2} \sqrt[3]{\ell _{\alpha ,1}^5 \ell
   _{\beta ,2}^5 \ell _{\gamma ,1}^2 \ell _{\gamma ,2}^2}+\\
   \sigma _1^2 \left(\ell _{\alpha ,2} \ell _{\beta ,1} \ell _{\gamma ,2}^{5/3}
   \left(\ell _{\alpha ,1} \ell _{\beta ,2} \ell _{\gamma
   ,1}\right){}^{2/3}+\ell _{\alpha ,1} \ell _{\beta ,2} \ell _{\gamma ,1}
   \ell _{\gamma ,2} \sqrt[3]{\ell _{\alpha ,2}^4 \ell _{\beta
   ,1}}+\sqrt[3]{\ell _{\alpha ,1}^4 \ell _{\alpha ,2}^2 \ell _{\beta ,1}^2
   \ell _{\beta ,2}^4 \ell _{\gamma ,1} \ell _{\gamma ,2}^4}\right)+\\
   \tau_1^2\bigg(\sigma _1^2 \ell _{\alpha ,1} \ell _{\gamma ,2} \sqrt[3]{\ell _{\alpha ,2}
   \ell _{\beta ,1}}-\ell _{\alpha ,1}^2 \ell _{\beta ,2} \ell _{\gamma ,1}
   \sqrt[3]{\ell _{\alpha ,2} \ell _{\beta ,1}}+\sigma _1^2 \ell _{\alpha ,1} \ell _{\gamma ,2} \sqrt[3]{\ell _{\alpha ,2}
   \ell _{\beta ,1}}+\\
   \sigma_1^2\bigg(2 \sigma _1 \ell _{\gamma ,2} \sqrt[3]{\ell _{\alpha ,2} \ell _{\beta
   ,1}}+\left(\ell _{\alpha ,1} \ell _{\beta ,2} \ell _{\gamma ,1} \ell
   _{\gamma ,2}\right){}^{2/3}+\sqrt[3]{\ell _{\alpha ,1} \ell _{\alpha
   ,2}^2 \ell _{\beta ,1}^2 \ell _{\beta ,2} \ell _{\gamma ,1} \ell _{\gamma
   ,2}}\bigg)\bigg)+\\
   \tau_1\bigg(\sigma _1^2 \ell _{\alpha ,1} \ell _{\beta ,2} \left(\ell _{\gamma
   ,1}+1\right) \ell _{\gamma ,2} \sqrt[3]{\ell _{\alpha ,2} \ell _{\beta
   ,1}}-\ell _{\alpha ,1}^2 \ell _{\beta ,2} \ell _{\gamma ,1} \sqrt[3]{\ell
   _{\alpha ,2} \ell _{\beta ,1}}-\ell _{\alpha ,2} \sqrt[3]{\ell _{\alpha
   ,1}^5 \ell _{\beta ,2}^5 \ell _{\gamma ,1}^2 \ell _{\gamma ,2}^2}+\\
   4 \sigma _1^3 \sqrt[3]{\ell _{\alpha ,1} \ell _{\alpha ,2}^2 \ell _{\beta
   ,1}^2 \ell _{\beta ,2} \ell _{\gamma ,1} \ell _{\gamma ,2}^4}+\sigma_1^2\bigg(\ell _{\alpha ,2} \left(\ell _{\beta ,1}+1\right) \left(\ell _{\alpha ,1}
   \ell _{\beta ,2} \ell _{\gamma ,1} \ell _{\gamma ,2}\right){}^{2/3}+\\
   \sqrt[3]{\ell _{\alpha ,1} \ell _{\alpha ,2}^2 \ell _{\beta ,1}^2 \ell
   _{\beta ,2} \ell _{\gamma ,1} \ell _{\gamma ,2}^4}+\sqrt[3]{\ell _{\alpha
   ,1}^4 \ell _{\alpha ,2}^2 \ell _{\beta ,1}^2 \ell _{\beta ,2} \ell
   _{\gamma ,1} \ell _{\gamma ,2}^4}\bigg)\bigg)\bigg).
\end{gather*}
We supressed the dependence of $\sigma_1$ and $\tau_1$ on $t$ to make the expressions (mildly) more readable. This equation is computed in Section 5 of the Mathematica code by giving the function \texttt{HamiltonianVF} the input function \texttt{traceFigure8}.\\

These are coupled ordinary differential equations, and after some attempts with Mathematica, we were not able to find a closed form solution. In Section \ref{sec : unipotent locus} we focus on the unipotent locus to have much more simple differential equations, but we were still unable to solve the system of equations symbolically.\\

However, it is possible to solve the system of differential equations numerically on given symplectic leaves and given initial conditions. We will pick symplectic leaves containing a Fuchsian structure. As in Section \ref{sec : fuchsian locus}, we let $F = (\ell_\alpha =3,\ell_\beta = 6, \ell_\gamma = 8)$, which defines a symplectic leaf $\SymLeaf_{L(F)}$, and whose Fuchsian representation, in coordinates $\sigma_1,\tau_1$ as in Lemma \ref{cor : parametrization of symplectic leaf as level sets} and from Equation \eqref{eq : coordinates for fuchsian locus}, is given by $(\sigma_1^F = 2,\tau_1^F = 1)$. The system of ordinary differential equations describing the Hamiltonian vector field of $\trace_\figeight\big|_{\SymLeaf_{L(F)}}$ reads (once again suppressing the dependence of the coordinates on $t$):
\begin{align}\label{eq : ODE system Fuchsian}
	\begin{dcases}\dot\sigma_1&=\frac{-6 \sigma _1^3 \left(\tau _1^2-1\right)+\sigma _1^2 \left(17-90 \tau
   _1^2\right)+\sigma _1 \left(15-408 \tau _1^2\right)-576 \tau _1^2+4}{12
   \tau _1}\\
   \dot\tau_1 &=\frac{\left(\tau _1+1\right) \left(12 \sigma _1^3 \left(\tau
   _1+1\right)+\sigma _1^2 \left(90 \tau _1+17\right)-576 \tau
   _1-4\right)}{12 \sigma _1} 
   \end{dcases}
\end{align}
This equation is computed in Section 5 of the Mathematica code by giving the function \texttt{HamiltonianVF} the input function \texttt{traceFigure8} with the length vector \texttt{3,1/3,6,1/6,8,1/8,1,1}. The numerical solution with initial condition at the Fuchsian locus $(2,1)$ is shown in Figure \ref{fig : num sol fuchsian}, where we see the periodicity of the flow. We can also see that the Fuchsian structure is not fixed. The level sets of the function $\trace_\figeight\big|_{\SymLeaf_{L(F)}}$ are shown in Figure \ref{fig : level sets for fuchsian locus}, where we see that the orbits are closed.
\begin{figure}[h]
	\centering
		\includegraphics[width = \textwidth]{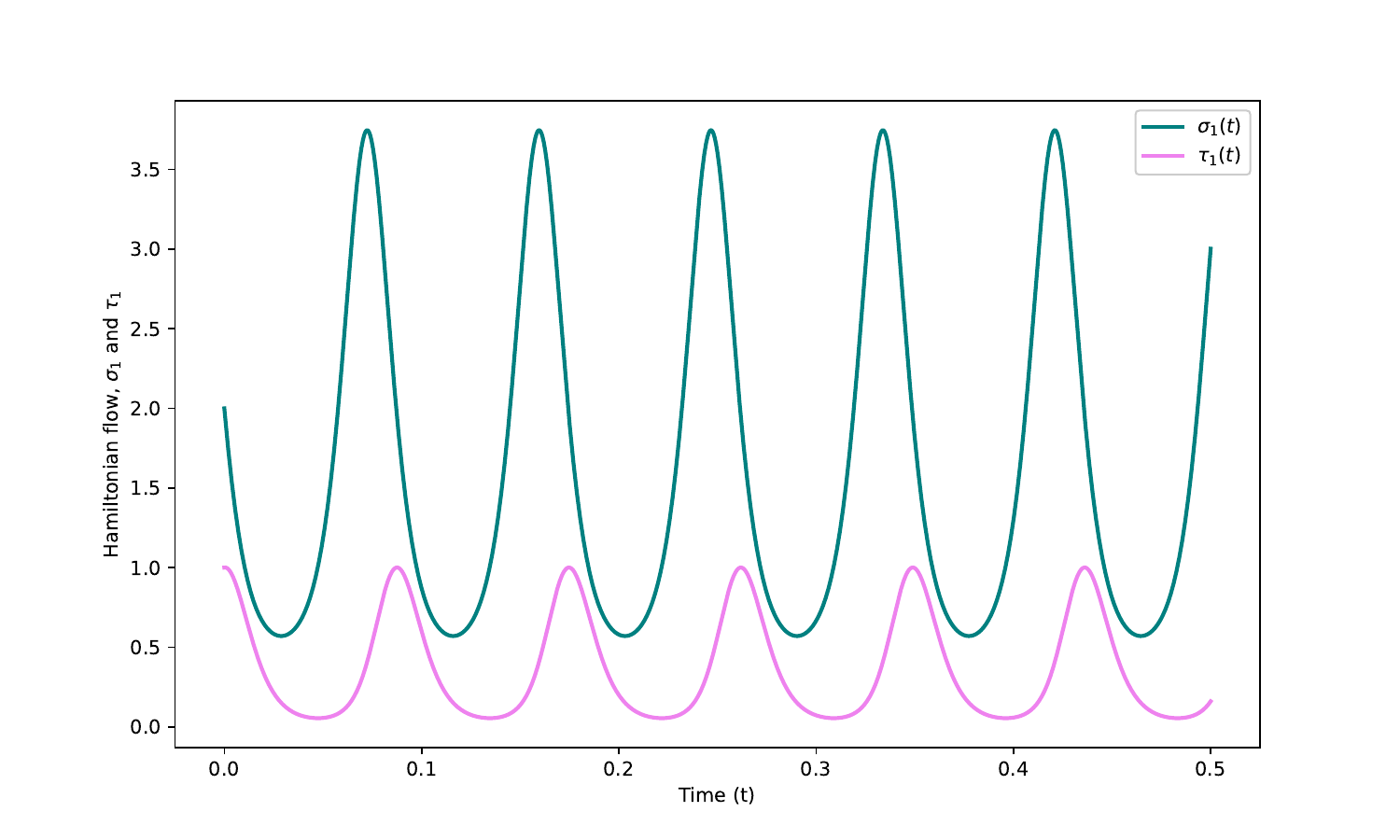}
		\caption{Numerical solution to the system of equations \eqref{eq : ODE system Fuchsian} with initial condition $(2,1)$ at the Fuchsian structure.}
		\label{fig : num sol fuchsian}
\end{figure}

\begin{figure}[h]
	\centering
	\includegraphics[scale = 0.7]{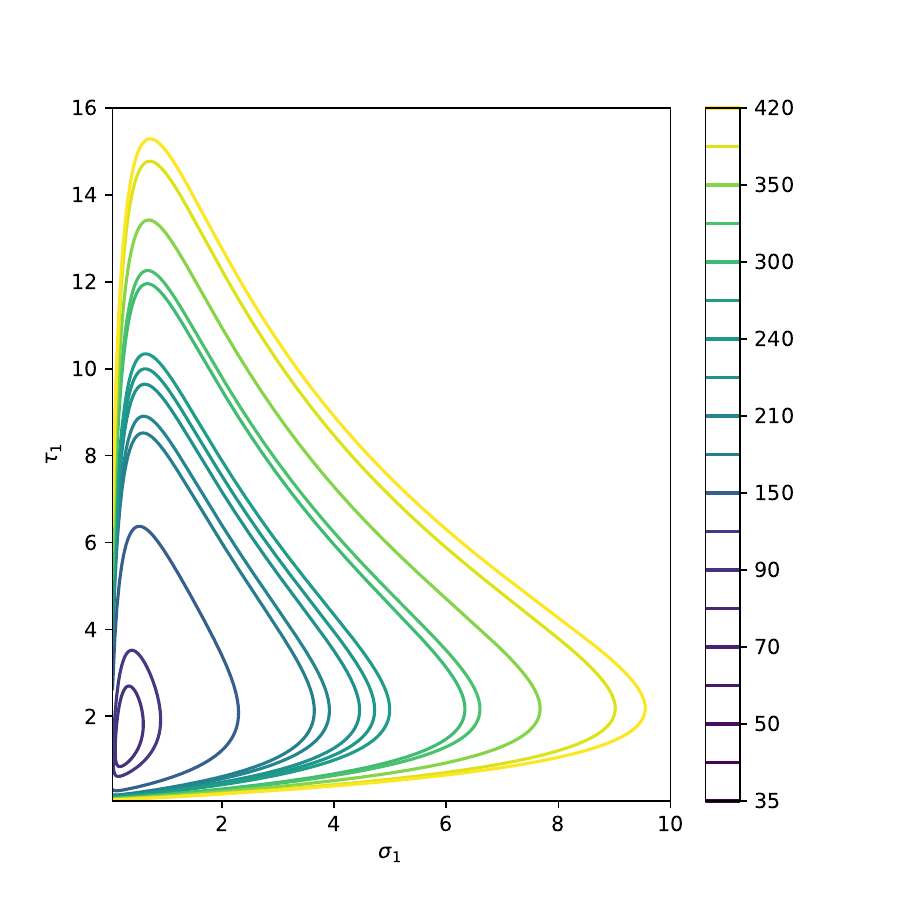}
	\caption{Level sets of the function $\trace_\figeight\big|_{\SymLeaf_{L(F)}}$ for $F = (3,6,8)$.}
	\label{fig : level sets for fuchsian locus}
\end{figure}

\begin{remark}\label{rmk : periods are not the same}
Since every orbit is periodic, one may wonder whether the orbits all have the same length and there is a circle action associated to the trace function. However, this is \emph{not} the case with the trace function, as we see in Figure \ref{fig : num sol for ic 43}.
\end{remark}

\begin{figure}[h!]
	\centering
	\includegraphics[width = \textwidth]{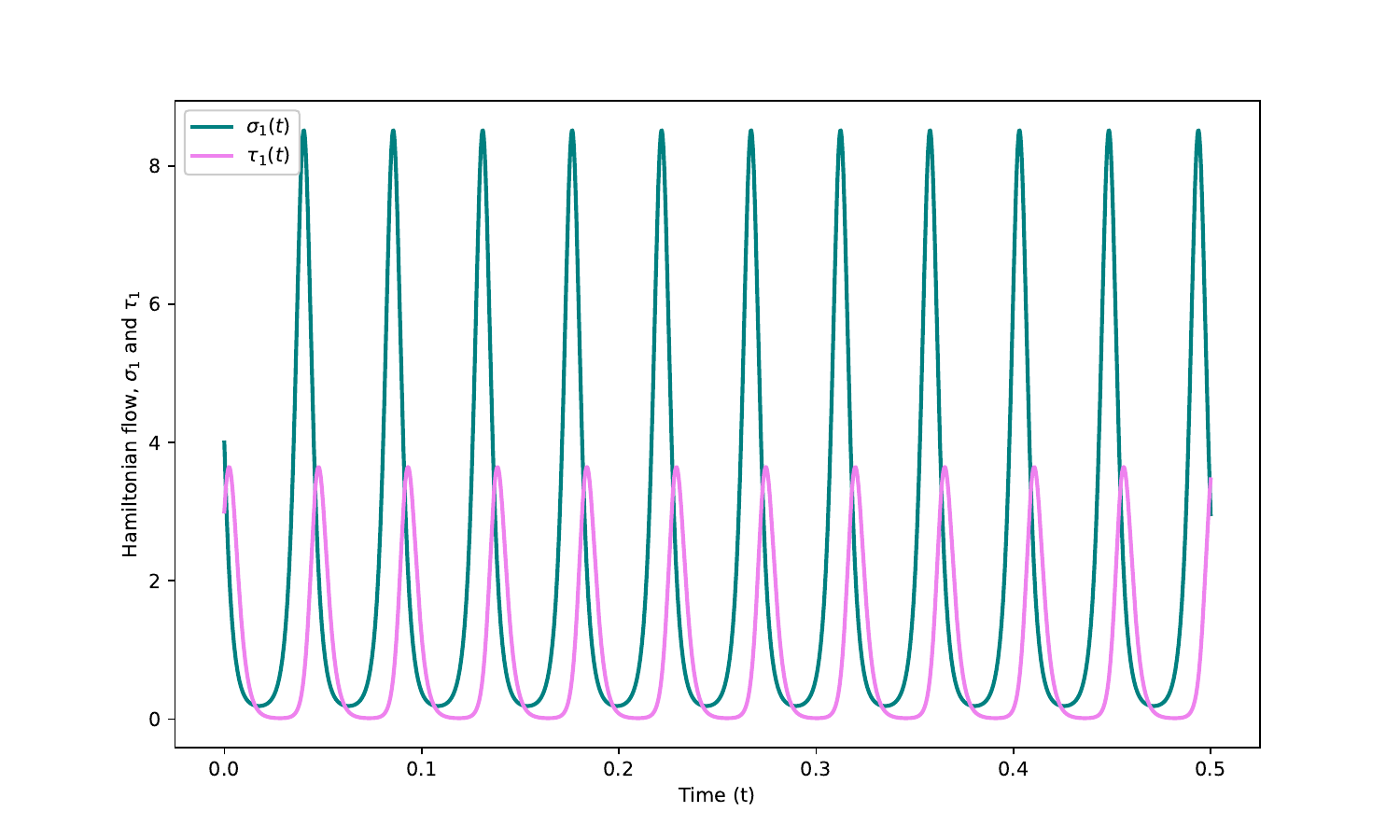}
	\caption{Numerical solution to the system of equations \eqref{eq : ODE system Fuchsian} with initial conditions $(4,3)$.}
	\label{fig : num sol for ic 43}
\end{figure}

\subsection{Trace of the $\Theta$--web}\label{sec : trace of the theta web}
We now prove Corollary \ref{cor Intro : web periodicity}. This is an application of Theorem \ref{thm : periodicity of figure 8 in general} where we show that an analogous version of the theorem holds when $\trace_\delta$ is replaced by the trace of the $\Theta$--web $\thweb$.\\

A \emph{web} is an embedded $3$--regular bipartite graph on the surface $\pants$\footnote{When considering representations from $\pi_1(\pants)\to\SL d$, \cite{DKS:webs} define $d$--webs. Since we are working in the $\PSL 3$ case, we refer to a $3$--web simply as a web.}. Given a representation $\rho\colon\pi_1(\pants)\to\PSL 3$ and a web $m$, \cite[Section 4]{Sikora:graphs} and \cite{DKS:webs} define the \emph{trace of the web $m$}, written as $\trace_m(\rho)$, which is invariant under conjugation. The definition of the trace of a web is in terms of tensor networks, and we do not recall the general definition here. Instead, we give the formula for a specific web.\\

Let $\thweb$ be the web with one black and one white vertex, with three edges between them each of multiplicity $1$, as in Figure \ref{fig : theta web}. The trace of $\thweb$ is computed in Section 6.2 of \cite{DKS:webs} to be
\begin{align*}
	\trace_{\thweb}\colon\CharVar_3(S)&\to\R\\
	[\rho]&\mapsto \trace(\rho(\alpha))\trace(\rho(\gamma))-\trace(\rho(\alpha\gamma^{-1})).
\end{align*} 
For any $L\in\R^6_{>0}$ there is an induced map
\begin{align*}
	\widehat\trace_{\thweb}\colon \SymLeaf_L&\to\R\\
	[f,\Sigma,\nu]&\mapsto \trace_{\thweb}(\rechol([f,\Sigma,\nu])).
\end{align*}
By Lemma \ref{lemma : symplectic leaves are exactly relative character varieties}, the quantity $\trace(\rho(\alpha))\trace(\rho(\gamma))$ is constant on the symplectic leaf $\SymLeaf_L$; denote this constant by $C_L$. The trace of the $\Theta$--web is thus given by $[\rho]\mapsto C_L - \trace(\rho(\delta))$, where $\delta = \alpha\gamma^{-1}$ is the figure eight curve from this section.  The following is then a direct corollary of Theorem \ref{thm : periodicity of figure 8 in general}.
\begin{corollary}\label{cor : trace of the theta web is periodic}
		Let $\thweb$ be the $\Theta$--web on a pair of pants $\pants$. Then for any $L\in\R^6_{>0}$, the trace function $\trace_{\thweb}\big|_{\SymLeaf_L}\colon\SymLeaf_L\to\R$ attains a unique maximum. Moreover, every orbit of the Hamiltonian flow of $\trace_{\thweb}\big|_{\SymLeaf_L}$ is periodic and there is a unique fixed point.
\end{corollary}

\subsection{The symmetrized trace}
In this section, we address Theorem \ref{thm intro : symmetrized trace}. The strategy for the proof is identical to that of Theorem \ref{thm : periodicity of figure 8 in general}. Since the computations are very similar to the ones done above, we delay the proof of Theorem \ref{thm intro : symmetrized trace}, Part \ref{thm intro : periodicity for symmetrized trace} to Appendix \ref{app : proof of periodicity for symmetrixed trace}. Here we prove Part \ref{thm intro : symmetrized trace fixed point} of Theorem \ref{thm intro : symmetrized trace}, which we restate now in more detail.

\begin{theorem}\label{thm : fixed points of symmetric trace}
	Let $\figeight = \alpha\gamma^{-1}$ be the figure eight curve. Let $F = (\ell_\alpha,\ell_\beta,\ell_\gamma)\in\R^3_{>0}$ define the length vector $L(F)\in\R^6_{>0}$ as in Equation \eqref{eq : Fuchsian locus lengths L(F)}. Let $\SymLeaf_{L(F)}$ be the symplectic leaf associated to $L(F)$. Then the unique fixed point of the Hamiltonian flow of the function $\trace_\figeight+\trace_{\figeight^{-1}}\big|_{\SymLeaf_F}$ is the unique hyperbolic structure in $\SymLeaf_{L(F)}$; which in coordinates is given by 
	\[
	\left(\sqrt\frac{\ell_\alpha\ell_\gamma}{\ell_\beta},\sqrt{\frac{\ell_\beta}{\ell_\alpha\ell_\gamma}},\sqrt{\frac{\ell_\alpha\ell_\beta}{\ell_\gamma}},\sqrt{\frac{\ell_\gamma}{\ell_\alpha\ell_\beta}},\sqrt{\frac{\ell_\beta\ell_\gamma}{\ell_\alpha}},\sqrt{\frac{\ell_\alpha}{\ell_\beta\ell_\gamma}},1,1\right).
	\] 
\end{theorem}

\begin{proof}
	By Part \ref{thm intro : periodicity for symmetrized trace} of Theorem \ref{thm intro : symmetrized trace}, we know that the Hamiltonian flow of $\trace_\figeight+\trace_{\figeight^{-1}}\big|_{\SymLeaf_{L(F)}}$ has a unique fixed point. Hence, we only need to show that at the hyperbolic structure, the Hamiltonian vector field is zero. Indeed, using Section 5 of the Mathematica code, we compute that the Hamiltonian vector field is given by
	\begin{gather*}
		\dot\sigma_1 = -\frac{1}{\sigma _1 \tau _1 \ell _{\alpha } \ell _{\beta } \ell _{\gamma }}\cdot\\
		2\big(\sigma _1^2 \ell _{\beta } \left(\ell _{\gamma } \left(\sigma _1 \left(\tau _1^2 \ell
   _{\beta }-1\right)+\tau _1^2-1\right)+\sigma _1 \left(\left(\sigma _1 \left(\tau
   _1^2-1\right)-1\right) \ell _{\beta }+\tau _1^2\right)\right)+\\
   \ell _{\alpha }^2 \ell _{\gamma } \left(\ell _{\gamma } \left(\sigma _1 \left(\tau
   _1^2 \ell _{\beta }-1\right)+\tau _1^2-1\right)+\sigma _1 \left(\left(\sigma _1
   \left(\tau _1^2-1\right)-1\right) \ell _{\beta }+\tau _1^2\right)\right)+\\
   \sigma_1\ell_{\alpha}\big(\ell _{\gamma } \left(\sigma _1^2 \ell _{\beta } \left(\tau _1^2-\ell _{\beta
   }\right)+\sigma _1 \left(\tau _1^2-1\right) \left(\ell _{\beta }^2+1\right)+\tau _1^2
   \ell _{\beta }-1\right)+\\
   \ell _{\gamma }^2 \left(\tau _1^2+\left(\sigma _1 \left(\tau _1^2-1\right)-1\right) \ell
   _{\beta }\right)+\sigma _1 \ell _{\beta } \left(\tau _1^2+\sigma _1 \left(\tau _1^2
   \ell _{\beta }-1\right)-1\right)\big)\big),
	\end{gather*}
	and
	\begin{gather*}
		\dot\tau = -\frac{2 \left(\tau _1+1\right)}{\sigma _1^2 \ell _{\alpha } \ell _{\beta } \ell _{\gamma
   }}\big(\sigma _1^3 \left(-\ell _{\beta }\right) \left(\tau _1+\tau _1 \ell _{\beta } \ell
   _{\gamma }+2 \sigma _1 \left(\tau _1+1\right) \ell _{\beta }+\ell _{\beta }+\ell
   _{\gamma }\right)+\\
   \sigma _1 \ell _{\alpha } \left(\ell _{\gamma }-\sigma _1^2 \ell _{\beta }\right)
   \left(\ell _{\beta } \left(\tau _1+\ell _{\gamma }\right)+\tau _1 \ell _{\gamma
   }+1\right)+\\
   \ell _{\alpha }^2 \ell _{\gamma } \left(\ell _{\gamma } \left(\sigma _1+2 \tau _1+\sigma
   _1 \tau _1 \ell _{\beta }+2\right)+\sigma _1 \left(\tau _1+\ell _{\beta
   }\right)\right)\big).
	\end{gather*}
	A computation in Mathematica then shows that at the values
	\[
	(\sigma_1,\tau_1) = \left(\sqrt{\frac{\ell_\alpha\ell_\gamma}{\ell_\beta}},1\right)
	\]
	make the two above expressions zero.
\end{proof}

\section{The unipotent locus}\label{sec : unipotent locus}
In this section, we focus on the symplectic leaf of $\widehat\DefSpace(\pants)$ corresponding to the convex projective structures where all boundaries are cuspidal.\\

By Lemma \ref{cor : parametrization of symplectic leaf as level sets}, we see that the unipotent locus is parameterized as
\[
\unipotentLocus = \left\{\left(\sigma_1,\frac{1}{\sigma_1},\sigma_1,\frac{1}{\sigma_1},\sigma_1,\frac{1}{\sigma_1},\tau_1,\frac{1}{\tau_1}\right)\in\R^8_{>0}\st\sigma_1,\tau_1>0\right\}.
\]

The expressions for traces of curves simplify significantly in this case, and allow us to prove similar results to Theorem \ref{thm : periodicity of figure 8 in general} for a larger class of self--intersecting curves. Namely, we prove Theorem \ref{thm : intro commutator and k intersections}, which is a version of Theorem \ref{thm : periodicity of figure 8 in general} for the commutator $[\alpha,\gamma]$ and the curve $\alpha^k\gamma^{-1}$ for $k\geq 2$, when restricted to the unipotent locus. Moreover, we find the fixed points of the respective Hamiltonian flows.\\

As another application of the simplicity of the expressions, we find another way of writing the Hamiltonian flows $\HmFlow_\hexagon$ and $\HmFlow_\eruption$, which we do in Section \ref{sec : eruption and hexagon conjugating matrices}.

\subsection{The commutator}\label{sec : commutator unipotent locus} The first self--intersecting curve we consider apart from the figure eight curve, is the commutator shown in Figure \ref{fig : commutator}.\\

Following the same proof of Theorem \ref{thm Intro : main theorem} in Section \ref{sec : trace of figure eight curve}, we show the following.
\begin{theorem}\label{thm : commutator periodicity and unique fixed point}
	The trace function $\trace_{[\alpha,\gamma]}\big|_{\unipotentLocus}\colon\unipotentLocus\to\R$ attains a unique minimum. Moreover, every orbit of the  Hamiltonian flow of $\trace_{[\alpha,\gamma]}\big|_{\unipotentLocus}$ is periodic and there is a unique fixed point.
\end{theorem}
Following the same structure as that of Section \ref{sec : trace of figure eight curve}, we begin by using Mathematica to compute the expression for the trace of the commutator on the unipotent locus.\\

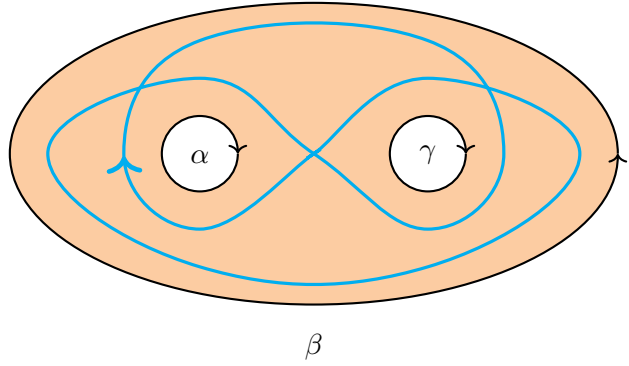
\begin{figure}[htb]
	\centering
	\begin{tikzpicture}
    \fill[Apricot!70] (0,0) ellipse (4cm and 2cm);
    
    \fill[white] (1.5,0) circle (0.5cm);
    \fill[white] (-1.5,0) circle (0.5cm);
    
    \draw[->,thick] (2,0) arc[start angle = 360, end angle = 0, radius = 0.5];
    \draw[->,thick] (-1,0) arc[start angle = 360, end angle = 0, radius = 0.5];
    \draw[->,thick] (4,0) arc [start angle = 0,end angle = 360, x radius = 4cm, y radius = 2cm];
    
    \node[label = $\gamma$] at (1.5,-0.4) {};
    \node[label = $\alpha$] at (-1.5,-0.4) {};
    \node[label = $\beta$] at (0,-3) {};
    
    \draw[very thick,cyan,] plot [smooth cycle, tension = 0.75] coordinates {(0,0) (-1.5,-1) (-2.5,0) (-1.5,1.5) (1.5,1.5) (2.5,0) (1.5,-1) (0,0) (-1.5,1) (-3.5,0) (-1.5,-1.5) (1.5,-1.5) (3.5,0) (1.5,1) (0,0)}  [arrow inside={opt={scale=1.5}}{0.05}];
	\end{tikzpicture}
\caption{The commutator $[\alpha,\gamma]$.}
\label{fig : commutator}
\end{figure}

\begin{lemma}
	In coordinates, we have that
	\begin{gather*}
		\trace_{[\alpha,\gamma]}\big|_{\unipotentLocus}\colon\unipotentLocus\to\R\\
		(\sigma_1,\tau_1)\mapsto 
		\frac{1}{\sigma_1^3\tau_1}\big(
		\sigma _1^6 \left(\tau _1+1\right){}^3+3 \sigma _1^5 \left(\tau
   _1+1\right){}^2 \left(2 \tau _1+1\right)+3 \sigma _1^4 \left(\tau
   _1+1\right){}^2 \left(5 \tau _1+1\right)\\
   +\sigma _1^3 \left(20 \tau
   _1^3+42 \tau _1^2+27 \tau _1+2\right)+3 \sigma _1^2 \left(\tau
   _1+1\right){}^2 \left(5 \tau _1+1\right)\\+3 \sigma _1 \left(\tau
   _1+1\right){}^2 \left(2 \tau _1+1\right)+\left(\tau
   _1+1\right){}^3\big)
	\end{gather*}
\end{lemma}

This is computed in Section 4 of the Mathematica code by displaying the function \texttt{traceCommutator} with length vector \texttt{1,1,1,1,1,1}. With this expression, we move on to show properness.
\begin{proposition}
	The function $\trace_{[\alpha,\gamma]}\big|_{\unipotentLocus}\colon\unipotentLocus\to\R$ is proper. In particular, it realizes a minimum in $\unipotentLocus$.
\end{proposition}
\begin{proof}
	Similarly to the proof of Proposition \ref{prop : trace of figure 8 is proper}, let $(\sigma_{1,n},\tau_{1,n})_{n\in\N}$ be a sequence such that as $n$ goes to $\infty$, $(\sigma_{1,n},\tau_{1,n})$ goes to a tuple in $\{(\infty,\infty), (0,\infty),(\infty,0),(x,0),(0,y)\st x,y\geq 0\}$. We need to show that for any of these options, $\trace_{[\alpha,\gamma]}(\sigma_{1,n},\tau_{1.n})\to\infty$ as $n\to\infty$. Once again, since the function is positive and all signs in front of the monomials are positive, we only need to find single terms in the expression for the trace going to $\infty$:
	\begin{enumerate}
		\item Assume $(\sigma_{1,n},\tau_{1,n})\to(\infty,\infty)$, then the term 
		\[
		\frac{\sigma_{1,n}^6\tau_{1,n}^3}{\sigma^3_{1,n}\tau_{1,n}}\xrightarrow{n\to\infty}\infty.
		\]
		\item Assume $(\sigma_{1,n},\tau_{1,n})\to(0,\infty)$, then the term
		\[
		\frac{20\sigma_{1,n}^3\tau_{1,n}^3}{\sigma_{1,n}^3\tau_{1,n}}\xrightarrow{n\to\infty}\infty. 
		\]
		\item Assume $(\sigma_{1,n},\tau_{1,n})\to(\infty,0)$, then the term
		\[
		\frac{\sigma_{1,n}^6}{\sigma_{1,n}^3\tau_{1,n}}\xrightarrow{n\to\infty}\infty. 
		\]
		\item Assume $(\sigma_{1,n},\tau_{1,n})\to(x,0)$ or $(0,y)$ for $x,y\geq 0$, then the term
		\[
		\frac{1}{\sigma_{1,n}^3\tau_{1,n}}\xrightarrow{n\to\infty}\infty.
		\]
	\end{enumerate}
	This shows that the function $\trace_{[\alpha,\gamma]}\big|_{\unipotentLocus}\colon\unipotentLocus\to\R$ is proper. Since the function is also strictly positive, it attains a minimum.
\end{proof}

The next step in the proof of Theorem \ref{thm : commutator periodicity and unique fixed point} is to show convexity along any mixed flow.\\

\begin{proposition}\label{prop : commutator strictly convex}
	The function $\trace_{[\alpha,\gamma]}\big|_{\unipotentLocus}\colon\unipotentLocus\to\R$ is strictly convex along any mixed flow.
\end{proposition}
\begin{proof}
	Consider first the mixed flow $\Psi^t_a=\HmFlow^{at}_\hexagon\circ\HmFlow_\eruption^t$. In Section 6 of the Mathematica code, set \texttt{testf} equal to \texttt{traceCommutator} using the length vector \texttt{1,1,1,1,1,1}. Then inputting the function \texttt{testfMix1} to \texttt{SecondDerFlow}, we obtain that
	\begin{gather*}
		\frac{\partial^2}{\partial t^2}(\trace_{[\alpha,\gamma]}(\Psi^t_a(\sigma_1,\tau_1))) = \frac{e^{-at}(1+\sigma_1)^4a^2}{\sigma_1^3\tau_1}\cdot\\
		\bigg(3e^{2at}\tau_1^2+4e^{3at}\tau_1^3+3\sigma_1^2\tau_1^2e^{2at}+4\sigma_1^2\tau_1^3e^{3at} + 3\sigma_1\tau_1^2e^{2at}+8\sigma_1\tau_1^3e^{3at}+\sigma_1+(\sigma_1-1)^2\bigg).
	\end{gather*}
	We observe that for any $a\in\R$ and $t\in\R$, and any $\sigma_1,\tau_1$, the above expression is positive. Hence $t\mapsto \trace_{[\alpha,\gamma]}\big|_{\unipotentLocus}(\Psi^t_a(\sigma_1,\tau_1))$ is a strictly convex function.\\
	
	Now consider the other mixed flow $\Psi^t_a=\HmFlow^{at}_\eruption\circ\HmFlow_\hexagon^t$. In Section 6 of the Mathematica code, set \texttt{testf} equal to \texttt{traceCommutator} using the length vector \texttt{1,1,1,1,1,1}. Then inputting the function \texttt{testfMix2} to \texttt{SecondDerFlow}, we obtain that
	\begin{gather*}
		\frac{\partial^2}{\partial t^2}(\trace_{[\alpha,\gamma]}(\Psi^t_a(\sigma_1,\tau_1))) = 21 a^2 \sigma _1 e^{a t}+\frac{21 a^2 e^{-a t}}{\sigma _1}+\frac{48 a^2
   e^{-2 a t}}{\sigma _1^2}+\frac{27 a^2 e^{-3 a t}}{\sigma _1^3}\\
   +\frac{3 a^2 e^{-a t}}{\sigma _1 \tau _1}+\frac{12 a^2 e^{-2 a t}}{\sigma
   _1^2 \tau _1}+\frac{9 a^2 e^{-3 a t}}{\sigma _1^3 \tau _1}+27 a^2 \sigma
   _1^3 e^{3 a t}+48 a^2 \sigma _1^2 e^{2 a t}\\
   +\frac{9 a^2 \sigma _1^3 e^{3 a t}}{\tau _1}+\frac{12 a^2 \sigma _1^2 e^{2 a
   t}}{\tau _1}+\frac{3 a^2 \sigma _1 e^{a t}}{\tau _1}+\frac{27 a^2 \tau _1
   e^{-3 a t}}{\sigma _1^3}\\
   +60 a^2 \sigma _1^2 \tau _1 e^{2 a t}+33 a^2 \sigma _1 \tau _1 e^{a
   t}+\frac{33 a^2 \tau _1 e^{-a t}}{\sigma _1}+\frac{60 a^2 \tau _1 e^{-2 a
   t}}{\sigma _1^2}\\
   +27 a^2 \sigma _1^3 \tau _1 e^{3 a t}+\frac{15 a^2 \tau _1^2 e^{-a t}}{\sigma
   _1}+\frac{24 a^2 \tau _1^2 e^{-2 a t}}{\sigma _1^2}+\frac{9 a^2 \tau _1^2
   e^{-3 a t}}{\sigma _1^3}\\
   +9 a^2 \sigma _1^3 \tau _1^2 e^{3 a t}+24 a^2 \sigma _1^2 \tau _1^2 e^{2 a
   t}+15 a^2 \sigma _1 \tau _1^2 e^{a t}
		\end{gather*}
		We observe again that for any $a\in\R$ and $t\in\R$, and any $\sigma_1,\tau_1$, the above expression is positive. Hence $t\mapsto \trace_{[\alpha,\gamma]}\big|_{\unipotentLocus}(\Psi^t_a(\sigma_1,\tau_1))$ is a strictly convex function.
\end{proof}

The proof of Theorem \ref{thm : commutator periodicity and unique fixed point} then is identical to the proof of Theorem \ref{thm : periodicity of figure 8 in general} in Section \ref{sec : proof of main thm}.\\

The system of differential equations associated to the function $\trace_{[\alpha,\gamma]}\big|_{\unipotentLocus}$ is given by
\begin{equation}\label{eq : differential eq commutator unipotent}
\begin{dcases}
	\dot\sigma_1 = \frac{2 \left(\sigma _1+1\right){}^4 \left(\tau _1+1\right) \left(\sigma
   _1^2 \left(2 \tau _1^2+\tau _1-1\right)+\sigma _1 \left(4 \tau _1^2-\tau
   _1+1\right)+2 \tau _1^2+\tau _1-1\right)}{\sigma _1^2 \tau _1}\\
   \dot\tau_1 = \frac{6 \left(\sigma _1-1\right) \left(\sigma _1+1\right){}^3 \left(\tau
   _1+1\right){}^2 \left(\sigma _1^2 \left(\tau _1+1\right)+2 \sigma _1 \tau
   _1+\tau _1+1\right)}{\sigma _1^3}
\end{dcases}
\end{equation}
This is computed by taking as input to the function \texttt{HamiltonianVF} in Section 5 of the Mathematica code, the function \texttt{traceCommutator} with length vector \texttt{1,1,1,1,1,1}. Once again, a symbolic solution to the differential equation was not possible for us, and we show in Figure \ref{fig : num sol comm} a numerical solution. In Figure \ref{fig : level sets comm} we show some level sets for the trace function of the commutator on the unipotent locus.\\

\begin{figure}[h!]
	\centering
	\includegraphics[width = \textwidth]{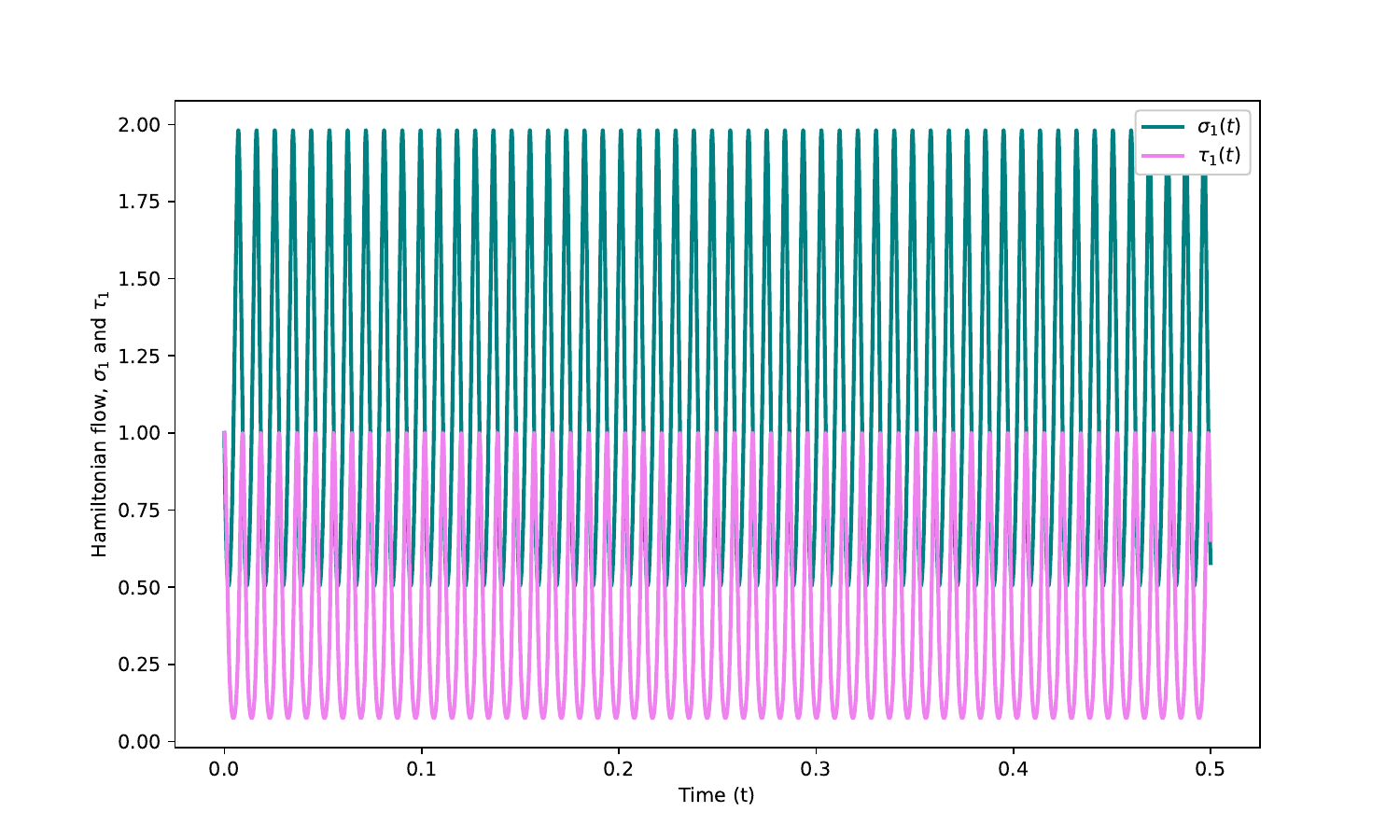}
	\caption{Numerical solution to the Hamiltonian flow of $\trace_{[\alpha,\gamma]}\big|_{\unipotentLocus}$ with initial conditions at the Fuchsian structure $(1,1)$.}
	\label{fig : num sol comm}
\end{figure}

\begin{figure}[h!]
	\centering
	\includegraphics[scale = 0.7]{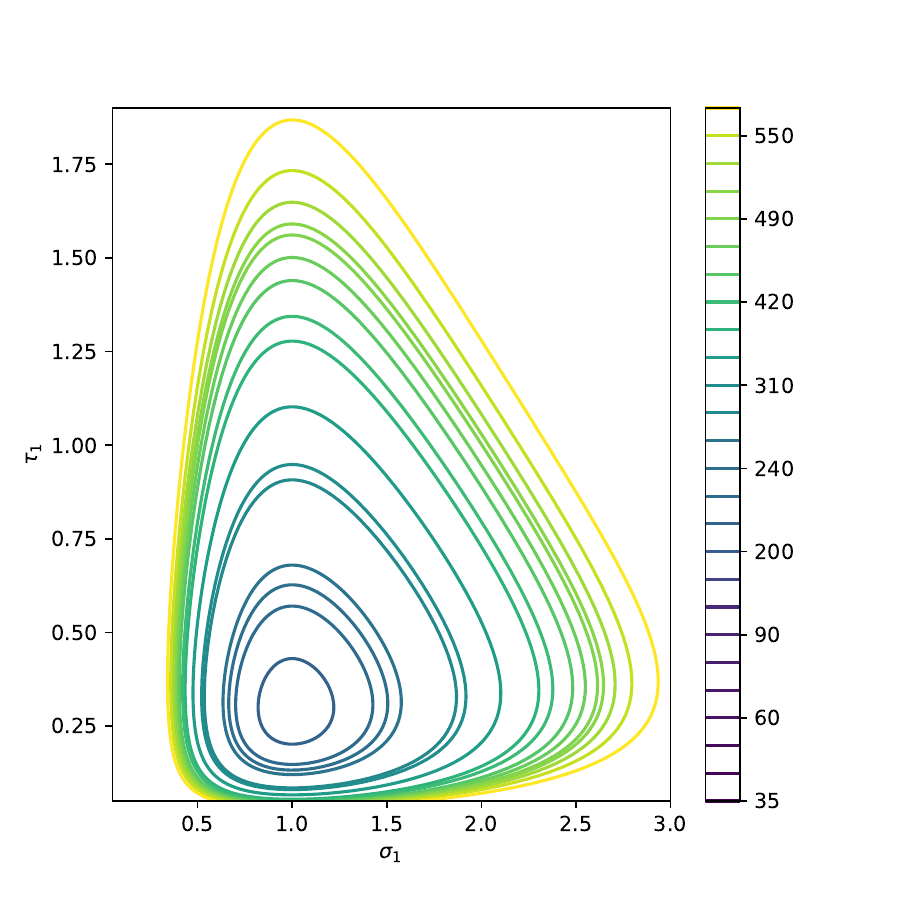}
	\caption{Level sets of the function $\trace_{[\alpha,\gamma]}\big|_{\unipotentLocus}$.}
	\label{fig : level sets comm}
\end{figure}

By Theorem \ref{thm : commutator periodicity and unique fixed point}, the function $\trace_{[\alpha,\gamma]}|_{\unipotentLocus}$ has a unique fixed point, and from the differential equation \eqref{eq : differential eq commutator unipotent} coming from the Hamiltonian vector field, we immediately obtain the following.

\begin{corollary}
	The fixed point of the Hamiltonian flow of the function $\trace_{[\alpha,\gamma]}\big|_{\unipotentLocus}$ is 
	\[
	(\sigma_1,\tau_1) = \left(1,\frac{\sqrt{33}-1}{16}\right).
	\]
	This value is also the minimum of the function.
\end{corollary}

\subsection{Curve with $k$--self intersections}\label{sec : alpha^kgamma^-1}
Similarly as above, we have the following result for a curve with $k$--self intersections as in Figure \ref{fig : k self intersections}.

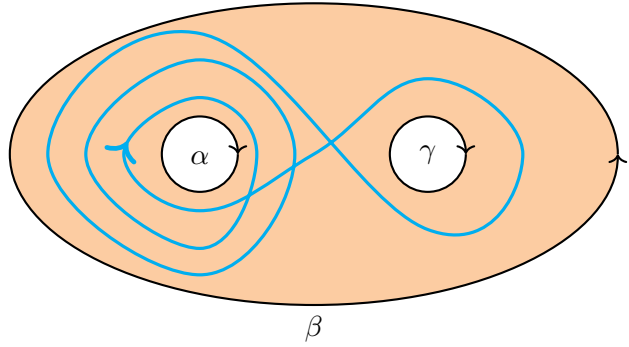
\begin{figure}[htb]
	\centering
	\begin{tikzpicture}
    \fill[Apricot!70] (0,0) ellipse (4cm and 2cm);
    
    \fill[white] (1.5,0) circle (0.5cm);
    \fill[white] (-1.5,0) circle (0.5cm);
    
    \draw[->,thick] (2,0) arc[start angle = 360, end angle = 0, radius = 0.5];
    \draw[->,thick] (-1,0) arc[start angle = 360, end angle = 0, radius = 0.5];
    \draw[->,thick] (4,0) arc [start angle = 0,end angle = 360, x radius = 4cm, y radius = 2cm];
    
    \node[label = $\gamma$] at (1.5,-0.4) {};
    \node[label = $\alpha$] at (-1.5,-0.4) {};
    \node[label = $\beta$] at (0,-2.75) {};
    
    \draw[very thick,cyan,] plot [smooth cycle, tension = 0.75] coordinates {(0,0) (-1.5,-0.75) (-2.5,0) (-1.5,0.75) (-0.75,0) (-1.5,-1.25) (-3,0) (-1.5,1.25) (-0.25,0) (-1.5,-1.6) (-3.5,0) (-1.5,1.6)  (1.5,-1) (2.75,0) (1.5,1)}  [arrow inside={opt={scale=1.5}}{0.05}];
	\end{tikzpicture}
\caption{An example of the curve $\alpha^k\gamma^{-1}$ with $k = 3$.}
\label{fig : k self intersections}
\end{figure}
\begin{theorem}\label{thm : periodicity alpha^kgammainverse}
	Let $k\in\N_{>0}$. The function $\trace_{\alpha^k\gamma^{-1}}\big|_{\unipotentLocus}\colon\unipotentLocus\to\R$ attains a unique minimum. Moreover, every orbit of the  Hamiltonian flow of $\trace_{\alpha^k\gamma^{-1}}\big|_{\unipotentLocus}$ is periodic and there is a unique fixed point.
\end{theorem}

We begin with a lemma expressing the trace of $\alpha^k\gamma^{-1}$ in the unipotent locus in coordinates.

\begin{lemma}
	In coordinates, we have that
	\begin{gather*}
		\trace_{\alpha^k\gamma^{-1}}\big|_{\unipotentLocus}\colon\unipotentLocus\to\R\\
		(\sigma_1,\tau_1)\mapsto \frac{k^2 \left(\sigma _1+1\right){}^3 \left(\tau _1+1\right){}^2+k
   \left(\sigma _1+1\right){}^3 \left(\tau _1+1\right){}^2+6 \sigma _1 \tau
   _1}{2 \sigma _1 \tau _1}
	\end{gather*}
\end{lemma}

This is computed in Section 7 of the Mathematica code. We can now show properness.

\begin{proposition}\label{prop : properness a^kgamma^-1}
	The function $\trace_{\alpha^k\gamma^{-1}}\big|_{\unipotentLocus}\colon\unipotentLocus\to\R$ is proper for any $k\in\N_{>0}$. In particular, it realizes a minimum in $\unipotentLocus$.
\end{proposition}
\begin{proof}
	Similarly to the proof of Proposition \ref{prop : trace of figure 8 is proper}, let $(\sigma_{1,n},\tau_{1,n})_{n\in\N}$ be a sequence such that as $n$ goes to $\infty$, $(\sigma_{1,n},\tau_{1,n})$ goes to a tuple in $\{(\infty,\infty), (0,\infty),(\infty,0),(x,0),(0,y)\st x,y\geq 0\}$. As above we show the following.
	\begin{enumerate}
		\item Assume $(\sigma_{1,n},\tau_{1,n})\to(\infty,\infty)$, then the term 
		\[
		\frac{k^2\sigma_{1,n}^3\tau_{1,n}^2}{2\sigma_{1,n}\tau_{1,n}}\xrightarrow{n\to\infty}\infty.
		\]
		\item Assume $(\sigma_{1,n},\tau_{1,n})\to(\infty,0)$, then the term 
		\[
		\frac{k^2\sigma_{1,n}^3}{2\sigma_{1,n}\tau_{1,n}}\xrightarrow{n\to\infty}\infty.
		\]
		\item Assume $(\sigma_{1,n},\tau_{1,n})\to(0,\infty)$, then the term 
		\[
			\frac{k^2\tau_{1,n}^2}{2\sigma_{1,n}\tau_{1,n}}\xrightarrow{n\to\infty}\infty.
		\]
		\item Assume $(\sigma_{1,n},\tau_{1,n})\to(x,0)$ or $(0,y)$ for $x,y\geq 0$, then the term
		\[
		\frac{k^2}{2\sigma_{1,n}\tau_{1,n}}\xrightarrow{n\to\infty}\infty.
		\]
	\end{enumerate}
	This shows that the function $\trace_{\alpha^k\gamma^{-1}}\big|_{\unipotentLocus}\colon\unipotentLocus\to\R$ is proper. Since the function is strictly positive, it attains a minimum.
\end{proof}

The next step in the proof of Theorem \ref{thm : periodicity alpha^kgammainverse} is to show convexity along any mixed flow.\\

\begin{proposition}\label{prop : a^kgamma^-1 convexity}
	For any $k\in\N_{>0}$, the function $\trace_{\alpha^k\gamma^{-1}}\big|_{\unipotentLocus}$ is strictly convex along any mixed flow.
\end{proposition}
\begin{proof}
	Consider first the mixed flow $\Psi^t_a=\HmFlow^{at}_\hexagon\circ\HmFlow_\eruption^t$. In Section 7 of the Mathematica code, take as input to the function \texttt{SecondDerFlowK} the function \texttt{testfMix1K} to get that 
	\begin{gather*}
		\frac{\partial^2}{\partial t^2}(\trace_{\alpha^k\gamma^{-1}}(\Psi^t_a(\sigma_1,\tau_1))) = \frac{a^2 k (k+1) \left(\sigma _1+1\right){}^3 e^{-a t} \left(\tau _1^2 e^{2
   a t}+1\right)}{2 \sigma _1 \tau _1}.
		\end{gather*}
		For any $a\in\R$, and $t\in\R$ the second derivative is strictly positive, and hence the function $t\mapsto \trace_{\alpha^k\gamma^{-1}}\big|_{\unipotentLocus}(\Psi^t_a(\sigma_1,\tau_1))$ is strictly convex.\\
		
		Now consider the other mixed flow $\Psi^t_a=\HmFlow^{t}_\hexagon\circ\HmFlow_\eruption^{at}$. In Section 7 of the Mathematica code, take as input to the function \texttt{SecondDerFlowK} the function \texttt{testfMix2K} to get that
		\[
		\frac{\partial^2}{\partial t^2}(\trace_{\alpha^k\gamma^{-1}}(\Psi^t_a(\sigma_1,\tau_1)))=\frac{a^2 k (k+1) \left(\tau _1+1\right){}^2 e^{-a t} \left(4 \sigma _1^3
   e^{3 a t}+3 \sigma _1^2 e^{2 a t}+1\right)}{2 \sigma _1 \tau _1}.
		\]
		This quantity is positive and hence the function $t\mapsto \trace_{\alpha^k\gamma^{-1}}\big|_{\unipotentLocus}(\Psi^t_a(\sigma_1,\tau_1))$ is strictly convex.
\end{proof}

The proof of Theorem \ref{thm : periodicity alpha^kgammainverse} then is identical to the proof of Theorem \ref{thm : periodicity of figure 8 in general} in Section \ref{sec : proof of main thm}. The system of differential equations associated to the function $\trace_{\alpha^k\gamma^{-1}}\big|_\unipotentLocus$ is given by
\begin{equation}\label{eq : Hamiltonian vector field powers of a}
\begin{dcases}
\dot\sigma_1 &= -\frac{k (k+1) \left(\sigma _1+1\right){}^3 \left(\tau _1^2-1\right)}{\tau
   _1}\\
   \dot\tau_1&=	\frac{k (k+1) \left(\sigma _1+1\right){}^2 \left(2 \sigma _1-1\right)
   \left(\tau _1+1\right){}^2}{\sigma _1}
\end{dcases}
\end{equation}
This is computed at the end of Section 7 of the Mathematica code. Once again, a symbolic solution (even for the case $k = 1$) was not possible for us, and we show in Figure \ref{fig : num sol powers fuchsian} a numerical solution for the case $k = 3$. In Figure \ref{fig : level sets powers gamma} we show the level sets for the case $k = 3$, and Figure \ref{fig : Intro level sets} shows level sets for the case $k =1$, i.e. for the figure eight curve.\\

\begin{figure}[h]
	\centering
	\includegraphics[width = \textwidth]{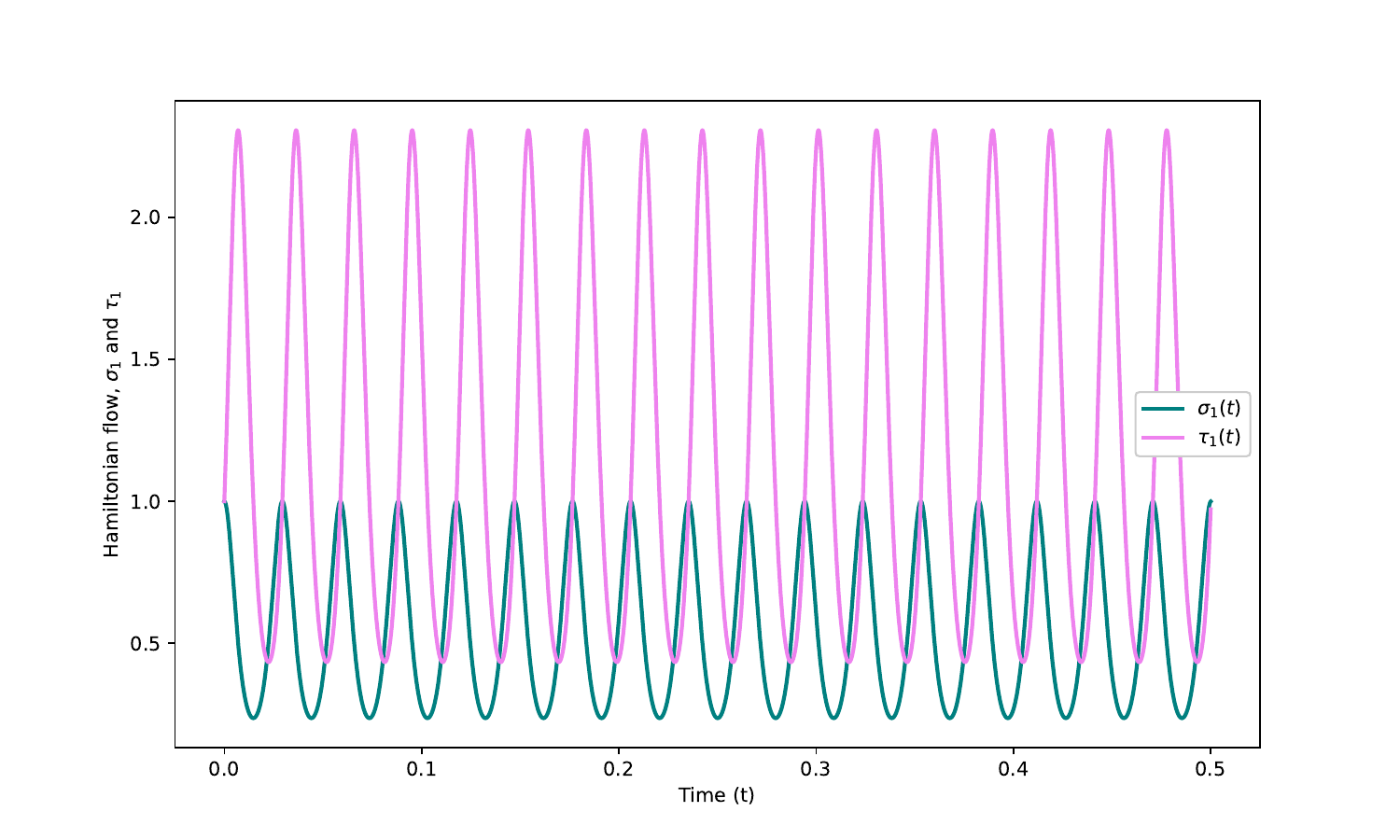}
	\caption{Numerical solution to $\trace_{\alpha^3\gamma^{-1}}\big|_\unipotentLocus$ with initial condition at the Fuchsian structure $(1,1)$.}
	\label{fig : num sol powers fuchsian}
\end{figure}

\begin{figure}[h]
	\centering
	\includegraphics[scale=0.7]{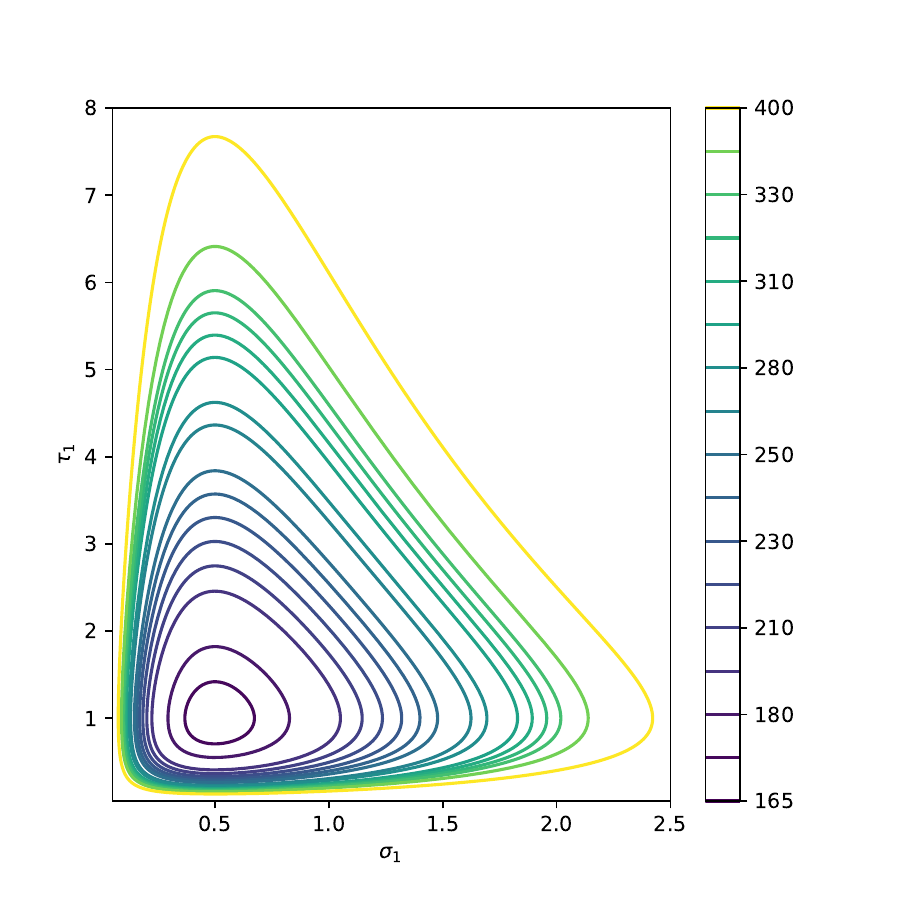}
	\caption{Level sets of the function $\trace_{\alpha^3\gamma^{-1}}\big|_\unipotentLocus$.}
	\label{fig : level sets powers gamma}
\end{figure}

By Theorem \ref{thm : periodicity alpha^kgammainverse}, the function $\trace_{\alpha^k\gamma^{-1}}|_{\unipotentLocus}$ has a unique fixed point, and from the differential equation \eqref{eq : Hamiltonian vector field powers of a} coming from the Hamiltonian vector field, we immediately obtain the following.

\begin{corollary}
	Let $k\in\N_{>0}$. The fixed point of the Hamiltonian flow of the function $\trace_{\alpha^k\gamma^{-1}}\big|_{\unipotentLocus}$ is 
	\[
	(\sigma_1,\tau_1) = \left(\frac{1}{2},1\right).
	\]
	This value is also the minimum of the function. In particular, it does not depend on $k$.
\end{corollary}

\subsection{Eruption and hexagon flows with conjugating matrices}\label{sec : eruption and hexagon conjugating matrices}
Denote by $\Phi_\eruption$ and respectively by $\Phi_\hexagon$ the flows $\rechol\circ\HmFlow_\eruption$ and respectively $\rechol\circ\HmFlow_\hexagon$ on $\CharVar^+_{3}(\pants)$. Let $\conjclass$ be tuple of conjugacy classes. Recall that a symplectic leaf $\SymLeaf_L$ is mapped through $\rechol$ to the relative character variety $\CharVar^+_{3,\conjclass}(\pants)$ by Lemma \ref{lemma : symplectic leaves are exactly relative character varieties}. Since the fundamental group of $\pants$ is generated by the three boundary curves, any flow $\Phi$ in  $\CharVar^+_{3,\conjclass}(\pants)$ can be written as follows. Let $[\rho]\in\CharVar^+_{3,\conjclass}(\pants)$. Then the flow $\Phi^t([\rho])$ is covered by a path of representations described by conjugations along the boundary. That is, there exist paths $g_t^\alpha,g_t^\beta,g_t^\gamma\in\PSL 3$ depending on $\rho$ such that the flow
\begin{equation}\label{eq : rhot is conjugation}
\rho_t = \begin{cases}
 	\alpha\mapsto g_t^\alpha\rho(\alpha)(g_t^\alpha)^{-1}\\
 	\beta\mapsto g_t^\beta\rho(\beta)(g_t^\beta)^{-1}\\
 	\gamma\mapsto g_t^\gamma\rho(\gamma)(g_t^\gamma)^{-1}
 \end{cases}
\end{equation}
covers the flow $\Phi$ on $\CharVar^+_{3,\conjclass}(\pants)$.\\

In this section, we find such conjugating matrices for the eruption and hexagon flows when we restrict to the unipotent locus.\\

The two following results are once again, computations using Mathematica. The conjugating matrices were found by solving matrix equations, and they can be easily checked, since the flows are explicit. The first result finds conjugating matrices for the eruption flow. The matrices were found by finding solutions to the expression in \eqref{eq : rhot is conjugation}. Instead of showing the solution, we simply verify in Sections 8 and 9 of the Mathematica code the following two theorems.

\begin{theorem}\label{thm : conjugating matrices for eruption}
	The flow $\Phi^t_\hexagon = \rechol\circ\HmFlow^t_\hexagon$ on $\unipotentLocus\subset \CharVar_3^+(\pants)$ is given by the following conjugations. For $(\sigma_1,\tau_1)\in\unipotentLocus$, let
	\[
	\zeta^\alpha_t = \begin{pmatrix}
		\frac{e^{-2t/3}(1+e^t\tau_1)^{2/3}}{(1+\tau_1)^{1/3}} & 0 & 0\\
		0 & \frac{e^{t/3}(1+\tau_1)^{1/3}}{(1+e^t\tau_1)^{1/3}} & 0\\
		0 & 0 &\frac{e^{t/3}(1+\tau_1)^{1/3}}{(1+e^t\tau_1)^{1/3}}
	\end{pmatrix},
	\]
	
	\begin{align*}
\zeta^\beta_t &= \frac{(1+\tau_1)^{1/3}}{e^{t/3}\tau_1(1+e^t\tau_1)^{1/3}}\cdot \\&\left(
\begin{array}{ccc}
 0 & -\left(\left(\sigma _1-1\right) \left(e^t \tau _1+1\right)\right) &
   -\frac{\left(\sigma _1 \tau _1+\sigma _1-1\right) \left(e^t \tau
   _1+1\right)}{\tau _1+1} \\
 0 & \sigma _1+\sigma _1 e^t \tau _1-e^t \tau _1 & \frac{\sigma _1
   \left(\tau _1+1\right) \left(e^t \tau _1+1\right)-e^t \tau _1+\tau
   _1}{\tau _1+1} \\
 -e^t \tau _1^2 & -\sigma _1 \left(e^t \tau _1+1\right)-e^t \tau _1
   \left(\tau _1+1\right) & -\sigma _1 \left(e^t \tau _1+1\right)-\frac{\tau
   _1 \left(e^t \left(2 \tau _1+1\right)+1\right)}{\tau _1+1} \\
\end{array}
\right)
	\end{align*}
	and
	\[
	\zeta^\gamma_t = \begin{pmatrix}
		\frac{(1+\tau_1)^{1/3}}{(1+e^t\tau_1)^{1/3}} & 0 & 0\\
		0 & \frac{(1+\tau_1)^{1/3}}{(1+e^t\tau_1)^{1/3}} & 0\\
		0 & 0 & \frac{(1+\tau_1)^{2/3}}{(1+e^t\tau_1)^{2/3}}.
	\end{pmatrix}.
	\]
	Then the flow 
	\[
\rho_t = \begin{cases}
 	\alpha\mapsto \zeta_t^\alpha\rho(\alpha)(\zeta_t^\alpha)^{-1}\\
 	\beta\mapsto \zeta_t^\beta\rho(\beta)(\zeta_t^\beta)^{-1}\\
 	\gamma\mapsto \zeta_t^\gamma\rho(\gamma)(\zeta_t^\gamma)^{-1}
 \end{cases}
\]
covers the flow $\Phi^t_\hexagon$.	
\end{theorem}

The next result is analogous to the above theorem, and concerns the hexagon flow.
\begin{theorem}\label{thm : conjugating matrices for hexagon}
	The flow $\Phi^t_\eruption = \rechol\circ\HmFlow^t_\eruption$ on $\unipotentLocus\subset \CharVar_3^+(\pants)$ is given by the following conjugations. For $(\sigma_1,\tau_1)\in\unipotentLocus$, let
	\[
	\eta^\alpha_t = \left(
\begin{array}{ccc}
 \frac{e^{-t/3} \left(\sigma _1 e^t+1\right)}{\sigma _1+1} & \frac{\sigma _1
   e^{-t/3} \left(e^t-1\right) \left(\tau _1+1\right)}{\left(\sigma
   _1+1\right) \tau _1} & 0 \\
 0 & e^{-t/3} & 0 \\
 0 & 0 & \frac{\left(\sigma _1+1\right) e^{2 t/3}}{\sigma _1 e^t+1} \\
\end{array}
\right)
	\]
	and
	\[
	\eta^\gamma_t = \left(
\begin{array}{ccc}
 \frac{\left(\sigma _1+1\right) e^{2 t/3}}{\sigma _1 e^t+1} & 0 & 0 \\
 0 & e^{-t/3} & 0 \\
 0 & \frac{\sigma _1 e^{-t/3} \left(e^t-1\right) \left(\tau
   _1+1\right)}{\sigma _1+1} & \frac{e^{-t/3} \left(\sigma _1
   e^t+1\right)}{\sigma _1+1} \\
\end{array}
\right).
	\]
	We define the matrix $\eta^\beta_t$ in terms of its entries below:
	\begin{gather*}
		(\eta^\beta_t)_{11} = 0,\\
		(\eta^\beta_t)_{12}= \frac{e^{-t/3} \left(\tau _1+1\right) \left(\sigma _1^2
   e^t-1\right)}{\left(\sigma _1+1\right) \tau _1}\\
		(\eta^\beta_t)_{13}=\frac{e^{-t/3} \left(\sigma _1^2 e^t \left(\tau _1+1\right)+\sigma _1 e^t
   \tau _1-1\right)}{\left(\sigma _1+1\right) \tau _1}\\
		(\eta^\beta_t)_{21}=0\\
		(\eta^\beta_t)_{22}=-\frac{e^{-t/3} \left(\sigma _1+\sigma _1^2 e^t \left(\tau _1+1\right)-\tau
   _1\right)}{\left(\sigma _1+1\right) \tau _1}\\
		(\eta^\beta_t)_{23}=-\frac{\sigma _1 e^{-t/3} \left(\sigma _1 e^t \left(\tau _1+1\right)+e^t
   \tau _1+1\right)}{\left(\sigma _1+1\right) \tau _1}\\
		(\eta^\beta_t)_{31}=\frac{\left(\sigma _1+1\right) e^{2 t/3} \tau _1}{\sigma _1 e^t+1}\\
		(\eta^\beta_t)_{32}=\frac{\left(\tau _1+1\right) \left(\sigma _1^2 e^t \left(\tau
   _1+2\right)+\sigma _1 \left(2 e^t \tau _1+1\right)+\sigma _1^3 e^{2
   t}+e^t \tau _1\right)}{\left(\sigma _1+1\right) e^{t/3} \tau _1 \left(\sigma _1
   e^t+1\right)}\\
		(\eta^\beta_t)_{33}=\frac{\sigma _1^3 e^{2 t} \left(\tau _1+1\right)+\sigma _1^2 e^t
   \left(\left(e^t+3\right) \tau _1+2\right)+\sigma _1 \left(\left(4
   e^t+1\right) \tau _1+1\right)+\left(e^t+1\right) \tau
   _1}{\left(\sigma _1+1\right) e^{t/3} \tau _1 \left(\sigma _1
   e^t+1\right)}.
	\end{gather*}
	Then the flow \[
\rho_t = \begin{cases}
 	\alpha\mapsto \eta_t^\alpha\rho(\alpha)(\eta_t^\alpha)^{-1}\\
 	\beta\mapsto \eta_t^\beta\rho(\beta)(\eta_t^\beta)^{-1}\\
 	\gamma\mapsto \eta_t^\gamma\rho(\gamma)(\eta_t^\gamma)^{-1}
 \end{cases}
\]
covers the flow $\Phi^t_\eruption$.
	\end{theorem}

\appendix
\section{Proof of Theorem \ref{thm intro : symmetrized trace}, Part \ref{thm intro : periodicity for symmetrized trace}}\label{app : proof of periodicity for symmetrixed trace}
In this appendix, we give a proof of Theorem \ref{thm intro : symmetrized trace}, Part \ref{thm intro : periodicity for symmetrized trace}, which we restate here.
\begin{theorem}\label{thm : app symmetric trace}
	Let $\figeight = \alpha\gamma^{-1}$ be a figure eight curve on a pair of pants $P$ and let $\SymLeaf$ be a symplectic leaf in $\widehat\DefSpace(\pants)$. The restriction of the function $\trace_{\figeight}+\trace_{\figeight^{-1}}\big|_{\SymLeaf}\colon\SymLeaf\to\R$ attains a unique minimum. Moreover, every orbit of the Hamiltonian flow of  $\trace_{\figeight}+\trace_{\figeight^{-1}}\big|_{\SymLeaf}$ is periodic, and there is a unique fixed point.
\end{theorem}
Here we prove that the symmetrized trace is a proper function on symplectic leaves, and that it is strictly convex along any mixed flow. The proof of the theorem is then identical to the proof of Theorem \ref{thm : periodicity of figure 8 in general} in Section \ref{sec : proof of main thm}, and we therefore leave it out.\\

 We begin by computing the trace function associated to $\figeight^{-1}$ in coordinates.

\begin{lemma}\label{lemma : trace fig eight inverse formula}
	Let $\figeight = \alpha\gamma^{-1}$ and $L = (\ell_{\alpha,1},\ell_{\alpha,2},\ell_{\beta,1},\ell_{\beta,2},\ell_{\gamma,1},\ell_{\gamma,2})\in\R^6_{>0}$ be a length vector defining a symplectic leaf $\SymLeaf_L$. Then
	\begin{gather*}
		\trace_{\figeight^{-1}}\big|_{\SymLeaf_L} = \frac{1}{\sigma _1^2 \tau _1 \ell _{\alpha ,1} \left(\ell _{\alpha ,2} \ell _{\beta ,1} \ell
   _{\beta ,2} \ell _{\gamma ,1} \ell _{\gamma ,2}\right){}^{2/3}}\cdot\\
   \bigg(\left(\tau _1+1\right) \left(\sigma _1+\tau _1+1\right) \ell _{\alpha ,1}^2 \ell _{\gamma
   ,1} \sqrt[3]{\ell _{\alpha ,2}^2 \ell _{\beta ,1} \ell _{\beta ,2}^4}+\sigma_1\ell_{\alpha,1}\bigg(\tau _1 \left(\tau _1+1\right) \ell _{\gamma ,1} \sqrt[3]{\ell _{\alpha ,2}^2 \ell
   _{\beta ,1} \ell _{\beta ,2}^4}+\\
   \sigma _1 \tau _1^2 \ell _{\gamma ,2} \sqrt[3]{\ell _{\alpha ,2}^2 \ell _{\beta ,1} \ell
   _{\beta ,2}}+\sigma _1 \left(\sigma _1+1\right) \ell _{\gamma ,1} \ell _{\gamma ,2}
   \sqrt[3]{\ell _{\alpha ,2}^2 \ell _{\beta ,1}^4 \ell _{\beta ,2}}\bigg)+\\\sigma_1\bigg(\sigma _1 \tau _1 \sqrt[3]{\ell _{\alpha ,2}} \left(\ell _{\alpha ,1} \ell _{\gamma ,1}
   \ell _{\gamma ,2}\right){}^{2/3} \left(\sigma _1 \ell _{\beta ,1}+\tau _1 \ell _{\beta
   ,2}\right)+\tau _1 \left(\tau _1+1\right) \ell _{\beta ,2} \sqrt[3]{\ell _{\alpha
   ,1}^5 \ell _{\alpha ,2} \ell _{\gamma ,1}^2 \ell _{\gamma ,2}^2}+\\
   \left(\sigma _1+\tau _1+1\right) \ell _{\beta ,1} \ell _{\beta ,2} \sqrt[3]{\ell _{\alpha
   ,1}^5 \ell _{\alpha ,2} \ell _{\gamma ,1}^2 \ell _{\gamma ,2}^2}+\sigma_1\tau_1\bigg(\ell _{\gamma ,2} \sqrt[3]{\ell _{\alpha ,2}^2 \ell _{\beta ,1} \ell _{\beta ,2}}+\\
   \sqrt[3]{\ell _{\alpha ,1}^4 \ell _{\beta ,1}^2 \ell _{\beta ,2}^2 \ell _{\gamma ,1} \ell
   _{\gamma ,2}}+\sqrt[3]{\ell _{\alpha ,1}^4 \ell _{\beta ,1}^2 \ell _{\beta ,2}^2 \ell
   _{\gamma ,1}^4 \ell _{\gamma ,2}}+\sqrt[3]{\ell _{\alpha ,1}^4 \ell _{\beta ,1}^2 \ell
   _{\beta ,2}^2 \ell _{\gamma ,1} \ell _{\gamma ,2}^4}\bigg)+\\
   \ell_{\alpha,2}\bigg(\sigma _1^2 \tau _1 \sqrt[3]{\ell _{\alpha ,1} \ell _{\beta ,1}^2 \ell _{\beta ,2}^2 \ell
   _{\gamma ,1} \ell _{\gamma ,2}^4}+(\tau _1+1)\bigg(\tau _1 \sqrt[3]{\ell _{\alpha ,1}^4 \ell _{\beta ,1}^2 \ell _{\beta ,2}^2 \ell _{\gamma
   ,1} \ell _{\gamma ,2}}+\sqrt[3]{\ell _{\alpha ,1}^4 \ell _{\beta ,1}^2 \ell _{\beta
   ,2}^2 \ell _{\gamma ,1}^4 \ell _{\gamma ,2}}\bigg)+\\
   \sigma_1\bigg(\tau _1^2 \sqrt[3]{\ell _{\alpha ,1} \ell _{\beta ,1}^2 \ell _{\beta ,2}^2 \ell _{\gamma
   ,1} \ell _{\gamma ,2}}+\sqrt[3]{\ell _{\alpha ,1}^4 \ell _{\beta ,1}^2 \ell _{\beta
   ,2}^2 \ell _{\gamma ,1}^4 \ell _{\gamma ,2}}+\tau_1\bigg(\sqrt[3]{\ell _{\alpha ,1} \ell _{\beta ,1}^2 \ell _{\beta ,2}^2 \ell _{\gamma ,1} \ell
   _{\gamma ,2}}+\\
   \sqrt[3]{\ell _{\alpha ,1}^4 \ell _{\beta ,1}^2 \ell _{\beta ,2}^2 \ell _{\gamma ,1} \ell
   _{\gamma ,2}}+\sqrt[3]{\ell _{\alpha ,1} \ell _{\beta ,1}^2 \ell _{\beta ,2}^2 \ell
   _{\gamma ,1}^4 \ell _{\gamma ,2}}+\sqrt[3]{\ell _{\alpha ,1}^4 \ell _{\beta ,1}^2 \ell
   _{\beta ,2}^2 \ell _{\gamma ,1}^4 \ell _{\gamma ,2}}+\\
   \sqrt[3]{\ell _{\alpha ,1} \ell _{\beta ,1}^2 \ell _{\beta ,2}^2 \ell _{\gamma ,1} \ell
   _{\gamma ,2}^4}+\sqrt[3]{\ell _{\alpha ,1}^4 \ell _{\beta ,1}^2 \ell _{\beta ,2}^2
   \ell _{\gamma ,1} \ell _{\gamma ,2}^4}\bigg)\bigg)\bigg)\bigg)\bigg).
	\end{gather*}
\end{lemma}
\begin{proof}
	The proof of this fact is a computation, found in Section 4 of the Mathematica code. The function \texttt{InvtraceFigure8} will give the above output with the length vectors as input.
\end{proof}

We now move on to prove properness of the symmetrized trace function.

\begin{proposition}
	Let $\figeight = \alpha\gamma^{-1}$ and $L = (\ell_{\alpha,1},\ell_{\alpha,2},\ell_{\beta,1},\ell_{\beta,2},\ell_{\gamma,1},\ell_{\gamma,2})\in\R^6_{>0}$ be a length vector defining a symplectic leaf $\SymLeaf_L$. Then the function $\trace_\figeight+\trace_{\figeight^{-1}}\big|_{\SymLeaf_L}\colon\SymLeaf_L\to\R$ is proper. In particular, it realizes a minimum in $\SymLeaf_L$.
\end{proposition}
\begin{proof}
	By Lemma \ref{lemma : trace fig eight inverse formula}, both $\trace_\figeight$ and $\trace_{\figeight^{-1}}$ are positive functions. In Proposition \ref{prop : trace of figure 8 is proper}, we proved that $\trace_\figeight$ is proper. Hence we only need to show that $\trace_{\figeight^{-1}}$ is proper.\\
	
	Let $(\sigma_{1,n},\tau_{1,n})_{n\in\N}$ be a sequence such that as $n$ goes to $\infty$, $(\sigma_{1,n},\tau_{1,n})$ goes to a tuple in $\{(\infty,\infty), (0,\infty),(\infty,0),(x,0),(0,y)\st x,y\geq 0\}$. We need to show that in any of these cases, $\trace_\figeight(\sigma_{1,n},\tau_{1,n})\to\infty$ as $n\to\infty$. Notice that all the variables are positive and that all the signs on the monomials are positive as well. Hence, it is enough to find terms in the expression in Lemma \ref{lemma : trace fig eight inverse formula} that diverges along any of the above sequences.
	
	\begin{enumerate}
		\item Assume $(\sigma_{1,n},\tau_{1,n})\to(\infty,\infty)$, then the term
	\[
	\frac{\sigma _{1,n}^2 \tau _{1,n} \sqrt[3]{\ell _{\alpha ,2}} \left(\ell _{\alpha ,1} \ell _{\gamma ,1}
   \ell _{\gamma ,2}\right){}^{2/3} \left(\sigma _{1,n} \ell _{\beta ,1}+\tau _{1,n} \ell _{\beta
   ,2}\right)}{\sigma _{1,n}^2 \tau _{1,n} \ell _{\alpha ,1} \left(\ell _{\alpha ,2} \ell _{\beta ,1} \ell
   _{\beta ,2} \ell _{\gamma ,1} \ell _{\gamma ,2}\right){}^{2/3}}\xrightarrow{n\to\infty}\infty.
	\]
	\item Assume $(\sigma_{1,n},\tau_{1,n})\to(0,\infty)$, then the term
	\[
	\frac{\left(\tau _{1,n}+1\right) \left(\sigma _{1,n}+\tau _{1,n}+1\right) \ell _{\alpha ,1}^2 \ell _{\gamma
   ,1} \sqrt[3]{\ell _{\alpha ,2}^2 \ell _{\beta ,1} \ell _{\beta ,2}^4}}{\sigma _{1,n}^2 \tau _{1,n} \ell _{\alpha ,1} \left(\ell _{\alpha ,2} \ell _{\beta ,1} \ell
   _{\beta ,2} \ell _{\gamma ,1} \ell _{\gamma ,2}\right){}^{2/3}}\xrightarrow{n\to\infty}\infty.
	\]
	\item Assume $(\sigma_{1,n},\tau_{1,n})\to(\infty,0)$, then the term
	\[
	\frac{\sigma _{1,n}^2 \tau _{1,n} \sqrt[3]{\ell _{\alpha ,2}} \left(\ell _{\alpha ,1} \ell _{\gamma ,1}
   \ell _{\gamma ,2}\right){}^{2/3} \left(\sigma _{1,n} \ell _{\beta ,1}+\tau _{1,n} \ell _{\beta
   ,2}\right)}{\sigma _{1,n}^2 \tau _{1,n} \ell _{\alpha ,1} \left(\ell _{\alpha ,2} \ell _{\beta ,1} \ell
   _{\beta ,2} \ell _{\gamma ,1} \ell _{\gamma ,2}\right){}^{2/3}}\xrightarrow{n\to\infty}\infty.
	\]
\item Assume $(\sigma_{1,n},\tau_{1,n})\to(x,0)$ or $(0,y)$ for $x,y\geq 0$, then the term
\[
\frac{\ell _{\alpha ,1}^2 \ell _{\gamma
   ,1} \sqrt[3]{\ell _{\alpha ,2}^2 \ell _{\beta ,1} \ell _{\beta ,2}^4}}{\sigma _{1,n}^2 \tau _{1,n} \ell _{\alpha ,1} \left(\ell _{\alpha ,2} \ell _{\beta ,1} \ell
   _{\beta ,2} \ell _{\gamma ,1} \ell _{\gamma ,2}\right){}^{2/3}}\xrightarrow{n\to\infty}\infty.
\]
	\end{enumerate}
	This finishes to proof.
\end{proof}

We finish by proving convexity.

\begin{proposition}
	Let $\figeight = \alpha\gamma^{-1}$ and $L = (\ell_{\alpha,1},\ell_{\alpha,2},\ell_{\beta,1},\ell_{\beta,2},\ell_{\gamma,1},\ell_{\gamma,2})\in\R^6_{>0}$ be a length vector defining a symplectic leaf $\SymLeaf_L$. The function $\trace_{\figeight}+\trace_{\figeight^{-1}}\big|_{\SymLeaf_L}\colon\SymLeaf_L\to\R$ is strictly convex along any mixed flow.
\end{proposition}
\begin{proof}
	Since we proved in Theorem \ref{thm : convexity of the trace function along mixed flows} that $\trace_\figeight$ is strictly convex, we only have to show that the function $\trace_{\figeight^{-1}}$ is strictly convex along any mixed flow. Consider first the mixed flow $\Psi^t_a=\HmFlow^{at}_\hexagon\circ\HmFlow_\eruption^t$. By Section 6 of the Mathematica code, using the function \texttt{InvtraceFigure8} and \texttt{testfMix1} as input, we obtain that 
	\begin{gather*}
		\frac{\partial^2}{\partial t^2}(\trace_{\figeight^{-1}}(\Psi^t_a(\sigma_1,\tau_1))) = \frac{e^{-t(2a+1)}}{\sigma _1^2 \tau _1 \ell _{\alpha ,1} \left(\ell _{\alpha ,2} \ell _{\beta ,1} \ell
   _{\beta ,2} \ell _{\gamma ,1} \ell _{\gamma ,2}\right){}^{2/3}}\cdot\\
   \bigg(\ell _{\alpha ,1}^2 \ell _{\gamma ,1} \sqrt[3]{\ell _{\alpha ,2}^2 \ell _{\beta ,1} \ell
   _{\beta ,2}^4} \left(a^2 \sigma _1 \tau _1 e^{a t+t}+8 a^2 e^t \tau _1+(a+1)^2 \sigma
   _1 e^{a t}+(1-2 a)^2 e^{2 t} \tau _1^2+(2 a+1)^2\right)+\\
   \sigma _1 e^{a t} \ell _{\alpha ,1}\bigg(a^2 e^t \tau _1 \ell _{\gamma ,1} \sqrt[3]{\ell _{\alpha ,2}^2 \ell _{\beta ,1} \ell
   _{\beta ,2}^4}+(a-1)^2 e^{2 t} \tau _1^2 \ell _{\gamma ,1} \sqrt[3]{\ell _{\alpha
   ,2}^2 \ell _{\beta ,1} \ell _{\beta ,2}^4}+\\
   \sigma _1 \tau _1^2 e^{(a+2) t} \ell _{\gamma ,2} \sqrt[3]{\ell _{\alpha ,2}^2 \ell
   _{\beta ,1} \ell _{\beta ,2}}+(a-1)^2 \sigma _1^2 e^{2 a t} \ell _{\gamma ,1} \ell
   _{\gamma ,2} \sqrt[3]{\ell _{\alpha ,2}^2 \ell _{\beta ,1}^4 \ell _{\beta ,2}}+\sigma
   _1 e^{a t} \ell _{\gamma ,1} \ell _{\gamma ,2} \sqrt[3]{\ell _{\alpha ,2}^2 \ell
   _{\beta ,1}^4 \ell _{\beta ,2}}\bigg)+\\
   \sigma_1\bigg((a+1)^2 \sigma _1^2 \tau _1^2 e^{(3 a+2) t} \ell _{\gamma ,2} \sqrt[3]{\ell _{\alpha
   ,2}^2 \ell _{\beta ,1} \ell _{\beta ,2}}+\\
   \sigma _1 \tau _1^2 e^{2 (a+1) t} \left(\ell _{\alpha ,2} \sqrt[3]{\ell _{\alpha ,1} \ell
   _{\beta ,1}^2 \ell _{\beta ,2}^2 \ell _{\gamma ,1} \ell _{\gamma ,2}}+\ell _{\beta ,2}
   \sqrt[3]{\ell _{\alpha ,1}^2 \ell _{\alpha ,2} \ell _{\gamma ,1}^2 \ell _{\gamma
   ,2}^2}\right)+\\
   (a+1)^2 e^{a t} \left(\ell _{\alpha ,2} \sqrt[3]{\ell _{\alpha ,1}^4 \ell _{\beta ,1}^2
   \ell _{\beta ,2}^2 \ell _{\gamma ,1}^4 \ell _{\gamma ,2}}+\ell _{\beta ,1} \ell
   _{\beta ,2} \sqrt[3]{\ell _{\alpha ,1}^5 \ell _{\alpha ,2} \ell _{\gamma ,1}^2 \ell
   _{\gamma ,2}^2}\right)+\\
   \sigma _1 e^{2 a t} \left(\ell _{\alpha ,2} \sqrt[3]{\ell _{\alpha ,1}^4 \ell _{\beta
   ,1}^2 \ell _{\beta ,2}^2 \ell _{\gamma ,1}^4 \ell _{\gamma ,2}}+\ell _{\beta ,1} \ell
   _{\beta ,2} \sqrt[3]{\ell _{\alpha ,1}^5 \ell _{\alpha ,2} \ell _{\gamma ,1}^2 \ell
   _{\gamma ,2}^2}\right)+\\
   a^2 \sigma _1^2 \tau _1 e^{3 a t+t} \sqrt[3]{\ell _{\alpha ,2}} \left(\ell _{\beta ,1}
   \left(\ell _{\alpha ,1} \ell _{\gamma ,1} \ell _{\gamma
   ,2}\right){}^{2/3}+\sqrt[3]{\ell _{\alpha ,1} \ell _{\alpha ,2}^2 \ell _{\beta ,1}^2
   \ell _{\beta ,2}^2 \ell _{\gamma ,1} \ell _{\gamma ,2}^4}\right)+\\
   a^2 \tau _1 e^{a t+t} \left(\left(\ell _{\beta ,1}+1\right) \ell _{\beta ,2}
   \sqrt[3]{\ell _{\alpha ,1}^5 \ell _{\alpha ,2} \ell _{\gamma ,1}^2 \ell _{\gamma
   ,2}^2}+\ell _{\alpha ,2} \left(\sqrt[3]{\ell _{\alpha ,1}^4 \ell _{\beta ,1}^2 \ell
   _{\beta ,2}^2 \ell _{\gamma ,1} \ell _{\gamma ,2}}+\sqrt[3]{\ell _{\alpha ,1}^4 \ell
   _{\beta ,1}^2 \ell _{\beta ,2}^2 \ell _{\gamma ,1}^4 \ell _{\gamma ,2}}\right)\right)\bigg)\bigg)\bigg)
	\end{gather*}
	This quantity is positive for any $a\in\R$ and any coordinates.\\
	
	Now consider the mixed flow $\Psi^t_a=\HmFlow^{t}_\hexagon\circ\HmFlow_\eruption^{at}$. From Section 6 of the Mathematica code this time using \texttt{testfMix2} as input, we obtain that 
	\begin{gather*}
		\frac{\partial^2}{\partial t^2}(\trace_{\figeight^{-1}}(\Psi^t_a(\sigma_1,\tau_1))) = \frac{e^{-t(2+a)}}{\sigma _1^2 \tau _1 \ell _{\alpha ,1} \left(\ell _{\alpha ,2} \ell _{\beta ,1} \ell
   _{\beta ,2} \ell _{\gamma ,1} \ell _{\gamma ,2}\right){}^{2/3}}\cdot\\
   \bigg(\bigg(\ell _{\alpha ,1}^2 \ell _{\gamma ,1} \sqrt[3]{\ell _{\alpha ,2}^2 \ell _{\beta ,1} \ell
   _{\beta ,2}^4} \left(\sigma _1 \tau _1 e^{a t+t}+(a+1)^2 \sigma _1 e^t+(a-2)^2 \tau
   _1^2 e^{2 a t}+8 \tau _1 e^{a t}+(a+2)^2\right)+\\
   e^t\sigma_1\ell_{\alpha,1}\bigg((a-1)^2 \tau _1^2 e^{2 a t} \ell _{\gamma ,1} \sqrt[3]{\ell _{\alpha ,2}^2 \ell _{\beta
   ,1} \ell _{\beta ,2}^4}+\tau _1 e^{a t} \ell _{\gamma ,1} \sqrt[3]{\ell _{\alpha ,2}^2
   \ell _{\beta ,1} \ell _{\beta ,2}^4}+\\
   a^2 \sigma _1 \tau _1^2 e^{2 a t+t} \ell _{\gamma ,2} \sqrt[3]{\ell _{\alpha ,2}^2 \ell
   _{\beta ,1} \ell _{\beta ,2}}+a^2 \sigma _1 e^t \ell _{\gamma ,1} \ell _{\gamma ,2}
   \sqrt[3]{\ell _{\alpha ,2}^2 \ell _{\beta ,1}^4 \ell _{\beta ,2}}+\\
   (a-1)^2 \sigma _1^2 e^{2 t} \ell _{\gamma ,1} \ell _{\gamma ,2} \sqrt[3]{\ell _{\alpha
   ,2}^2 \ell _{\beta ,1}^4 \ell _{\beta ,2}}\bigg)+\sigma_1\bigg((a+1)^2 \sigma _1^2 \tau _1^2 e^{(2 a+3) t} \ell _{\gamma ,2} \sqrt[3]{\ell _{\alpha
   ,2}^2 \ell _{\beta ,1} \ell _{\beta ,2}}+\\
   a^2 \sigma _1 \tau _1^2 e^{2 (a+1) t} \left(\ell _{\alpha ,2} \sqrt[3]{\ell _{\alpha ,1}
   \ell _{\beta ,1}^2 \ell _{\beta ,2}^2 \ell _{\gamma ,1} \ell _{\gamma ,2}}+\ell
   _{\beta ,2} \sqrt[3]{\ell _{\alpha ,1}^2 \ell _{\alpha ,2} \ell _{\gamma ,1}^2 \ell
   _{\gamma ,2}^2}\right)+\\
   (a-1)^2 \tau _1^2 e^{2 a t+t} \left(\ell _{\alpha ,2} \sqrt[3]{\ell _{\alpha ,1}^4 \ell
   _{\beta ,1}^2 \ell _{\beta ,2}^2 \ell _{\gamma ,1} \ell _{\gamma ,2}}+\ell _{\beta ,2}
   \sqrt[3]{\ell _{\alpha ,1}^5 \ell _{\alpha ,2} \ell _{\gamma ,1}^2 \ell _{\gamma
   ,2}^2}\right)+\\
   (a+1)^2 e^t \left(\ell _{\alpha ,2} \sqrt[3]{\ell _{\alpha ,1}^4 \ell _{\beta ,1}^2 \ell
   _{\beta ,2}^2 \ell _{\gamma ,1}^4 \ell _{\gamma ,2}}+\ell _{\beta ,1} \ell _{\beta ,2}
   \sqrt[3]{\ell _{\alpha ,1}^5 \ell _{\alpha ,2} \ell _{\gamma ,1}^2 \ell _{\gamma
   ,2}^2}\right)+\\
   a^2 \sigma _1 e^{2 t} \left(\ell _{\alpha ,2} \sqrt[3]{\ell _{\alpha ,1}^4 \ell _{\beta
   ,1}^2 \ell _{\beta ,2}^2 \ell _{\gamma ,1}^4 \ell _{\gamma ,2}}+\ell _{\beta ,1} \ell
   _{\beta ,2} \sqrt[3]{\ell _{\alpha ,1}^5 \ell _{\alpha ,2} \ell _{\gamma ,1}^2 \ell
   _{\gamma ,2}^2}\right)+\\
   \sigma _1^2 \tau _1 e^{(a+3) t} \sqrt[3]{\ell _{\alpha ,2}} \left(\ell _{\beta ,1}
   \left(\ell _{\alpha ,1} \ell _{\gamma ,1} \ell _{\gamma
   ,2}\right){}^{2/3}+\sqrt[3]{\ell _{\alpha ,1} \ell _{\alpha ,2}^2 \ell _{\beta ,1}^2
   \ell _{\beta ,2}^2 \ell _{\gamma ,1} \ell _{\gamma ,2}^4}\right)+\\
   e^{t(1+a)}\tau_1\bigg(\left(\ell _{\beta ,1}+1\right) \ell _{\beta ,2} \sqrt[3]{\ell _{\alpha ,1}^5 \ell
   _{\alpha ,2} \ell _{\gamma ,1}^2 \ell _{\gamma ,2}^2}+\\
   \ell _{\alpha ,2} \left(\sqrt[3]{\ell _{\alpha ,1}^4 \ell _{\beta ,1}^2 \ell _{\beta
   ,2}^2 \ell _{\gamma ,1} \ell _{\gamma ,2}}+\sqrt[3]{\ell _{\alpha ,1}^4 \ell _{\beta
   ,1}^2 \ell _{\beta ,2}^2 \ell _{\gamma ,1}^4 \ell _{\gamma ,2}}\right)\bigg)\bigg)\bigg)\bigg).
		\end{gather*}
		This quantity is also positive for any $a\in\R$ and any coordinates. Hence, the function $t\mapsto \trace_\figeight(\Psi_a^t(\sigma_1,\tau_1))+\trace_{\figeight^{-1}}(\Psi_a^t(\sigma_1,\tau_1))$ is strictly convex for any $(\sigma_1,\tau_1)$.
\end{proof}

\section{Proof of Lemma \ref{lemma : expression for FG cross ratios}}\label{app : proof of cross ratios lemma}
\begin{proof}
	Since the cross ratio is a projective invariant, we can assume that the flags are in the following positions:
	\begin{align*}
		p_1 = \Vvector{1}{0}{0},\quad\ell_1 = \Lline{0}{1}{1}\\
		p_2 = \Vvector{0}{1}{0},\quad\ell_2 = \Lline{1}{0}{1}\\
		p_3 = \Vvector{0}{0}{1},\quad\ell_3 = \Lline{1}{x}{0}\\
		p_4 = \Vvector{1}{y}{z},\quad\ell_4 = \Lline{-wz-y}{1}{w}
	\end{align*}
	where $x,y,z,w\in\R\smallsetminus\{0\}$. The lines appearing in the first cross ratio in the lemma are given by
	 \begin{align*}
	 	\overline{p_1p_2}= \Lline{0}{0}{1},\quad \overline{p_1p_3} = \Lline{0}{1}{0},\quad\overline{p_1p_4} = \Lline{0}{z}{-y}.
	 \end{align*}
	 To take the cross ratio of the four lines, we choose the line $h = \overline{p_1p_4}$ to identify it with $\R\cup\{\infty\}$. We have that
	 \[
	 h = \Lline{-z}{0}{1}=\left\{\Vvector{1}{b}{z}\st b\in\R \right\}
	 \]
	 and we choose the isomorphism
	 \begin{align}\label{eq : identification with RP1 h}
	  h&\to\R\\
	  \Vvector{1}{b}{z}&\mapsto b.\nonumber
	 \end{align}
	 The intersections of the relevant lines with $h$ are
	 \[
	 \ell_1\cap h = \Vvector{1}{-z}{z},\, \overline{p_1p_2}\cap h = \Vvector{0}{1}{0}, \,\overline{p_1p_3}\cap h = \Vvector{1}{0}{z},\,  \overline{p_1p_4}\cap h =\Vvector{1}{y}{z}. 
	 \] 
Under the above identification of $h$ with $\R\cup\{\infty\}$, we have that
	 \[
	 \ell_1\cap h\mapsto -z, \quad \overline{p_1p_2}\cap h\mapsto \infty,\quad \overline{p_1p_3}\cap h\mapsto 0,\quad \overline{p_1p_4}\cap h\mapsto y.
	 \]
	 Their cross ratio is therefore
	 \[
	 	\Cr(\ell_1,\overline{p_1p_2},\overline{p_1p_3},\overline{p_1p_4}) = \frac{-y}{y+z}.
	 \]
	 On the other hand
	 \[
	 \Cr_1(F_1,F_2,F_3,F_4) = -\frac{\left(\Lline{0}{1}{1}\Vvector{0}{1}{0}\right)\left(\Lline{0}{1}{0}\Vvector{1}{y}{z}\right)}{\left(\Lline{0}{1}{1}\Vvector{1}{y}{z}\right)\left(\Lline{0}{1}{0}\Vvector{0}{1}{0}\right)}=\frac{-y}{y+z}
	 \]
	 as desired.\\
	 
	 For the second cross ratio, we do a similar computation. The lines appearing in the second cross ratio are
	 \[
	 \overline{p_3p_1} = \Lline{1}{0}{0},\quad \overline{p_1p_3} = \Lline{0}{1}{0},\quad \overline{p_1p_4} = \Lline{y}{-1}{0}.
	 \]
	 Using the same projective line $h$ as above and using the identification in \eqref{eq : identification with RP1 h}, we have that
	 \begin{gather*}
	 	\ell_3\cap h = \Vvector{1}{-1/x}{z}\mapsto \frac{-1}{x},\qquad \overline{p_3p_4}\cap h = \Vvector{1}{y}{z}\mapsto y\\
	 	\overline{p_3p_1}\cap h = \Vvector{1}{0}{z}\mapsto 0,\qquad \overline{p_3p_2}\cap h = \Vvector{0}{1}{0} \mapsto \infty.
	 \end{gather*}
	 Thus
	 \[
	 \Cr(\ell_3,\overline{p_3p_4},\overline{p_3p_1},\overline{p_3p_2}) = \frac{-1-xy}{xy}.
	 \]
	 On the other hand,
	 \[
	 \Cr_2(F_1,F_2,F_3,F_4) = -\frac{\left(\Lline{1}{x}{0}\Vvector{1}{y}{z}\right)\left(\Lline{0}{1}{0}\Vvector{0}{1}{0}\right)}{\left(\Lline{1}{x}{0}\Vvector{0}{1}{0}\right)\left(\Lline{0}{1}{0}\Vvector{1}{y}{z}\right)} = \frac{-1-xy}{xy}
	 \]
	 as desired.
\end{proof}
\bibliographystyle{amsalpha}
\bibliography{PeriodicHamFlows}
\end{document}